\documentclass[12pt]{article}

\usepackage{amssymb,amsmath,amstext,color}
\usepackage{graphicx,overpic,psfrag}
 
\usepackage{subfigure}
 
\usepackage[arrow, matrix, curve,cmtip]{xy}
 
\usepackage{MnSymbol}
  
\usepackage{tikz-cd}
\usepackage{algpseudocode}
\usepackage{algorithm,float}
\usepackage{mathrsfs}
\usepackage{hyperref}

\usepackage{adjustbox}
\usepackage{extpfeil}

\newtheorem{lemma}{Lemma}
\newtheorem{theorem}[lemma]{Theorem}
\newtheorem{corollary}[lemma]{Corollary}
\newtheorem{definition}[lemma]{Definition}

\newif\ifqedused
\newenvironment{proof}{\qedusedfalse\noindent{\em Proof:}}{\ifqedused\else\hfill $\Box$~\\\fi}
\newcommand{\qedhere}{\global\qedusedtrue\tag*{$\Box$}}

\usepackage{etoolbox}
\AtBeginEnvironment{proof}{\setlength{\parskip}{0pt}}

\usepackage{amsmath, amssymb}
\usepackage{tikz-cd}
\usepackage{graphicx}                   
\usetikzlibrary{decorations.pathreplacing, calc}

\newcommand{\upsubsetneq}{\rotatebox[origin=c]{45}{$\subsetneq$}}
\newcommand{\downsubsetneq}{\rotatebox[origin=c]{-45}{$\subsetneq$}}
\newcommand{\versubseteq}{\rotatebox[origin=c]{90}{$\subseteq$}}

\newtheorem{proposition}[lemma]{Proposition}
\newtheorem{remark}[lemma]{Remark}
\newtheorem{example}[lemma]{Example}

\newtheorem{notation}[lemma]{Notation}

\usepackage{enumitem}
\setlist{itemsep=0pt, topsep=2pt, parsep=1pt}

\def\be{\begin{eqnarray}}
\def\ee{\end{eqnarray}}
\def\bes{\begin{eqnarray*}}
\def\ees{\end{eqnarray*}}

\def\R{\mathbb R}       
\def\C{\mathbb C}      
\def\K{\mathbb K} 
\def\A{\mathbb A} 
\def\P{\mathbb P} 
\def\V{\mathbb V}  
\def\L{\mathbb L}  
\def\T{\mathbb T} 

\def\cP{\mathscr P}  
\def\cPU{\mathscr{P}_{\mathscr{U}}} 
\def\cPR{\mathscr{P}_{\mathscr{U},\mathbb{R}}} 
\def\cV{\mathscr{V}}
\def\cU{\mathscr{U}}
 \def\cT{\mathscr{V}_{\T}}
 
\def\F{\mathbb F}
\def\cA{\mathscr{A}_\cP}

\def\mT{\T_{\min}}

\def\cD{\mathscr{D}_\cP}

\newcommand{\Base}[1]{\mathscr{B}_{#1}}

\def\cR{\mathscr{R}}

\def\ot{\,\overline{t}\,}
\def\oT{\,\overline{T}\,}
\def\od{\,\overline{\delta}\,}
\def\oD{\,\overline{\Delta}\,}
\def\ox{\,\overline{x}\,}
\def\oX{\,\overline{X}\,}

\def\E{\mathcal{E}}

\usepackage{float}
\newfloat{diagram}{htbp}{lod}
\floatname{diagram}{Diagram}

\author{
Javier Sendra-Arranz\thanks{Department of Mathematics, CUNEF Universidad, Madrid, Spain.
Emails: \texttt{javier.sendraarranz@cunef.edu},
\texttt{jrafael.sendra@cunef.edu},
\texttt{juana.sendra@cunef.edu}}
\and
Juan Rafael Sendra\footnotemark[1]
\and
Juana Sendra\footnotemark[1]
}

\title{
On the Structure of Radical Towers and their Parametrizations}

\date{}  

\begin{document} 
\maketitle

\begin{abstract}
Radical parametrizations describe algebraic varieties by means of nested
radicals of polynomials, and their theory relies on three varieties
associated with a parametrization: the incidence, tower and
radical varieties. In this paper we revisit this framework from a
computational point of view and show that it is governed by the radical
tower rather than by the particular parametrization. 
We prove that the tower
variety depends only on the tower and that the incidence variety is
always irreducible, so that no prime decomposition is needed. We show how to replace a tower by an isomorphic one whose tower variety
is defined by the equations of the tower, so that it can be computed
without elimination. We introduce universal radical parametrizations
and we use them to settle, under a natural 
hypothesis on the coefficients, an open
case in the transformation of differential equations with radical
coefficients into equations with rational coefficients. We show that
radical subtowers form a lattice, which yields minimal radical
parametrizations, and we give algorithms to compute them. We also prove
that dominance between towers induces a unique dominant rational map
between their tower varieties, through which unirationality is
transferred.
 In addition, we study towers over subfields
of $\mathbb{C}$ and, in the real case, relate the real points of the
tower variety to the real domains of the associated parametrizations.
 Finally, we present computational experiments which show that our methods succeed in many cases
where the previous approach does not finish within a reasonable time.
\end{abstract}

\noindent\textbf{Keywords:} radical parametrization, radical tower,
tower variety, unirationality, minimal parametrization, real algebraic
geometry.

\noindent\textbf{AMS subject classification:} 14Q05, 14Q10, 14M20,
14P10, 68W30.

\section{Introduction}

Rational varieties, and more generally unirational varieties, can be represented
by means of parametrizations given by rational functions. Radical varieties can
be regarded as a generalization of unirational varieties, in the sense that they
admit parametrizations given by quotients of nested radicals of polynomials (see \cite{PirolaSchlesinger2005}  and \cite{Zariski1926}); in
particular, unirational (and hence rational) varieties are a special case of
radical varieties. From an algorithmic point of view,
algorithms for parametrizing algebraic curves of genus at most six by radicals
are presented in \cite{Harrison2013,SendraSevilla2011}, and
in \cite{SendraSevilla2013} it is shown how to parametrize certain families of
algebraic surfaces by radicals. In a second stage, the algebraic theoretical
framework for radical varieties was established in \cite{RSV2017}.   There, the main ingredients in
this algebraic perspective are three varieties associated to a radical tower $\T$ and a radical parametrization $\cP$: the incidence variety $\mathscr{A}_\cP$, the tower variety $\cV_\T$ and the radical variety $\cV_\cP$. 
The algebro-geometric study of these varieties plays a fundamental role in the theory and applications of radical parametrizations. 
The surjectivity of radical parametrizations of curves was analyzed in \cite{CaravantesSendraSevillaVillarino2024}.

The theory of radical parametrizations of algebraic varieties has found
applications in a wide range of contexts. In applied algebraic geometry, the
results of \cite{RSV2017} are extended in \cite{KKS2025} to radical
parametrizations restricted to irreducible varieties. On this basis, the authors
give a rigorous definition of the algebraic degree of optimization problems over
a variety, and in particular of the $p$-norm distance degree. In mathematical
physics, rationalizing the square roots that arise in the computation of Feynman
integrals leads to the study of the rationality of certain double covers. Besier
and Festi \cite{BesierFesti2021} point to radical parametrizations as the
natural framework when such a rationalization is not possible. The more general case of ordinary, and partial,  differential equations was treated in \cite{CSSV2021,FS2024}.

Parametrizations by radicals also arise in cryptography. Couveignes and
Kammerer \cite{CouveignesKammerer2012} construct parametrizations of elliptic
curves over finite fields by means of cubic radicals in order to obtain
deterministic hash functions into elliptic curves, and they present their
construction as a special case of the general problem of parametrizing curves
by radicals. In the study of the real varieties associated with the normalized
Ricci flow on generalized Wallach spaces, Batkhin \cite{Bartkhin} proposes
radical parametrization as a means of describing these varieties globally.

In computer algebra and computational geometry, Lewis \cite{Lewis2018} uses the
Dixon resultant to obtain implicit parametrizations of intersection curves of
surfaces, a notion closely related to the description of curves introduced
in \cite{RSV2017}. Agarwal et al. \cite{AgarwalEtAl2022} include radical
parametrizations of algebraic curves among the admissible representations in
their data structures for intersection queries on semi-algebraic objects in
$\mathbb{R}^3$. Finally, radical parametrizations of curves play a central role
both in the computation of radical solutions of first-order algebraic ordinary
differential equations \cite{CSSV2021,FS2024,Grasegger2014} 
and in the explicit determination of Zolotarev
polynomials \cite{GraseggerVo2017,RackVajda2021}. Regarding the latter problem,
Rack and Vajda \cite{RackVajda2021} obtain parametrizations in terms of nested
square roots for degrees $5 \le n \le 7$.

Several questions,  {with impact on applications and computations, were left open in the algebraic treatment of radical parametrizations carried out in \cite{RSV2017}}. Among
them the irreducibility of the incidence variety $\cA $ and the
computational role of the radical tower were opened. In this paper, we answer these
questions and, as a consequence, improve the algorithmic treatment of
radical varieties. Our approach rests on two ideas: the use of a canonical
representation of the elements of the tower, and the observation that the
theory is governed by the radical tower rather than by the particular
parametrization. Our main contributions are the following.

\noindent\emph{Refinement of the general framework.}
We prove that the tower variety of every radical parametrization  depends only on its tower (see Corollary \ref{co: no need of para}).  Moreover, given a tower $\T$ we show how to compute an isomorphic tower $\T'$ to $\T$ (see Algorithm \ref{alg: tower}), so that $\cV_\T$ and $\cV_{\T'}$ are birational and $\cV_{\T'}$ 
 can be obtained without any elimination (Corollary \ref{co: no denom}, Theorem \ref{theo-change of var}). We also
prove that the incidence variety $\cA $ is always irreducible
(Theorem \ref{thm: incidence var}), so no prime decomposition is required to identify the
component associated with the parametrization.

\noindent\emph{Universal radical parametrizations.}
We introduce the notion of universal radical parametrization of a tower and
show that its incidence, radical and tower varieties are isomorphic
(Theorem \ref{theo: univ para var}). As a consequence, we settle the case left open in \cite{CSSV2021}, and the same argument applies to \cite{FS2024}, when the coefficients form a radical sequence (Remark \ref{rem-edo}).

\noindent\emph{Minimal radical parametrizations.}
We show that radical subtowers form a lattice under union and intersection
(Proposition \ref{prop:inters of tower}), which yields a minimal radical subtower, and hence a
minimal radical parametrization, for every radical variety. We also give
algorithms that compute the minimal subtower directly from the canonical
representation. Working with minimal subtowers simplifies the computations,
since the tower variety keeps its dimension but lives in an affine space of
smaller dimension, and it also helps to detect theoretical properties such
as unirationality: the minimal subtower may have a rational tower variety
even when the original one does not (Remark \ref{rem:smaller tower}, Example \ref{ex: minimal tower better}).

\noindent\emph{Dominance between towers.}
 If the last field of a radical tower $\T^*$ is contained in that of another radical tower $\T$, there is a unique
dominant rational map $\iota\colon\cV_{\T}\dashrightarrow\cV_{\T^*}$ between the corresponding tower varieties,
independent of the parametrization and compatible with all of them
(Theorem \ref{theo:initial maps}). Consequently, unirationality passes from $\cV_{\T}$ to
$\cV_{\T^*}$ and to every radical variety defined over it. We also show  that a radical variety may be unirational but
not rational while its tower variety is rational (see Example \ref{ex:nonluroth}).

\noindent\emph{The real case.}
We extend the theory to towers defined over subfields of $\mathbb{C}$ and
study the real case in detail. Introducing the $\mathbb{R}$-associated
parametrizations, obtained by sign changes of the radicals of even index, we
prove that their real domains cover the real points of the tower variety,
outside the zeros of the denominators, with equality for minimal
parametrizations (Theorem \ref{thm:real domain}). This gives a semialgebraic cover of
the real points of the tower variety.

\noindent\emph{Computational experiments.}
We compare the approach of \cite{RSV2017} with our methods on twelve
towers with three to ten floors (Appendix \ref{app:towers}).
Algorithm \ref{alg: tower} gives, without elimination, the tower variety of
an isomorphic tower in all cases in less than three seconds, whereas the
approach of \cite{RSV2017} only finishes for the towers with three floors.
For the radical variety, working on the minimal radical subtower makes the
computation feasible for four of the five towers with ten floors, where the
other methods do not finish,  even though  none of the methods finish in five cases  (Tables \ref{tab:times}
and \ref{tab:summary}).

The experiments were designed by the authors; their implementation was carried out with the AI system Claude (Anthropic) under the supervision of the authors, who reviewed the methodology, the code and the results. The scripts and data needed to reproduce the experiments are available on
Zenodo \cite{zenodo}.

  The structure of the paper is as follows.
Section \ref{sec:preliminares} recalls the basic notions on radical parametrizations and
Section \ref{sec:refinamient} refines the general framework. Section \ref{sec:universal-param} studies universal
parametrizations and Section \ref{sec:subtower} minimal ones, together with the
corresponding algorithms. Section \ref{sec:dominant} deals with dominance between towers, 
Section \ref{sec:subsfields} with subfields and the real case, and
Section \ref{sec:computation} presents the computational experiments.
Appendix \ref{app:towers} lists the radical towers and parametrizations
used in these experiments.

\section{Preliminaries on Radical Parametrizations}\label{sec:preliminares}
In this section, we recall the main concepts on radical parametrizations. For additional details, the interested reader is referred to \cite{RSV2017}, on which this section is essentially based. 

Let $\oT=(T_1,\ldots,T_n)$ be a tuple of independent parameters. First, 
we consider the   \textsf{radical tower of fields} $\T$  
\begin{equation}\label{tower}
\T:=\left[\C(\oT)=:\F_0\subsetneq \F_1=\F_0(\delta_1)\subsetneq\F_2=\F_1(\delta_2)\subsetneq \cdots \subsetneq \F_m=\F_{m-1}(\delta_m) \right]
\end{equation}
where 
\begin{equation}\label{eq-delta} 
\text{$\F_i=\F_{i-1}(\delta_i)$, $\delta_i^{e_i} = \alpha_i \in \F_{i-1}$, for some  $e_i \in \mathbb{Z}_{>1}$, }
\end{equation}
is a field extension of $\F_{i-1}$ of degree $[\F_i:\F_{i-1}]=e_i$. 
So 
\begin{equation}\label{eq-Etilde}
\widehat{E}_i:=\Delta_i^{e_i}-\alpha_i\in \F_{i-1}[\Delta_i]
\end{equation}
is irreducible over $\F_{i-1}$, where $\Delta_i$ is a new variable. 
That is
\begin{equation}\label{eq-Fi}
\F_i=\F_{i-1}[\Delta_i]/\langle \widehat{E}_i\rangle.
\end{equation}
We will use the notation $\overline{\delta}=(\delta_1,\ldots,\delta_m)$ and $\overline{\Delta}=(\Delta_1,\ldots,\Delta_m)$. 
The distinction between $\delta_i$ and $\Delta_i$ lies in the nature of both objects: $\delta_i$ is an analytic object and $\Delta_i$ is purely algebraic. More concretely, $\delta_i$ represents one of the complex radical functions
\begin{equation}\label{eq-raices-unidad}
\xi_i^k\cdot \sqrt[e_i]{\alpha_i(t)}, \,\,\,\text{where} \,\,\,\xi_i=\exp^{\frac{2\pi \sqrt{-1}}{e_i}}
\end{equation} and $0\leq k\leq e_i-1$. 
On the other hand, $\Delta_i$ is an algebraic variable used in the polynomial ring $\F_{i-1}[\Delta_i]$ to build algebraically the radical tower. Throughout this section, we discuss in more detail the relation between $\delta_i$ and $\Delta_i$.

In the sequel we will express the polynomial $\widehat{E}_i$  as
\begin{equation}\label{eq-hE} 
\widehat{E}_i=\Delta_i^{e_i}-\dfrac{\alpha_{i,\mathrm{N}}(\overline{T},\overline{\Delta})}{\alpha_{i,\mathrm{D}}(\overline{T},\overline{\Delta})}\in\F_{i-1}[\Delta_i] \,\text{ for }\, 1\leq i\leq m.
\end{equation}

A  \textsf{radical parametrization} is  an element   
\begin{equation}\label{eq-p}
\cP(\oT,\oD):=(p_1(\oT,\oD),\ldots,p_{r}(\oT,\oD))\in \F_{m}^{r},\end{equation}
with $r>n$,  whose Jacobian matrix $\mathcal{J}_{\oT}(\cP)$, with respect to $\oT$, has full rank. Note that if $m=0$, then we get a unirational parametrization. As above, we will express each $p_i$ as
\begin{equation}\label{eq-p_i}
    p_i(\oT,\oD)=\dfrac{p_{i,\mathrm{N}}(\oT,\oD)}{p_{i,\mathrm{D}}(\oT,\oD)} .
 \end{equation}
Then,    
a  \textsf{radical parametrization} is a triple
\begin{equation}\label{def-rad-param} \mathcal{P}:=\left(\mathbb{T},\cP=(p_1,\ldots,p_{r}),(\overline{t}_0,\overline{\delta}(\overline{t}_0))\right)
\end{equation}
where $\mathbb{T}$ is the radical tower \eqref{tower}, $\cP$ is as in \eqref{eq-p}, and $\overline{t}_0 \in \C^n$ is such that $\cP(\overline{t}_0,\overline{\delta}(\overline{t}_0))$ is well defined. The evaluation $\overline\delta(\ot_0)$ is indeed used to choose recursively one of the $e_i$ solutions to the equation $\widehat{E}_i=0$ corresponding to the $e_i$ roots in $\C$.
In other words, $\overline{\delta}(\overline{t}_0)$ fixes the complex radical function in \eqref{eq-raices-unidad} that each $\delta_i$ represents.

Let  $\emptyset \neq \mathscr{F}\subset    \C[x_1,\ldots,x_r]$.
We say that $\cP$ is a \textsf{radical parametrization} of $\mathbb V_\C(\mathscr F)$ if $F(p_1,\dots,p_r)=0$ for every $F\in\mathscr  F$ and the radical map $t\mapsto\cP(\ot,\overline{\delta}(\ot))$, defined on a nonempty  Euclidean open  subset $\Omega\subset\C^n$, is dominant onto $\mathbb V_\C(\mathscr{F})$. In this case we say that   $\V_\C(\mathscr{F})$ is a  \textsf{radical variety.} 
 
Let $\oX=(X_1,\ldots,X_r),$ $\oT=(T_1,\ldots,T_n),\oD=(\Delta_1,\ldots,\Delta_m)$,   where $X_i,T_i,\Delta_j$ are variables, with $r>n$. We introduce the polynomials (see \eqref{eq-hE}, \eqref{eq-p_i})
\begin{equation}\label{eq-pols}
\left.
\begin{array}{llll} 
   E_i(\oT,\oD) &:= & \alpha_{i,\mathrm{D}}(\oT,\oD)\Delta_i^{e_i} - \alpha_{i,\mathrm{N}}(\oT,\oD), &i\in \{1,\ldots,m\}, \\
   G_i(X_i,\oT,\oD) &:=& p_{i,\mathrm{D}}(\oT,\oD)X_i -p _{i,\mathrm{N}}(\oT,\oD), &i\in \{1,\ldots,r\},\\
  G_W(W,\oT,\oD) &:=& W\cdot A(\oT,\overline{\Delta})-1,& 
    \end{array} \right\}
\end{equation}
where $A=\mathrm{lcm}(\alpha_{1,\mathrm{D}},\ldots, \alpha_{m,\mathrm{D}},
  p_{1,\mathrm{D}},\ldots, p_{r,\mathrm{D}})$. 
We consider the  variety
\[ \cD=\mathbb{V}_{\C}(  E_1,\ldots,E_m,G_1,\ldots,G_r, G_W  ) \subset \C^{n+m+r+1}.  \]
Now, we introduce the projection (see Diagram \ref{fig:diagram1})
\begin{equation}\label{eq-pitilde}
\tilde{\pi}:  \cD  \longrightarrow  \C^{n+m+r}:  (\ot, \overline{\delta}, \overline{x},w)  \longmapsto   (\ot, \overline{\delta}, \overline{x}).
\end{equation}
 The variety $\cA$, that we call the  \textsf{incidence variety}, is defined as the Zariski closure of the projection of $\cD$ via $\tilde{\pi}$
\begin{equation}\label{eq-cA} \cA:=\overline{\tilde{\pi}(\cD)}.
\end{equation}
 In \cite{RSV2017}, the irreducibility of $\cA$ is not established. Instead, the authors introduce $\mathcal{B}$ as the irreducible component of $\cA$ associated with the given parametrization; for this purpose,  $(\overline{t}_0,\overline{\delta}(\overline{t}_0))$ is used. However, in Section \ref{sec:refinamient}, Theorem \ref{thm: incidence var}, we prove that $\cA$ is, in fact, irreducible. Consequently, $\cA=\mathcal{B}$, which will simplify Diagram  \ref{fig:diagram1} into Diagram \ref{fig:diagram2}. Nevertheless, in order to remain faithful to the exposition in \cite{RSV2017}, we present  this section following the original paper. For this purpose, we consider the map 
\begin{equation}\label{eq-varphi} \varphi: \C^n \rightarrow \cA; \,\,\ot \mapsto (\ot,  \overline{\delta}(\ot), \overline{x}(\ot,\overline{\delta}(\ot)))
\end{equation}
and we define  
 $ \mathcal{B}=\overline{\mathrm{im}(\varphi)}.$ 
We next introduce two new projections 
\begin{equation}\label{eq-pi}
\begin{array}{ll}
 \pi:  \mathcal{B} \longrightarrow  \C^{r}:  (\ot, \overline{\delta}, \overline{x})   \longmapsto   \overline{x}, \,\,\,\,\,
 \pi^*:  \mathcal{B} \longrightarrow  \C^{n+m}:  (\ot, \overline{\delta}, \overline{x} )  \longmapsto  (\ot, \overline{\delta}). 
 \end{array}
\end{equation}
Associated with these maps, we introduce the varieties $\cV_\mathscr{P}:=\overline{\pi(\mathcal{B})}$  and $\cT:=\overline{\pi^*(\mathcal{B})}$. 
We call $\cV_\mathscr{P}$ the  \textsf{radical variety} and $\cT$ the  \textsf{tower variety}.

The projection $\pi^*$ is birational (see \cite{RSV2017}), hence $\mathcal{B}$ and $\cT$ are birationally equivalent.  In addition, we can introduce a rational map relating  $\cT$  and $\cV_\mathscr{P}$. More precisely 
\[ \cR:=\pi \circ (\pi^{*})^{-1}: \cT \longrightarrow \cV_\cP. \]
Moreover, considering the maps 
\[ \begin{array}{ll} \psi: \C^n \longrightarrow \cT;\,\, \ot \longmapsto (\ot,\overline{\delta}), \,\,\,\,\,\,\, \cP: \Omega\subset \C^n \longrightarrow \cV_\mathscr{P} \subset \C^r; \ot \longmapsto \cP(\ot),
\end{array} \]
we get the commutative Diagram   \ref{fig:diagram1}. 
In \eqref{eq-p}, a radical parametrization $\cP$ is represented as an element in $\F_m^r$, and hence, as a tuple of polynomials in $\F_0[\overline{\Delta}]$ modulo the ideal $\E$. To represent $\cP$ as an analytic function $\cP:\C^n\dashrightarrow\C^r$ we need to identify inductively each $\Delta_i$ with $\delta_i$, which is one of the $e_i$ complex radical functions in \eqref{eq-raices-unidad}. 
Such a choice of complex radical can be done through the evaluation $\overline\delta(\overline{t}_0)$. This selection is also equivalent to fixing the map $\psi$ in Diagram \ref{fig:diagram1}. In Section \ref{sec:universal-param}, we study the map $\psi$ as a map intrinsic of the tower and not on the parametrization.

\begin{diagram}[ht]
\centering
 
\begin{tikzcd}
\cD\subset \mathbb{C}^{n+m+r+1} \arrow[d, "\begin{array}{l} \tilde{\pi} \\  \end{array}"]                                                &  &  &  \\
\overline{\mathrm{im}(\varphi)}=\mathcal{B}\subset \cA \subset \mathbb{C}^{n+m+r} \arrow[dd, "{\begin{array}{c} {\pi}   \end{array}}"'] \arrow[rrd, "\begin{array}{c} {\pi^*} \,\, (1:1)  \end{array}"] &  & & \\      &  & \cT\subset\mathbb{C}^{n+m} \arrow[lld, "\cR"'] & 
\\
\cV_\mathscr{P}\subset \mathbb{C}^{r}   &  & \mathbb{C}^n \arrow[ll, "\cP"'] \arrow[u, "\psi"'] \arrow[lluu, "\varphi"'] &                                       \end{tikzcd}

\caption{Commutative diagram of the radical parametrization}
\label{fig:diagram1}
\end{diagram} 

\begin{remark}\label{Remark-first}
We summarize now some of the most relevant properties of this construction.
\begin{enumerate}
\item $\cV_\cP$ and $\cT$ are irreducible and $\dim(\mathcal{B})=\dim(\cV_\mathscr{P})=\dim(\cT)=n,$ see  \cite[Theorem 3.11, Theorem 4.6]{RSV2017}.
\item\label{Rem2} If $\cT$ is unirational, parametrized by $\mathscr{T}:\C^n\rightarrow \cT$, then $\cR\circ\mathscr{T}$ is a unirational parametrization of $\cV_\mathscr{P}$, see \cite[Theorem 4.11]{RSV2017}.  However, in general, $\cV_\cP$ and $\cV_\T$ are not birationally equivalent; see Examples   \ref{ex: minimal tower better} and \ref{ex-minimal-strong}. An interesting situation is also described in Example \ref{ex:nonluroth} where the radical variety is unirational, non-rational, and the tower variety is rational.
\item Using Gr\"obner bases, one can compute the generators of $\cV_\cP$ and $\cT$, see e.g. \cite[Closure Theorem]{CLO2015}.
\end{enumerate}
\end{remark}

\section{Refinements of Radical Parametrizations}\label{sec:refinamient}

In this section we present new results regarding the incidence and tower varieties introduced in the previous section, which simplify   Diagram  \ref{fig:diagram1}. For this purpose, a key point in our reasoning is to use a suitable representation of the elements in $\F_m$. 

Let $\widehat{E}_i$ be as in \eqref{eq-hE}. By construction, we can write the field $\F_1$ as 
    \[
    \F_1=\F_0[\Delta_1]/\langle \widehat{E}_1\rangle = \F_0[\Delta_1]/\langle E_1\rangle
    \]
    by multiplying $\widehat{E}_1$  by $\alpha_{1,D}(\overline{T})$, which is a unit of $\F_0$. Note that since $\F_0[\Delta_1]$ is a principal ideal domain and $\widehat{E}_1$ and $E_1$ are irreducible by construction, the above quotient is a field.
    Repeating this process inductively,  we get that $\widehat{E}_1,\ldots,\widehat{E}_m$ and  
    $E_1,\ldots,E_m $ can be expressed as polynomials in $\F_0[\oD]$ as 
    \begin{equation}\label{eq-ET}
    \begin{array}{l}
    \displaystyle
    \widehat{E}_i(\oT, \oD)=  \Delta_{i}^{e_i}-  \frac{\alpha_{i,\mathrm{N}}(\oT,\oD)}{\alpha_{i,\mathrm{D}}(\oT)}\, \,\,\,\,\text{and}\,\,\,\,\,\displaystyle
    E_i(\oT, \oD)= \alpha_{i,\mathrm{D}}(\oT) \Delta_{i}^{e_i}-  \alpha_{i,\mathrm{N}}(\oT,\oD).
    \end{array}
    \end{equation}
   In particular, we get that 
    \begin{equation}\label{eq:field desc}
\F_m=\F_0[\overline{\Delta}]/\langle\widehat{E}_1,\ldots,\widehat{E}_m\rangle=\F_0[\overline{\Delta}]/\langle {E}_1,\ldots, {E}_m\rangle.
    \end{equation} 
    Similarly, the polynomials $G_i$ in \eqref{eq-pols} can be expressed as 
  \begin{equation}\label{eq-ETGi}
    G_i(X_i,\oT, \oD)= p_{i,\mathrm{D}}(\oT) X_i-  p_{i,\mathrm{N}}(\oT,\oD)
    \end{equation}  

    \begin{remark}\label{remark-bases1}
   We observe that  $\F_j$ can be seen as a $\F_0$-vector space of dimension $e_1\cdots e_j$. A basis of this vector space is 
\begin{equation}\label{eq:basis Fm}
\Base{\T,\F_j}:=\{\Delta_1^{u_1}\cdots\Delta_j^{u_j}: \text{ where } 0\leq u_i\leq e_i-1 \text{ for }i=1,\ldots,j\}.
\end{equation}
Computationally, the coordinates w.r.t. $\Base{\T,\F_j}$ can be determined 
using Gr\"obner bases. 
\end{remark}

 \begin{definition}\label{def-canonical}
We say that $\beta\in \F_j$ is represented in \textsf{canonical form} if it is expressed as a vector of $\F_j$ w.r.t. the basis $\Base{\T,\F_j}$.
\end{definition}
In the sequel, \textsf{we assume} that $\alpha_i$ (see \eqref{eq-delta}) and $p_i$ (see \eqref{eq-p}) are expressed canonically. Therefore, $\alpha_{i,\mathrm{D}}$ and $p_{i,\mathrm{D}}$ depend only on $\oT$.

Now, let us consider the polynomial ideals 
\begin{equation}\label{eq-E}
\E= \langle E_1,\ldots,E_m \rangle\subset \C[\oT,\oD] \ \text{ and } \ \widehat{\E}= \langle \widehat{E}_1,\ldots,\widehat{E}_m \rangle \subset \C(\oT)[\oD] 
\end{equation}
where $\widehat{E}_i$ and $E_i$ are introduced in \eqref{eq-ET}. 
We denote by $U_\cP$ the Zariski-open subset on which the denominators of the construction do not vanish, and $E_j,G_j$ are as in \eqref{eq-pols}.
Similarly, let $U_\T$ be the Zariski-open subset on which the denominators $\alpha_{1,\mathrm{D}},\ldots,\alpha_{m,\mathrm{D}}$ in \eqref{eq-ET} do not vanish. The opens $U_\T$ and $U_\cP$ are defined in $\C^n$, and hence, they can also be seen as opens in $\C^{n+m}$ and $\C^{n+m+r}$.
Additionally, we consider the varieties
\[
\mathscr{Y}_\cP:=\mathbb{V}(E_1,\ldots,E_m,G_1,\ldots,G_r)\subset\C^{n+m+r} \ \text{ and } \ \mathscr{Y}_\T:=\mathbb{V}(E_1,\ldots,E_m)\subset\C^{n+m}.
\] 
Note that we have that 
$
\cA =\overline{\mathscr{Y}_\cP\cap U_\cP}
$.
We start with the following lemma.

\begin{lemma}\label{lemma- loc injection}
   Let  $\alpha_{\mathrm{D}}(\oT)=\alpha_{1,\mathrm{D}}(\oT)\cdots \alpha_{m,\mathrm{D}}(\oT)$. Then the localization map 
   \[
   \C[\oT,\oD]_{\alpha_{\mathrm{D}}(\oT)}/\widehat{\E}\longrightarrow \F_0[\oD]/\widehat{\E}=\F_m
   \]
   is injective. Here $\C[\oT,\oD]_{\alpha_{\mathrm{D}}(\oT)}$ is the localization by $\alpha_{\mathrm{D}}(\oT)$.
\end{lemma}
\begin{proof}
     For $0\leq i\leq m$, we consider the ring
    \[
    R_i:=R_{i-1}[\Delta_i]/\langle \widehat{E}_i \rangle , \ \text{ with } \ R_0=\C[\oT]_{\alpha_{\mathrm{D}}(\oT)}.
    \]
    We prove that  by induction on $0\leq i\leq m$, the localization map $ R_i\rightarrow \F_i$ is injective. For $i=0$, it holds as $\C[\oT]$ is included in $R_0$. Now, assume that $R_{i-1}\hookrightarrow\F_{i-1}$. 
    Let $g(\Delta_i)$ be in the kernel of the localization map  $R_i\rightarrow\F_i$. In other words, $g(\Delta_i)\in R_{i-1}[\Delta_i]$ belongs to $\langle \widehat{E}_i\rangle$ after moving to $\F_{i-1}[\Delta_i]$. Since $\widehat{E}_i$ is monic in $R_{i-1}[\Delta_i]$, division by $\widehat{E}_i$ can be  carried out in $R_{i-1}[\Delta_i]$.
    Dividing $g$ by $ \widehat{E}_i$ we get that 
    \[
    g=q\cdot \widehat{E}_i+r \ \text{ for }q,r\in R_{i-1}[\Delta_i] \ \text{ with } \deg( r)<e_i.
    \]
    However, $g$ belong to $\langle  \widehat{E}_i\rangle $ in $\F_{i-1}[\Delta_i]$. By uniqueness of the remainder,  $r=0$  in $\F_{i-1}[\Delta_i]$. By induction, $r=0$ in $R_{i-1}[\Delta_i]$. 
    So, $g\in\langle  \widehat{E}_i\rangle$ in $R_{i-1}[\Delta_i]$, and hence $g=0$ in $R_i$. We conclude that  the localization map is injective.
\end{proof}

\begin{proposition}\label{prop- new E}
    $\cV_\T = \overline{\mathscr{Y}_\T\cap U_\T}$ and $\cV_\T = \overline{\mathscr{Y}_\T\cap U_\cP}$.
    
\end{proposition}
\begin{proof}
    We observe that $\pi^{*}(\cA)$ is contained in $\mathscr{Y}_\T$. Since $\mathcal{B} $ is an irreducible component of $\cA$, we have that $\cV_\T$ is contained in $\mathscr{Y}_\T$. Therefore, it is enough to show that $\overline{\mathscr{Y}_\T\cap U_\T}$ is irreducible of dimension $n$. The coordinate ring of $\mathscr{Y}_\T\cap U_\T$ is the localization
    \begin{equation}\label{eq-local}
   \left( \C[\oT,\oD]/\E\right)_{\alpha_{\mathrm{D}}(\oT)},
    \end{equation}
    where $\alpha_{\mathrm{D}}(\oT)=\alpha_{1,\mathrm{D}}(\oT)\cdots \alpha_{m,\mathrm{D}}(\oT)$. In particular, we have
    \[
    \left( \C[\oT,\oD]/\E\right)_{\alpha_{\mathrm{D}}(\oT)}= \C[\oT,\oD]_{\alpha_{\mathrm{D}}(\oT)}/\E=\C[\oT,\oD]_{\alpha_{\mathrm{D}}(\oT)}/\widehat{\E}.
    \]
    By Lemma \ref{lemma- loc injection}, we have the following inclusion
    \[
    \C[\oT,\oD]_{\alpha_{\mathrm{D}}(\oT)}/\widehat{\E}\hookrightarrow \F_0[\oD]/\widehat{\E}=\F_m,
    \]
   and  we deduce that the coordinate ring of $\mathscr{Y}_\T\cap U_\T$ is an integral domain whose field of fraction is $\F_m$.
    Additionally, the transcendence degree of $\F_m$ over $\C$ is $n$, the number of parameters $T_i$. Thus, $\dim(\mathscr{Y}_\T\cap U_\T)=n$.
    We conclude that $\overline{\mathscr{Y}_\T\cap U_\T}$ is irreducible of dimension $n$, and hence, $\overline{\mathscr{Y}_\T\cap U_\T}=\cV_\T$.

    Now we check that $\cV_\T = \overline{\mathscr{Y}_\T\cap U_\cP}$. We have shown that $\cV_\T\cap U_\T = \mathscr{Y}_\T\cap U_\T$. Using $U_\cP\subseteq U_\T$, it is enough to check that $\cV_\T\cap U_\cP$ is dense in $\cV_\T$. On the contrary, assume that  $\cV_\T\cap U_\cP$ is not dense in $\cV_\T$. This means that $\cV_\T$ is contained in a hypersurface of the form $\V(h(\oT))$ for $h(\oT)\in\C[\oT]$. This contradicts the fact that the fraction field of $\cV_\T$ contains the subfield $\F_0$.
\end{proof}

\noindent From the proof of  Proposition \ref{prop- new E}  we get the already known result (see \cite{RSV2017} and Section \ref{sec:preliminares})

\begin{corollary}\label{cor-VT}
$\cT$ is irreducible of dimension $n$.
\end{corollary}

From  Proposition \ref{prop- new E} we derive the following result.

\begin{theorem}\label{thm: incidence var} For every radical parametrization $(\T,\cP)$,
  $\cA$ is irreducible.
\end{theorem}
\begin{proof}
 Taking into account \eqref{eq-pitilde} and \eqref{eq-cA}, one has that 
  $$\cA\cap U_\cP= \mathscr{Y}_\cP\cap U_\cP  \ \text{ and } \
\cA= \overline{\mathscr{Y}_\cP\cap U_\cP}.
$$ 
  Hence, $\cA\cap U_\cP$ is a nonempty Zariski open subset of $\cA$. Using the projection $\pi^{\star}: \cA\rightarrow \mathbb{C}^{n+m}$ (see \eqref{eq-pi}), it holds that 
 $$\overline{\mathrm{Im}(\pi^{\star}(\cA)})\subseteq \mathscr{Y}_\T.$$
 In particular, we get a map $\pi^{*}:\cA\longrightarrow\mathscr{Y}_\T$.
 Since the polynomials involved in the definition of the open subset $U_\cP$ depend only on $\oT$, we can restrict $\pi^{*}$ also to $U_\cP$ and get the projection
 \begin{equation}\label{eq:proj birat}
 \pi^{*}:\cA\cap U_\cP=\mathscr{Y}_\cP\cap U_\cP\longrightarrow\mathscr{Y}_\T\cap U_\cP.
 \end{equation}
The map \eqref{eq:proj birat} is an isomorphism and its inverse is given by the map
\[
\begin{array}{ccc}
\displaystyle\left(\overline{t},\overline{\delta}\right)&\longmapsto& \displaystyle\left(\overline{t},\overline{\delta},\frac{p_{1,\mathrm{N}}(\overline{t},\overline{\delta})}{p_{1,\mathrm{D}}(\overline{t},\overline{\delta})},\ldots, \frac{p_{r,\mathrm{N}}(\overline{t},\overline{\delta})}{p_{r,\mathrm{D}}(\overline{t},\overline{\delta})}\right)
\end{array}
\]
The proof follows from Proposition \ref{prop- new E}, the irreducibility of $\cV_\T$ and that $\overline{\cA\cap U_\cP}=\cA$.
\end{proof}

From Theorem \ref{thm: incidence var} it follows that $\cD$ is  also irreducible since the map $\tilde{\pi}$  is birational and its inverse is defined over $\cA\cap U_\cP$.
In particular, it follows that all varieties in Diagram \ref{fig:diagram1} are irreducible. However, if we do not consider $G_W$, the variety $\mathscr{Y}_\cP$ might be reducible as illustrated in the following example.

\begin{example}
Consider the radical tower
\[
\T:=\left[
\F_0=\C(T)\subsetneq\F_1=\F_0\left(\sqrt{\frac{T}{T-2}}\right)\subsetneq \F_2=\F_1\left( \sqrt[3]{T^3+T^2+1}\right)
\right],
\]
and the radical parametrization $\cP=\left(\frac{\Delta_1}{T},T\right)$. Here, 
\[
\Delta_1^2=\frac{T}{T-2} \ \text{ and } \ \Delta_2^3=T^3+T^2+1.
\]
Then,  $\mathscr{Y}_\cP=\V(E_1,E_2,G_1,G_2)$  has four irreducible components: $\cA$ and three lines.  
\end{example}

  Part of the definition of radical parametrizations given in the literature and presented in  Section \ref{sec:preliminares} is the evaluation $\overline\delta(\overline{t}_0)$ chooses one of the $e_i$ solutions to the equation $\widehat{E}_i=0$ corresponding to the $e_i$ complex radicals; see \eqref{def-rad-param}. In particular, $\overline\delta(\overline{t}_0)$ identifies $\delta_i$ with $\Delta_i$.
  Additionally, $\overline\delta(\overline{t}_0)$ is used to choose $\mathscr{B}$ among the different components of $\cA$. By Theorem \ref{thm: incidence var}, $\cA$ is irreducible and this evaluation is no longer needed for this purpose. However, $\overline{\delta}(\overline{t}_0)$ could be used to specify which of the distinct complex radical functions in \eqref{eq-raices-unidad} we are using. This choice is equivalent to choosing the map $\psi$.

 \textsf{As a convention}, we will assume $\psi$ is the map given by the principal complex radical and the tower is constructed so that the desired branches of the radicals are achieved in $\F_m$; we identify $\delta_i$ with the complex radical function in \eqref{eq-raices-unidad} corresponding to the trivial root of unity.
In other words, we fix the map $\psi$ to be 
    \begin{equation}\label{eq:principal roots}
      \begin{array}{cccl}
        \psi:&\C^n&\rightarrow&\cV_\T\subseteq\C^{n+m} \\
         &\overline{t} & \mapsto&\left(\overline{t},\sqrt[e_1]{\alpha_1(\overline{t})},\ldots,\sqrt[e_m]{\alpha_m(\overline{t})}\right).
      \end{array}
      \end{equation}
      Thus, since the choice of the branches is already fixed by convention, $\overline\delta(\overline{t}_0)$ may be omitted from the construction. If one wants to work with other  branches of the radicals, one can either change the map $\psi$ or change the tower. Changing the tower will lead to an isomorphic tower (and isomorphic radical, incidence and tower variety). Such isomorphism is given by the change of variables $\Delta_i\mapsto \xi_i^{k}\cdot\Delta_i$ and some $0\leq k\leq e_i-1$ (see \eqref{eq-raices-unidad}). Example \ref{ex: distinct roots} illustrates these distinct methods of choosing the complex radicals.

  \begin{example}\label{ex: distinct roots}
      Consider the radical tower 
      \[
      \T:=\left[
        \F_0=\C(T)\subsetneq\F_1=\F_0[\Delta_1]/\langle \Delta_1^2-T\rangle
        \subsetneq\F_2=\F_1[\Delta_2]/\langle \Delta_2^2-T-\Delta_1\rangle
      \right].
      \]
      We identify $\Delta_1$ with $\delta_1=+\sqrt{T}$ and $\Delta_2$ with $\delta_2=+\sqrt{T+\sqrt{T}}$. We get the radical map 
      \[
      \begin{array}{cccl}
        \psi:&\C&\rightarrow&\cV_\T\subseteq\C^3 \\
         & t & \mapsto&(t,\sqrt{t},\sqrt{t+\sqrt{t}})
      \end{array}
      \]
      of Diagram \ref{fig:diagram1}. Now, a parametrization $\cP\in\F_2^3$ leads into the rational map $\cR:\cV_\T\dashrightarrow\cV_\cP$ of Diagram \ref{fig:diagram1}, and the radical parametrization $\cP:\C\dashrightarrow\cV_\cP$ is the composition $\cR\circ\psi$. The evaluation $\overline\delta(\overline{t}_0)$ represents the choice of radical map $\psi$. Nevertheless, such a choice can be seen from the construction of the tower. For instance, if we would have chosen the radical $-\sqrt{T}$ instead of $+\sqrt{T}$, the second radical would have been $\sqrt{T-\sqrt{T}}$. However, there is no element $g\in \F_2$ such that $g^2=T-\Delta_1$. Therefore, the radical map $t\mapsto\sqrt{t-\sqrt{t}}$ can not be seen as an element of $\F_m$. To do so, one can do two things. One can change the map $\psi$ by the map 
      \[
      \begin{array}{cccl}
        \psi':&\C&\rightarrow&\cV_\T\subseteq\C^3 \\
         & t & \mapsto&(t,-\sqrt{t},\sqrt{t-\sqrt{t}}).
      \end{array}
      \]
      A second option is to change the tower $\T$ by the tower
      \[
      \T':=\left[
        \F_0=\C(T)\subsetneq\F_1'=\F_0[\Delta_1']/\langle (\Delta_1')^2-T\rangle
        \subsetneq\F_2'=\F_1'[\Delta_2']/\langle (\Delta_2')^2-T+\Delta_1'\rangle
      \right]
      \]
      and identify directly $\Delta_1'$ with $\sqrt{T}$ and $\Delta_2'$ with $\sqrt{T-\sqrt{T}}$. Note that the towers $\T$ and $\T'$ are isomorphic via the change of variable $\Delta_1\rightarrow-\Delta_1'$. Therefore, the constructions of the radical and tower variety using $\T$ or $\T'$ are the same up to this scalar change of variable.
  \end{example}

 Summarizing, Diagram \ref{fig:diagram1} can be simplified as Diagram \ref{fig:diagram2}, which will be the diagram used in the rest of the paper. Observe that we have introduced a new map, namely
\begin{equation}\label{eq-gorro} \widehat{\pi}: \cV_\T\rightarrow \C^n; \,\,(\ot,  \overline{\delta})  \mapsto \ot
\end{equation}
Note that $\psi$ is injective and $\widehat{\pi}$ is generically $\ell:1$ where $\ell=e_1\cdots e_m$. Moreover 
 $\widehat{\pi}\circ \psi=\mathrm{id}$ but  $\psi \circ \widehat{\pi}\neq \mathrm{id}$ although $\psi(\ot_0)$ is in the fiber $\widehat{\pi}^{-1}(\ot_0)$.

\begin{diagram}[H]
\centering

\begin{tikzcd}
\cD\subset \mathbb{C}^{n+m+r+1} \arrow[d, "\begin{array}{l} \tilde{\pi} \\  \end{array}"]                                                &  &  &  \\
  \cA \subset \mathbb{C}^{n+m+r} \arrow[dd, "{\begin{array}{c} {\pi}   \end{array}}"'] \arrow[rrd, "\begin{array}{c} {\pi^*}\,\, (1:1)  \end{array}"] &  & & \\      &  & \cT\subset\mathbb{C}^{n+m} \arrow[lld, "\cR"'] \arrow[d, "\widehat{\pi}"{pos=0.5}', bend left=40] &
  \\
\cV_\mathscr{P}\subset \mathbb{C}^{r}   &  & \mathbb{C}^n\arrow[ll, "\cP"'] \arrow[u, "\psi"{pos=0.43}', bend left=40]  &                                       \end{tikzcd}

\caption{Improved commutative diagram of the radical parametrization}
\label{fig:diagram2}
\end{diagram}

We finish this section with some results on the tower variety and its equations.
Theorem \ref{thm: incidence var} shows that $\cA$ is irreducible, and hence, $\cV_\T=\overline{\pi^{*}(\cA)}$. However, this construction of $\cV_\T$ still depends on the radical parametrization $\cP$. The relevance of
Proposition \ref{prop- new E} 
is that it gives a construction of the tower variety that does not depend on the choice of $\cP$, only on the radical tower. In particular, we obtained the following result.

\begin{corollary}\label{co: no need of para}
For any two radical parametrizations $\cP\in \F_m^r$ and $\cP'\in\F_m^{r'}$, their tower varieties are equal.
\end{corollary}

Additionally, Proposition \ref{prop- new E} provides a method to compute the ideal of the tower variety without the need of $\cA$. The ideal of the tower variety is obtained through the saturation 
\begin{equation}\label{eq:saturation ideal}
\E:\langle \alpha_{1,\mathrm{D}}(\oT)\cdots \alpha_{m,\mathrm{D}}(\oT)\rangle^{\infty}.
\end{equation}
This saturation is not needed for generic choice of $\alpha_i$ and $\E$ will be the ideal of the tower variety. However, in general, there is not a generic formula for the equations of this ideal since this saturation process is necessary in some cases as illustrated in the next example

\begin{example}\label{ex: sat need}
Consider the radical tower
      \[
      \T:=\left[
        \F_0=\C(T)\subsetneq\F_1=\F_0[\Delta_1]/\langle \Delta_1^3-T^5\rangle
        \subsetneq\F_2=\F_1[\Delta_2]/\langle T\Delta_2^2-\Delta_1^2\rangle
      \right].
      \]
      Here the ideal $\E$ is not prime and $\mathscr{Y}_\T$ has two irreducible components. The ideal of $\cV_\T$ is given by 
      \[
      \langle \Delta_1^3-T^5, T\Delta_2^2-\Delta_1^2\rangle :\langle T \rangle^{\infty}=
      \langle T\Delta_2^2-\Delta_1^2 , T^3\Delta_1- \Delta_2^4,T^4-\Delta_1\Delta_2^2\rangle.
      \]
      We deduce that $\cV_\T\subsetneq\mathscr{Y}_\T$ and the saturation process is required for computing the ideal of the tower variety. 
\end{example}

Example \ref{ex: sat need} does not only provide an instance in which $\E$ is not the ideal of the tower variety, but it also shows that the number of generators after this saturation might change. This suggests the difficulty of finding general formulas for the equations of $\cV_\T$. 
Nevertheless, the following proposition shows a situation in which we avoid this challenge.

\begin{corollary}\label{co: no denom}
    Let $\T$ be a radical tower  with $\alpha_{i,\mathrm{D}}(\oT)=1$ for every $1\leq i\leq m$. Then, $\E$ is prime and $\cV_\T=\V(\E)$.
\end{corollary}
\begin{proof}
By Proposition \ref{prop- new E}, the ideal of the tower variety is the saturation \eqref{eq:saturation ideal}. Since $\alpha_{i,\mathrm{D}}(\oT)=1$ for every $1\leq i\leq m$, the ideal of the tower variety is 
$ 
\E:\langle 1\rangle^{\infty} = \E.
$
We conclude that $\cV_\T=\V(\E)$ and since $\cV_T$ is irreducible, $\E$ is prime.
\end{proof}

Corollary \ref{co: no denom} shows a family of radical towers for which the ideal of the tower variety is $\E$. This provides a computational improvement as one can avoid the saturation process to obtain this ideal. As shown in Example \ref{ex: sat need}, in general, this saturation can not be avoided. We propose the following algorithm to slightly change the tower to 
a new isomorphic tower satisfying the assumptions of Corollary \ref{co: no denom}
and get a tower variety with simple ideal. The strategy is to build a new radical tower $\T'$ of the form 
\begin{equation}\label{eq:new tower}
\T':=\left[
\F_0\subsetneq\F_1':=\F_0(\gamma_1)\subsetneq\cdots \subsetneq\F_m':=\F_{m-1}(\gamma_m)
\right]
\end{equation}
with 
\[
\F_i'=\F_{i-1}'(\gamma_i)=\F_{i-1}'[\Gamma_i]/\langle E_i'\rangle \ \text{ and }\ E_i'=\Gamma_i^{e_i}-\alpha_i'(\oT,\Gamma_1,\ldots,\Gamma_{i-1})
\]
such that $\F_i'=\F_i$ for $1\leq i\leq m$.

\begin{algorithm}[H]
\floatname{algorithm}{Algorithm}
    \caption{Given a radical tower $\T$, it computes an isomorphic  tower $\T'$ as in \eqref{eq:new tower}.}\label{alg: tower}
    
    \begin{itemize}[leftmargin=1.2em, itemsep=2pt]
        \item \textbf{Input:} A radical tower $\T$.
        \item \textbf{Output:}  A radical tower $\T'$ as in \eqref{eq:new tower}  isomorphic to $\T$.
        \begin{enumerate}[leftmargin=1em, itemsep=2pt]
            \item  Set $\tilde{\alpha}_{1,\mathrm{D}}(\oT)=\alpha_{1,\mathrm{D}}(\oT)$ and  fix $\gamma_1=\tilde{\alpha}_{1,\mathrm{D}}(\overline{t})\delta_1$ and $\Gamma_1=\tilde{\alpha}_{1,\mathrm{D}}(\oT)\Delta_1$. Construct the first floor of $\T'$ as
            \[
            \F_1':=\F_0(\gamma_1)=\F_0[\Gamma_1]/\langle E_1'\rangle \ \text{ with } \ E_1'=\Gamma_1^{e_1}-\alpha_{1,\mathrm{D}}^{e_1-1}\alpha_{1,\mathrm{N}}.
            \]
            \item For $2\leq i\leq m$ do
            \begin{enumerate}
                \item  Apply the change of variables
                \[
                \Delta_{j}=\frac{\Gamma_j}{\tilde{\alpha}_{j,\mathrm{D}}(\oT)} \ \text{ for }1\leq j\leq i-1
                \]
                to $\alpha_i(\oT,\Delta_1,\ldots,\Delta_{i-1})$. Express this new element of $\F_{i-1}'$ in canonical form. Let $\tilde{\alpha}_{i,\mathrm{N}}(\oT,\Gamma_1,\ldots,\Gamma_{i-1})$ and $\tilde{\alpha}_{i,\mathrm{D}}(\oT)$ be the numerator and denominator of this expression in the canonical form. 
                \item Set $\gamma_i =\tilde{\alpha}_{i,\mathrm{D}}(\overline{t}) \delta_i $ and $\Gamma_i =\tilde{\alpha}_{i,\mathrm{D}}(\oT) \Delta_i $.  Construct the i--th  floor of $\T'$ as
            \[
            \F_i':=\F_{i-1}'(\gamma_i)=\F_{i-1}'[\Gamma_i]/\langle E_i'\rangle \ \text{ with } \ E_i'=\Gamma_i^{e_i}-\tilde{\alpha}_{i,\mathrm{D}}^{e_i-1}\tilde{\alpha}_{i,\mathrm{N}}.
            \]
            \end{enumerate}
        \item Return the radical tower $\T'$ whose floors are $\F_0,\F_1',\ldots, \F_m'$.
        \end{enumerate}
    \end{itemize}
\end{algorithm} 

 The following result shows that Algorithm \ref{alg: tower} returns the desired output.

\begin{theorem}\label{theo-change of var}
    The output $\T'$ obtained from Algorithm \ref{alg: tower} is a radical tower that satisfies 
    \begin{enumerate}

        \item \label{theo-change 1 } $\F_i'=\F_i$ for every $1\leq i\leq m$.
        \item\label{theo-change 2 } The ideal of the tower variety $\cV_{\T'}$ of $\T'$ is   $\langle E_1',\ldots,E_m'\rangle$.
        \item\label{theo-change 3 } The map 
        \begin{equation}\label{eq:some birat map}
        \begin{array}{ccl}
        \cV_{\T}&\longrightarrow&\cV_{\T'}\\
        (\overline{t},\overline{\delta})&\longmapsto&(\overline{t},\tilde{\alpha}_{1,\mathrm{D}}(\overline{t})\delta_1,\ldots, \tilde{\alpha}_{m,\mathrm{D}}(\overline{t})\delta_m)
        \end{array}
        \end{equation}
        is birational.
    \end{enumerate}
\end{theorem}
\begin{proof}
First, we show that $\T'$ is a radical tower. To do so, we check that $\overline{\gamma}$ satisfies the equation $E_i'=0$ for every $i$. Indeed, we have that 
\[
E_i'(\oT,\gamma_1,\ldots,\gamma_{i})=\gamma_i^{e_i}-\tilde{\alpha}_{i,\mathrm{D}}^{e_i-1}\tilde{\alpha}_{i,\mathrm{N}}=
\left(\delta_i\tilde{\alpha}_{i,\mathrm{D}}\right)^{e_i}-\tilde{\alpha}_{i,\mathrm{D}}^{e_i-1}\tilde{\alpha}_{i,\mathrm{N}}
=
\tilde{\alpha}_{i,\mathrm{D}}^{e_i-1}\left(\tilde{\alpha}_{i,\mathrm{D}}\delta_i^{e_i}-\tilde{\alpha}_{i,\mathrm{N}}\right)
\]
Undoing the change of variables applied in Step 2 (a) of Algorithm \ref{alg: tower}, we get that 
\[
\tilde{\alpha}_{i,\mathrm{D}}(\oT)\delta_i^{e_i}-\tilde{\alpha}_{i,\mathrm{N}}(\oT,\gamma_1,\ldots,\gamma_{i-1})=E_i(\oT,\delta_1,\ldots,\delta_{i})=0.
\]
Next, we need to check that  $E_i'$ is irreducible in $\F_{i-1}'[\Gamma_i]$. To do so, note that the change of variables applied in Step 2 (a) of Algorithm \ref{alg: tower} leads to an isomorphism between
$ 
\F_{i-1}[\Delta_i] \ \text{ and } \F_{i-1}'[\Gamma_i].$ 
that sends $\widehat{E}_i$ to $E_i'$. So the irreducibility of $E_i'$ follows from the irreducibility of $\widehat{E}_i$. Therefore, $\T'$ is a radical tower.

We show 
Theorem \ref{theo-change of var} \eqref{theo-change 1 } by induction on $0\leq i\leq m$. For $i=0$, the towers $\T$ and $\T'$ have the same base field $\F_0$. Now, by induction we have that 
\[
\F_i'=\F_{i-1}'(\gamma_i)=\F_{i-1}(\gamma_i)=\F_{i-1}(\tilde{\alpha}_{i,\mathrm{D}}(\oT)\cdot\delta_i)=\F_{i-1}(\delta_i)=\F_i.
\]
Theorem \ref{theo-change of var} \eqref{theo-change 2 } follows from Corollary \ref{co: no denom}. Finally, to show \eqref{theo-change 3 } we consider the following rational map 
\begin{equation}\label{eq:other birat map}
 \begin{array}{ccc}
      \cV_{\T'}&\dashrightarrow &\cV_{\T}\\
        \displaystyle(\overline{t},\overline{\gamma})&\longmapsto&\displaystyle\left(\overline{t},\frac{\gamma_1}{\tilde{\alpha}_{1,\mathrm{D}}(\overline{t})},\ldots, \frac{\gamma_m}{\tilde{\alpha}_{m,\mathrm{D}}(\overline{t})}\right).
 \end{array}
\end{equation}
The domain of definition of this rational map is the Zariski-open subset $U':=\{\tilde{\alpha}_{1,\mathrm{D}}(\oT)\cdots \tilde{\alpha}_{m,\mathrm{D}}(\oT)\neq 0\}$.  Since the fraction field of $\cV_{\T'}$, which is $\F_m'=\F_m$, contains $\F_0$ and $U'$ is defined by polynomials on $\oT$, we deduce that $U'\cap \cV_{\T'}$ is nonempty. 
In particular, \eqref{eq:other birat map} is a well defined map on a dense Zariski-open subset of $\cV_{\T'}$.
Then, the birationality of \eqref{eq:some birat map} follows from the fact that \eqref{eq:some birat map} and \eqref{eq:other birat map} are mutually inverse  in $U'$.
\end{proof}

\begin{example}
    Consider the radical tower
      \[
      \T:=\left[\begin{array}{c}
        \F_0=\C(T)\subsetneq\F_1=\F_0[\Delta_1]/\langle \Delta_1^3-T^5\rangle
        \subsetneq\F_2=\F_1[\Delta_2]/\langle T\Delta_2^2-\Delta_1^2\rangle \\
       \subsetneq \F_3=\F_2[\Delta_3]/ \langle T\Delta_3^2-\Delta_2+T\rangle
        \end{array}
      \right].
      \]
      The ideal of the tower variety of $\T$ has $11$ generators. Applying Algorithm \ref{alg: tower} will result in a simpler tower variety. Applying Step 1 to build the first floor of the new tower we  get 
      \[
      \F_1'=\F_0(\gamma_1)=\F_0[\Gamma_1]/\langle \Gamma_1^3-T^5\rangle \text{ with } \gamma_1=\delta_1 \text{ and } \Gamma_1=\Delta_1.
      \]
      For the second floor, we have that $\tilde{\alpha}_{2,\mathrm{D}}(T)=\alpha_{2,\mathrm{D}}(T)$ and we get that 
      \[
        \F_2'=\F_1'(\gamma_2)=\F_1'[\Gamma_2]/\langle \Gamma_2^2-T\,\Gamma_1^2\rangle \text{ with } \gamma_2=T\,\delta_2 \text{ and } \Gamma_2=T\,\Delta_2.
      \]
      For the third floor, we apply the change of variables of Step 2 (a) to get
      \[
      \alpha_{3}(T,\Gamma_1,\Gamma_2 T^{-1})=\frac{\Gamma_2-T^2}{T^2}.
      \]
      Therefore, $\tilde{\alpha}_{3,\mathrm{D}}(T)= T^2$ and we get that 
      \[
       \F_3'=\F_2'(\gamma_3)=\F_2'[\Gamma_3]/\langle \Gamma_3^2-T^2(\Gamma_2-T^2)\rangle \text{ with } \gamma_3=T^2\,\delta_3 \text{ and } \Gamma_3=T^2\,\Delta_3.
      \]
      For this new radical tower $\T'=[\F_0\subsetneq\F_1'\subsetneq\F_2'\subsetneq\F_3']$, the tower variety is defined by the ideal 
      \[
     \left \langle
      \Gamma_1^3-T^5,
\Gamma_2^2-T\,\Gamma_1^2,
        \Gamma_3^2-T^2(\Gamma_2-T^2)
      \right\rangle.
      \]
\end{example}

\section{Universal Radical Parametrizations}\label{sec:universal-param}
  In this section, we introduce the notion of universal parametrization, establish its main properties, and illustrate how this notion can be used to strengthen the results obtained in some applications.

\subsection{Notion and Properties}

The construction of the tower variety introduced in \cite{RSV2017}, or in Section \ref{sec:preliminares}, is  a priori dependent on the chosen parametrization. Nevertheless, the variables involved in the equations of the tower variety  only depend on the chosen tower and not on the parametrization (see Corollary \ref{co: no need of para}). This motivates the following concept.

\begin{definition}\label{def-uni}
Let $\T$ be as in \eqref{tower}. 
We define the \textsf{universal radical parametrization}  of $\T$,  
   as the radical  parametrization $(\T,\cPU)$ where
    \[
    \cPU(\oT)=(\oT,\Delta_1(\oT),\ldots,\Delta_m(\oT)),
    \]
    and $\oT=(T_1,\ldots,T_n)$.
    \end{definition}

\begin{remark} \
\begin{enumerate}
\item  In Def. \ref{def-uni},  $\Delta_i$ represents the analytic radical map
    from $\C^n$ to $\C$ given by the 
    principal $e_i$-th root  $\ot \mapsto\delta_i(\ot):=\sqrt[e_i]{\alpha_i(\ot)}$. In particular, we can identify the universal radical parametrization $\cPU$ with the map $\psi$ in Diagram \ref{fig:diagram2}. In Section \ref{sec:preliminares}, the construction of the map $\psi$ depends on a given parametrization. By Theorem \ref{thm: incidence var}, $\psi$ should depend only on the radical tower.
    The universal radical parametrization $\cPU$ is the redefinition of the map $\psi$ arising intrinsically from the tower.
\item Note that $\cPU\in \F^{n+m}_{m}$, i.e. $r=n+m$, and its Jacobian w.r.t. $\oT$ has generic rank $n$.
\item  Observe that the polynomials in \eqref{eq-pols}, corresponding to $\cPU$, are  
\begin{equation}\label{eq-pols-U}
\left.
\begin{array}{llll} 
    E_i(\oT,\oD)& := & \alpha_{i,\mathrm{D}}(\oT)\Delta_i^{e_i} - \alpha_{i,\mathrm{N}}(\oT,\oD),& i\in \{1,\ldots,m\}   \\
    G_{i,\cU}(X_i,\oT)& := &X_i -T_i,& i\in \{1,\ldots,n\} \\
    G_{j,\cU}(X_j,\oD)& := &X_{n+j} -\Delta_j,&  j\in \{1,\ldots,m\} \\
    G_W(W,\oT)& := &W\cdot \mathrm{lcm}(\alpha_{1,\mathrm{D}},\ldots,\alpha_{m,\mathrm{D}})-1 &
    \end{array} \right\}
\end{equation}
where $\alpha_{i,\mathrm{N}},\alpha_{i,\mathrm{D}}$ are as in \eqref{eq-hE}.
\end{enumerate}
\end{remark}

 In this situation, the following theorem holds.

\begin{theorem}\label{theo: univ para var}
The varieties  $\mathscr{A}_{\cPU}$, $\mathscr{V}_{\cPU}$ and $\cT$ are isomorphic.
\end{theorem}
\begin{proof}
We reason directly with Diagram \ref{fig:diagram2}, specialized to the case of $\cPU$, and we analyze
the maps $\pi$ and $\pi^\star$.  It is direct to see that both maps are indeed isomorphisms whose inverses are
\[ \begin{array}{cccc}
\pi^*: & \mathscr{A}_{\cPU} & \longrightarrow & \cT, \\
 & (\ot, \od, \ox) & \longmapsto & (\ot,\od).
\end{array} \,\,\, \begin{array}{cccc}
(\pi^*)^{-1}: &  \cT  & \longrightarrow &\mathscr{A}_{\cPU}   \\
 & (\ot,\od)& \mapsto &  (\ot, \od, (\ot,\od)).
\end{array} \]
\[ \begin{array}{cccc}
\pi: & \mathscr{A}_{\cPU} & \longrightarrow & \mathscr{V}_{\cPU}, \\
 & (\ot, \od, \ox) & \mapsto & \ox.
\end{array} \]
\[ \begin{array}{cccc}
\pi^{-1}: &  \mathscr{V}_{\cPU}  & \longrightarrow &\mathscr{A}_{\cPU}   \\
 & \ox:=(x_1,\ldots,x_n,x_{n+1},\ldots,x_{n+m})  & \longmapsto &  ((x_1,\ldots,x_n), (x_{n+1},\ldots,x_{n+m}), \ox).
\end{array} 
\]
   So, the statement holds.
\end{proof}

  In \cite{RSV2017}, the degree of   $\pi$ is called the \textsf{tracing index of the parametrization}. Taking into  account Diagram \ref{fig:diagram2}, we have that $\pi$ is 1:1 if and only if $\cR$ is 1:1.  Therefore, we get the following corollary.

 \begin{corollary}\label{cor:tracing index}
The tracing index of every universal parametrization is one.
 \end{corollary}
 
\begin{remark}
By Diagram \ref{fig:diagram2}, we have that 
$\cR=\pi\circ (\pi^\star)^{-1}: \cT \rightarrow \mathscr{V}_{\cPU}; (\ot,\od) \mapsto (\ot,\od).$ 
\end{remark}

\begin{remark}\label{rmk: simple diag}
    Associated to every parametrization $(\T,\cP)$ we have a universal parametrization $(\T,\cPU)$. We obtain   the following commutative diagram.  
\[
\begin{tikzcd}
& \cV_\T \simeq\cV_{\cP_\cU} \arrow[ld, "\cR"'] \\
\cV_\cP & \C^n \arrow[u, "\cP_\cU"'] \arrow[l, "\cP"]
\end{tikzcd}
\] This diagram appears hidden in Diagram \ref{fig:diagram2} by identifying the map $\psi$ with $\cPU.$
\end{remark}

\subsection{Revisiting Some Applications}

In some applications, we   face  a problem whose input involves radicals, possibly nested, of polynomials, and the goal is to determine a rational change of variables that transforms the original problem into an equivalent one involving no radicals. To this end, a radical parametrization is associated with the input, and the problem is addressed by exploiting the relationship between its tower variety and its radical variety. We know that if the tower variety is unirational, then so is the radical variety. However, the converse implication does not hold in general (see Example \ref{ex: minimal tower better}). Nevertheless, it suffices to observe that the parametrization usually associated with the problem turns to be universal and, consequently, the two varieties are isomorphic (see Theorem \ref{theo: univ para var}). This property allows us to strengthen the results obtained in these applications till now.

In \cite{CSSV2021,FS2024} different problems concerning differential equations are considered.  In \cite{CSSV2021} the authors focus on ODEs and PDEs  whose coefficients involve,  possibly nested, radicals of polynomials in the independent variables. In \cite{FS2024}, the same type of differential equations are considered, but radicals involving the dependent variables and their derivatives are also allowed. The problem is then to decide whether there exists a rational change of variables that transforms the given differential equation into one whose coefficients are rational functions and, if so, to determine such a change of variables.

We illustrate this observation in the case of \cite{CSSV2021} for the case of ODEs. The reasoning for improving the results in the other applications is analogous.  We consider a radical tower. 
Let $x$ be an independent variable, then 
\begin{equation}
\T:=\left[\mathbb{F}_0 = \mathbb{C}(x) \subsetneq \cdots \subsetneq \mathbb{F}_{m-1} \subsetneq \mathbb{F}_m\right]
\end{equation}
where $\F_i,\delta_i,\alpha_i,e_i$ are as in \eqref{tower},\eqref{eq-delta}.  Let $F\in \F_m[w_0,\ldots, w_s]$. We consider the differential equation
\begin{equation}\label{eq-F}
F(x,y,y',\ldots,y^{(s)})=0,
\end{equation} 
where $y$ is a function on $x$.  The problem is to find, if it exists,
a rational change of variable of the form 
$x = r(z)$ 
such that the new equation
\begin{equation}\label{eq-G}
G\left(r(z),Y(z),\dots,Y^{(s)}(z)\right) = 0, \quad \text{where } Y(z) = y(r(z)),
\end{equation}
is algebraic; that is, all the coefficients are rational functions. Moreover, in the affirmative case, to control how the solutions of  \eqref{eq-F} and \eqref{eq-G} relate. 

In \cite{CSSV2021}, the authors introduce the radical parametrization $(\T,\cP:=(x,\overline{a}(x)))$, where  $\overline{a}(x) \in \mathbb{F}_m^{\ell}$
is the tuple of non-rational (function) coefficients of $F(x,\overline{w})$ w.r.t. $\overline{w}$,
taken in any order. Then, they provide the solution to the problem in case $\cV_\T$ is rational, leaving open  the case where   the tracing index is not 1. Using the results in this section we  go a step further. More precisely, we have the following result.
\begin{remark}\label{rem-edo}
Let $\overline{a}=(a_1,\dots,a_\ell)$  be  such that, 
 \begin{equation}\label{tower2}
\T^*:=\left[\C(x)=:\F_0^*\subsetneq \F_1^*=\F_0^*(a_1)\subsetneq\F_2^*=\F_1^*(a_2)\subsetneq \cdots \subsetneq \F_\ell^*=\F_{\ell-1}^*(a_\ell) \right], 
\end{equation} 
with $a_{j}^{e_j}\in  \F^{*}_{j-1}$, and $e_j\in \mathbb{Z}_{>1}$, is a radical tower.
 Then $(\T^*,\cP=(x,\overline{a}(x)))$ is the universal parametrization and, hence, $\mathscr V_{\mathcal P}$ and $ \mathscr V_{\mathbb T^*}$ are isomorphic (see Theorem \ref{theo: univ para var}). Thus, \eqref{eq-G}  can be achieved if and only if $\mathscr V_{\mathbb T^*}$ is rational. 
\end{remark}
More generally, the conclusion of Remark \ref{rem-edo} holds whenever   $(\mathbb T^*,(x,\overline a))$
is strong (Definition \ref{def-strong}), since then $\mathscr V_{\mathscr P}$ and
$\mathscr V_{\mathbb T^*}$ are birational (Corollary \ref{cor-fuerte}).

\section{Minimal Radical Parametrizations}\label{sec:subtower}

 Throughout this section, we will work with the tower $\T$ introduced in \eqref{tower}  as the \textit{main} tower, and we will study subtowers of $\T$.  In this section, we will always assume that the elements of $\F_m$ are represented in canonical form, including $\alpha_i$ and $p_j$ (see Def. \ref{def-canonical}).

\subsection{Radical Subtowers}

First, we introduce the notion of radical subtower of $\T$.

\begin{definition}
Let 
$(i_1,\ldots,i_k)$  be a sequence of $k$  distinct elements of 
$\{1,\ldots,m\}$. A \textsf{radical subtower of $\T$ (associated to $(i_1,\ldots,i_k)$)}  is a radical tower of the form 
\begin{equation}\label{eq-subtower}
\T_{(i_1,\ldots,i_k)}:=[\F_0\subsetneq\F_{i_1}'\subsetneq \F_{i_2}'\subsetneq \cdots\subsetneq\F_{i_k}']
\end{equation}
such that $\F_{i_j}'=\F_{i_{j-1}}'(\delta_{i_j})$ and $\alpha_{i_j}\in \F_{i_{j-1}}'$ for  $1\leq j\leq k$, where $\delta_{i_j}\in \{\delta_1,\ldots,\delta_m\}$, and $\alpha_{i_j}\in \{\alpha_1,\ldots,\alpha_m\}$ with $\delta_i,\alpha_i$  as in \eqref{eq-delta}, and $\F_{i_0}':=\F_0$.  In addition, we define $\T_{\emptyset}$ as 
\[
\T_{\emptyset}:=[\F_0]. \]
We write $\T_{(i_1,\ldots,i_k)}\subseteq \T$ to express that  $\T_{(i_1,\ldots,i_k)}$  is a radical subtower of $\T$.
\end{definition}

In other words, a radical subtower $\T_{(i_1,\ldots,i_k)}$  of $\T$ is a radical tower in which in each step of the tower we add one of the radical elements $\delta_{i_1},\ldots,\delta_{i_k}$ of the initial tower. In terms of the fields descriptions, $\T_{(i_1,\ldots,i_k)}$ is a radical subtower if
\[
\F_{i_j}'= \F_{i_{j-1}}'[\Delta_{i_j}]/\langle \Delta_{i_j}^{e_{i_j}}-\alpha_{i_j}\rangle \ \text{ with } \alpha_{i_j}\in \F_{i_{j-1}}'\ \text{ for } \ 1\leq j\leq k.
\]

 \begin{example}
    Consider the radical tower
    \[
    \T:=\left[ \F_0=\C(T)\subsetneq \F_1=\F_0\left(\sqrt{T}\right)\subsetneq  \F_2=\F_1\left(\sqrt[3]{T}\right)\subsetneq  \F_3=\F_2\left(\sqrt{T+\sqrt{T}}\right) \right].
    \]
    Then, the tower $\T_{(1,2)}:=[\F_0\subsetneq \F_1\subsetneq  \F_2]$ is a radical subtower of $\T$. Another example is
    \[
    \T_{(1,3)}:= \left[\F_0\subsetneq \F_1\subsetneq \F_2':=\F_1\left(\sqrt{T+\sqrt{T}}\right) \right]\]
    which is a radical subtower of $\T$ since each floor of $ \T_{(1,3)}$ is a radical field extension and it is obtained by 
    adding one of the $\delta_1,\delta_3$ of the initial tower.

    Finally, a nonexample is the case of the radical tower $\F_0\subsetneq\F_0(\sqrt[6]{T})$. This is a radical extension satisfying that $\F_0(\sqrt[6]{T})=\F_2$, but $\sqrt[6]{T}$ is not one of the $\delta_i$ of the initial tower.
\end{example}

\begin{example}\label{rmk:order}
We note that the sequences  associated to the radical subtower  might not follow the order given by the initial tower. For instance, consider the tower
\[
\T:= \left[ \F_0:=\C(T)\subsetneq\F_1:=\F_0(\sqrt{T})\subsetneq\F_2:=\F_1(\sqrt{T^2+T+1})\subsetneq\F_3:=\F_2(\sqrt[3]{\sqrt{T}+1}) \right].
\]
We have the subtowers of $\T$
\[
\begin{array}{cll}
\T_{(1,2)}:=& \left[\F_0\subsetneq\F_1\subsetneq\F_2\right] & \subseteq \T, 
\\
\T_{(2,1)}:=&\left[\F_0\subsetneq\F_1':=\F_0(\sqrt{T^2+T+1})\subsetneq\F_2':=\F_1'(\sqrt{T})\right] & \subseteq \T, 
\end{array}
\]
which do not follow the same order. But  not every sequence or order leads to a radical subtower. For instance, the $\T_{(3,1)} \not\subseteq\T$   since to build the radical $\sqrt[3]{\sqrt{T}+1}$, one needs $\sqrt{T}$. On the contrary, changing the order we get that $\T_{(1,3)}\subseteq\T$.
\end{example}

The following example shows the importance of choosing the canonical representation.

 \begin{example}\label{ex: representatives}
    Consider the radical tower $\T:=[\F_0=\C(T)\subsetneq\F_1\subsetneq\F_2\subsetneq\F_3]$ given by the polynomials
    \begin{equation}\label{eq-1}
    E_1=T\Delta_1^2-1,\ E_2=T\Delta_2^2-T-1,\ \text{ and } E_3=\Delta_3^3-\Delta_2^2.
    \end{equation}
    The ideal of the tower variety $\cV_\T$ of $\T$ is $\mathcal{E}=\langle E_1,E_2,E_3\rangle$. Representing the polynomials  in the canonical form we have 
  \begin{equation}\label{eq-2}
    E_1=T\Delta_1^2-1,\ E_2=T\Delta_2^2-T-1,\ \text{ and } E_3=\Delta_3^3-\Delta_2^2=T\Delta_3^3-T-1.
   \end{equation}
   Let us analyze whether $\T_{(3)}\subseteq \T$. From \eqref{eq-1}, the answer might not be immediate. However, from \eqref{eq-2} one has that
   \[
    \Delta_3^3=\frac{T+1}{T},
    \]   
   and the answer is clearly positive. 
   
\end{example}

Example \ref{rmk:order} motivates the next definition.

\begin{definition}
Let $\T_{(i_1,\ldots,i_k)},\T_{(j_1,\ldots,j_k)}\subseteq\T$. We say that  $\T_{(i_1,\ldots,i_k)}$ and $\T_{(j_1,\ldots,j_k)}$ are \textsf{equivalent} if  $\{i_1,\ldots,i_k\}=\{j_1,\ldots,j_k\}$ and we write $\T_{(i_1,\ldots,i_k)}\equiv \T_{(j_1,\ldots,j_k)}$.
\end{definition}
\begin{notation}\label{not-1} Let   $\T_1,\T_2$ be towers such that the last field of one of the towers is included in the last field of the other tower; this happens, for instance, if $\T_1\subseteq\T_2$. Then, we may have the same parametrization $\cP$ referred to different towers. To emphasize this we will write $\cP_{\T_i}:=(\T_i,\cP)$. Similarly, we will use $\mathscr{V}_{\T_i}, \mathscr{A}_{\cP_{\T_i}}, \mathscr{V}_{\cP_{\T_i}}$.
\end{notation}

The next proposition analyzes the behavior of the varieties associated to parametrizations with equivalent towers.

\begin{proposition}\label{prop: no order}
    Let  $\cP_{\T_R}:=(\T_R,\cP), \cP_{\T_S}:=(\T_S,\cP)$ be radical parametrizations with  $\T_R\equiv \T_S$. 
    Then, $\mathscr{A}_{\cP_{\T_R}}$, $\cV_{\cP_{\T_R}}$, $\cV_{\T_R}$ 
    are isomorphic to $\mathscr{A}_{\cP_{\T_S}}$, $\cV_{\cP_{\T_S}}$,  $\cV_{\T_S}$, respectively.
    Furthermore, $\cV_{\cP_{\T_R}}=\cV_{\cP_{\T_{S}}}$.
\end{proposition}
\begin{proof}
    Without loss of generality, we may assume that $\T_R=\T$ and $\T_S$ is obtained by permuting the order of some floors of $\T$. So, there exists a permutation $\sigma$ of $\{1,\ldots,m\}$ such that $S=(\sigma(1),\ldots,\sigma(m))$.
    First we show that $\cV_\T\simeq \cV_{\T_S}$, i.e. they are isomorphic. 
    Note that both $\cV_\T$ and $\cV_{\T_S}$ are contained in the affine space $\C^{n+m}$ (up to the permutation of the variables via $\sigma$). Let $i\in \{1,\ldots,m\} $ and let $j$ be such that $\sigma(j)=i$. 
    Since $\T_{S}$ is a radical subtower, we have that $\alpha_{i}$ is contained in $\F_{i-1}$ and $\F_0(\delta_{\sigma(1)},\ldots,\delta_{\sigma(j-1)})$. As $\alpha_i$ is written in canonical form,  the variables appearing in $\alpha_{i}$ are  only $\overline{T}$ and $\Delta_a$ for $a\in\{1,\ldots,i-1\}\cap\{\sigma(1),\ldots,\sigma(j-1)\}$. Therefore, the polynomial $E_i=E_{\sigma(j)}$ is contained in the rings $\C[\overline{T},\Delta_1,\ldots,\Delta_{i}]$ and $\C[\overline{T},\Delta_{\sigma(1)},\ldots,\Delta_{\sigma(j)}]$. In particular, we can use the same polynomials $E_{\sigma(1)},\ldots,E_{\sigma(m)}$ for the construction of the tower $\T_S$.   Using Proposition \ref{prop- new E} and the ideal \eqref{eq:saturation ideal}, we deduce that the ideal of $\cV_{\T_R}$ and $\cV_{\T_S}$ coincide up
to the permutation of the variables via $\sigma$. So they are isomorphic. The same argument works for the incidence varieties $\mathscr{A}_{\cP_{\T_{R}}}$ and $\mathscr{A}_{\cP_{\T_{S}}}$. Finally, $\cV_{\cP_{\T_S}}=\cV_{\cP_{\T_R}}$ since they are both the closure of the image of the same variety via the same projection.
\end{proof}

 Proposition \ref{prop: no order} shows that the order of the floors of a radical subtower does not matter. Therefore, we may work with the equivalence class of radical subtowers up to reordering. In particular, we may see the sequence $S$ associated to a radical subtower $\T_S$ as a subset of $\{1,\ldots,m\}$. Working with these equivalent classes allows us to define a partial order on the set of radical subtower of $\T$ given by the inclusion on the subsets of $\{1,\ldots,m\}$ with $\T_{\emptyset}=\F_0$ and $\T$ being the minimal and maximal elements of this order respectively. Based on this, in the sequel we will speak about subtowers associated to subsets of $\{1,\ldots,m\}$.

Besides equivalent towers, the phenomenon described above can also happen with subtowers of $\T$ (see Notation \ref{not-1}). To be more precise, let $\cP\in\F_{m}^{r}$ be a radical parametrization and let $\T_S$, with 
$S=\{i_1,\ldots,i_k \}\subseteq\{1,\ldots,m\}$, be a subtower of $\T$ such that  $\cP_S\in(\F_{i_k}')^{r}$. We analyze the relation between 
  $\cA, \cV_\cP$ and $\cT$ and $\mathscr{A}_{\cP_{\T_S}}, \cV_{\cP_{\T_S}}$ and $\cV_{\T_S}$.  The answer for the incidence and tower varieties is postponed to Theorem \ref{theo:initial maps} and Corollary \ref{co: univ maps}, where a more general case is treated and which essentially state that they are connected via a dominant map. Now, we deal with the radical varieties.

\begin{theorem}\label{thm:subtower var}  
$\cV_\cP= \cV_{\cP_{\T_S}}$.
    \end{theorem} 
    \begin{proof} 
Let us assume w.l.o.g. that $S=\{1,\ldots,k\}$ and let us denote  $\hat{\delta}=(\delta_1,\ldots,\delta_k)$. Consider the ideals $K=\langle E_1,\ldots,E_m,G_1,\ldots,G_r\rangle$ 
and $K_S=\langle E_1,\ldots,E_k,G_1,\ldots,G_r\rangle$.  As in the proof of Theorem \ref{thm: incidence var}, let $U_\cP$ be the Zariski-open subset of $\C^{n+m+r}$ on which all the denominators of the construction do not vanish. Similarly, let $U_S$ be the analogous Zariski-open subset for $\T_S$.
Note that $\cA=\overline{\V(K)\cap U_\cP}$ and
$\mathscr{A}_{\cP_{\T_S}}=\overline{ \V(K_S)\cap U_S}$.
We consider the projection  
  \[ \Theta :\C^{n+m+r}  \longrightarrow \C^{n+k+r}; (\ot,\od,\ox) \longmapsto (\ot,\hat{\delta},\ox). \]
 Since $\Theta $ is a projection, it is open and $\Theta(U_\cP)\subseteq U_S$ is a Zariski-open subset.
Now, by elimination theory, we have that $\overline{\Theta(\V(K))}\subseteq\V(K_S)$. In particular, we get that
$$\Theta(U_\cP\cap\V(K))\subseteq \Theta(U_\cP)\cap\Theta(\V(K)) \subseteq \Theta(U_\cP)\cap \V(K_S)\subseteq U_S\cap\V(K_S),$$
and we deduce that $\overline{\Theta(\cA)}\subseteq\mathscr{A}_{\cP_{\T_S}}$. 

Next, we claim that $\Theta(U_\cP)\cap \mathscr{A}_{\cP_{\T_S}}$ is a nonempty dense Zariski-open subset of $\mathscr{A}_{\cP_{\T_S}}$. Since $\Theta$ is open and $\mathscr{A}_{\cP_{\T_S}}$ is irreducible by Theorem \ref{thm: incidence var}, it is enough to check that $\Theta(U_\cP)\cap \mathscr{A}_{\cP_{\T_S}}$ is  nonempty. On the contrary, assume that $\Theta(U_\cP)\cap \mathscr{A}_{\cP_{\T_S}}=\emptyset$. Then, there exists a denominator in the construction that vanishes at any choice of $(\overline{t},\overline{\delta})\in\C^{n+m}$ satisfying the equations of the tower variety. This implies that such denominator is zero as an element in $\F_m$, which leads to a contradiction. Therefore, we conclude that $\Theta(U_\cP)\cap \mathscr{A}_{\cP_{\T_S}}$ is a nonempty dense Zariski-open subset of $\mathscr{A}_{\cP_{\T_S}}$. Now, $\Theta$ restricted to $U_\cP\cap\cA$ is a finite map into $\Theta(U_\cP)\cap \mathscr{A}_{\cP_{\T_S}}$. Since $\dim \cA=\dim \mathscr{A}_{\cP_{\T_S}} = n$ and both $\cA$ and $\mathscr{A}_{\cP_{\T_S}}$ are irreducible, we conclude that $\overline{\mathrm{Im}\ \Theta( \cA)}=\mathscr{A}_{\cP_{\T_S}}$. Now, consider the commutative diagram 
\begin{equation}\label{eq:some cd}
    \begin{tikzcd}
\cA \arrow[r, "\Theta"] \arrow[d, "\pi"'] & \mathscr{A}_{\cP_{\T_S}} \arrow[ld, "\pi_S"] \\
\C^r                                      &                          
\end{tikzcd}
\end{equation}
where $\pi_S$ is the analogous projection to $\pi$ in Diagram \ref{fig:diagram2} for the subtower $\T_S$. By Diagram \eqref{eq:some cd} and using that $\overline{\Theta(\cA})=\mathscr{A}_{\cP_{\T_S}}$ 
we conclude that 
\begin{equation*} \cV_{\cP}=\overline{\mathrm{Im}\ \pi(\cA)}=\overline{\mathrm{Im}\ \pi_S(\mathscr{A}_{\cP_{\T_S}})}=\cV_{\cP_{\T_S}}.   \qedhere 
\end{equation*}
 \end{proof}

\begin{remark}\label{rem:smaller tower}
The use of a smaller tower  leads  to a more efficient diagram, since the dimension of the tower varieties stay the same but the dimension of the ambient affine space is smaller. On the other hand,  rationality may change from one tower variety to the other; see Example  \ref{ex: minimal tower better}.
\end{remark}

  \subsection{Minimal Radical Parametrizations}

    Corollary \ref{co: no need of para} shows that the tower variety does not depend on the parametrization,  but on the fixed radical tower we decide to work with. However,  given a parametrization in a radical tower, a smaller radical subtower containing the parametrization might lead to a more efficient construction. The following natural question arises: given a parametrization $\cP$ is there a minimal tower we can work with? We   approach this question here.
   We start with the following remark. 
   
   \begin{remark}\label{remark-bases}
   Let
$
\T_R:= [\F_0\subsetneq \F_1'\subsetneq\cdots\subsetneq\F_r']
$
   be a subtower of $\T$. We observe that, if $\Delta_{i_1}^{u_{i_1}}\cdots\Delta_{i_k}^{u_{i_k}}\in \Base{\T_R,\F_k'}$ and $u_{i_j}>0$, then $i_j\in R.$  
\end{remark}

\begin{proposition}\label{prop:inters of tower}
Let $\T_R,\T_S\subseteq \T$, then  $\T_{R\cap S},\T_{R\cup S}\subseteq \T$.
\end{proposition}
\begin{proof}
Let $R=\{i_1,\ldots,i_r\}$ and $S=\{j_1,\ldots,j_s\}$ and write their corresponding towers as
\[
\begin{array}{cc}
\T_R:=& [\F_0\subsetneq \F_1':=\F_0(\delta_{i_1})\subsetneq\cdots\subsetneq\F_r':=\F_{r-1}'(\delta_{i_r})],
    \\
\T_S:=&[\F_0\subsetneq \widehat{\F}_1:=\F_0(\delta_{j_1})\subsetneq\cdots\subsetneq\widehat{\F}_s:=\widehat{\F}_{s-1}(\delta_{j_s})].
\end{array}
\]
First, we show that $\T_{R\cup S}$ is a radical subtower. Let $S\setminus R = \{j_{a_1},\ldots,j_{a_k}\}$  with $ a_1<\cdots<a_k$. Then $R\cup S= \{i_1,\ldots,i_r,j_{a_1},\ldots,j_{a_k}\}$ and its associated radical subtower $\T_{R\cup S}$ is 
\[
\F_0\subsetneq\F_1'\subsetneq\cdots\subsetneq\F_r'\subsetneq
\F_r'(\delta_{j_{a_1}})
\subsetneq\cdots\subsetneq\F_r'(\delta_{j_{a_1}},\ldots,\delta_{j_{a_k}}).
\]
To check that $\T_{R\cup S}$  is indeed a radical subtower, it is enough to see that each extension is radical.
The first $r$ floors of $\T_{R\cup S}$ are radical extensions since they are also floors of  $\T_R$. Now, for $1\leq l\leq k$, we have that 
\[
\alpha_{j_{a_l}}\in \widehat{\F}_{j_{a_l-1}}
\subseteq
\F_r'(\delta_{j_{a_1}},\ldots,\delta_{j_{a_{l-1}}}).
\]
Note that $ \widehat{\F}_{j_{a_l-1}}
\subseteq
\F_r'(\delta_{j_{a_1}},\ldots,\delta_{j_{a_{l-1}}})$ since every $j_c$ with $c\leq a_l-1$ is either in $R$ or in $\{j_{a_1},\ldots,j_{a_{l-1}}\}$.
We conclude that 
\[
\F_r'(\delta_{j_{a_1}},\ldots,\delta_{j_{a_{l-1}}})\subsetneq \F_r'(\delta_{j_{a_1}},\ldots,\delta_{j_{a_{l}}})
\] is a radical extension.

For $R\cap S$, we write $R\cap S=\{i_{a_1},\ldots,i_{a_k}\}$ as a subset of $R$  with $a_1<\cdots<a_k$, and $R\cap S=\{j_{b_1},\ldots,j_{b_k}\}$ as a subset of $S$. Say $i_{a_1}=j_{b_1},\ldots,i_{a_k}=j_{b_k}$. As before, to check that the tower
\[\begin{array}{cc}
\T_{R\cap S}:& \left[\F_0\subsetneq\tilde{\F}_1:=\F_0(\delta_{i_{a_1}})\subsetneq\cdots\subsetneq \tilde{\F}_k:=\tilde{\F}_{k-1}(\delta_{i_{a_{k}}}) \right]
\end{array}
\]
is a radical subtower of $\T$, it is enough to check that $\alpha_{i_{a_l}}\in\tilde{\F}_{l-1}$  for every $1\leq l\leq k$.  $\L:=\F_{i_{a_l-1}}'\cap \widehat{\F}_{j_{b_l-1}}$ is a non-zero vector subspace of $\F_m$ with canonical basis  (Remark \ref{remark-bases1})
\[ 
\Base{\T,\L}=
\Base{\T,\F_{i_{a_l-1}}'} \cap \Base{\T,\widehat{\F}_{j_{b_l-1}}}=
\Biggl\{
\prod_{\substack{
1\leq \lambda\leq l-1\\
b_\lambda<b_l
}}
\Delta_{i_{a_\lambda}}^{u_\lambda}
\;:\;
0\leq u_\lambda\leq e_{i_{a_\lambda}}-1,
\quad \lambda=1,\ldots,l-1
\Biggr\}.
\]
In particular, we obtain that 
\[
\Base{\T,\L}\subseteq \left\{
\Delta_{i_{a_1}}^{u_1}\cdots  \Delta_{i_{a_{l-1}}}^{u_{l-1}}\ :\ 0\leq u_\lambda\leq e_{i_{a_\lambda}}-1
\text{ for } \lambda=1,\ldots,l-1
\right\}.
\]
We deduce that 
 $\L\subset \tilde{\F}_{l-1}$ and, by construction, $\alpha_{i_{a_l}}=\alpha_{j_{b_l}}\in \L$. Thus,  $\alpha_{i_{a_l}}\in\tilde{\F}_{l-1}$.
\end{proof}

Using Proposition \ref{prop:inters of tower}, we derive the following result.

\begin{corollary}\label{cor:minimal tower}
    Given   $\mathscr{F}\subset \F_m$,  there exists  {$R\subset\{1,\ldots,m\}$} such that $\T_{R}$ is the minimal radical subtower of $\T$ containing $\mathscr{F}$.
\end{corollary}
\begin{proof}
Let $R$ be the intersection of all subsets $S\subseteq\{1,\ldots,m\}$ such that 
$\T_S$ is a radical subtower of $\T$ containing $\mathscr{F}$. By Proposition \ref{prop:inters of tower}, 
$\T_{R}$ is a radical subtower. By construction, $\T_{R}$ is the minimal radical subtower of $\T$ containing $\mathscr{F}$.
\end{proof}

Corollary \ref{cor:minimal tower} leads to the following definition.

\begin{definition}
    Let $\cP=(p_1,\ldots,p_r)\in \F_{m}^{r}$ be a  radical parametrization. The minimal radical subtower  containing $p_1,\ldots,p_r$ is called the \textsf{minimal radical subtower of $\cP$}, and it is denoted by  $\mT$.
    We say that $(\mT,\cP)$ is a \textsf{minimal radical parametrization} of $\cP$. 
\end{definition}

Applying Corollary \ref{cor:minimal tower} to the set defined by all the components of a parametrization, one gets the next result.

\begin{corollary}
Every radical variety admits a minimal radical parametrization.   
\end{corollary}

Taking into account Definition \ref{def-uni}, one has the following result.

\begin{corollary}
Every universal radical parametrization is minimal.   
\end{corollary}

 The following example exhibits the advantages of working with the  smallest radical subtower associated to a parametrization.

\begin{example}\label{ex: minimal tower better}
    Continuation of Example \ref{ex: representatives}. We take the tower $\T$, introduced there with the polynomials $E_i$ written in canonical form. Now, we consider 
    \[
    \cP=(\Delta_2+T-1,\Delta_2\Delta_3-T^2+1)\in\F_3^2
    \]
    and the radical parametrizations $(\T,\cP)$ and $(\mT,\cP)$.     Since $\Delta_1$ does not appear in   $\cP$, one may check that the minimal radical subtower of $\cP$ is 
    \[
    \T_{\{2,3\}}:=\left[
    \F_0\subsetneq \F_1'=\F_0[\Delta_2]/\langle E_2\rangle
    \subsetneq \F_2'=\F_1'[\Delta_3]/\langle E_3\rangle
    \right].
    \]
    As we have seen in Theorem \ref{thm:subtower var}, $\cV_\cP$ does not depend on which tower we work with. Indeed, in this example, the radical variety is a plane curve  of genus $0$, defined by a polynomial of degree $17$ with $97$ terms.
    Using $(\T,\cP)$ we get the tower variety
    \[ \cT=\V(T\Delta_{1}^{2}-1, T \Delta_{2}^{2}-T -1, T \Delta_{3}^{3}-T -1)\subset \C^{4}. \]
    However, the genus of $\cT$ is 2. Using $(\mT,\cP)$ we get 
    \[  \mathscr{V}_{\mT}= \V( T\Delta_{2}^{2}-T -1, T \Delta_{3}^{3}-T -1)\subset \C^{3} \]
    that has genus 0. Indeed, it can be parametrized birationally as 
    \[ \mathscr{T}:=\left(\dfrac{1}{h^{6}-1}, h^{3}, h^{2}\right). \]
     So, using Remark  \ref{Remark-first} (\ref{Rem2}) we have that $\cV_\cP$ can be parametrized as
    \[ \cP\circ\mathscr{T} = \left(\frac{h^{9}-h^{6}-h^{3}+2}{h^{6}-1}, 
\frac{h^{5} \left(h^{12}+h^{7}-2 h^{6}-2 h +1\right)}{\left(h^{6}-1\right)^{2}}
\right).
\]
\end{example}

\subsection{Computing Minimal Radical Parametrizations}\label{sec:algorithms}

In this subsection we study how to identify whether a parametrization is minimal or not, and if not, how to find the minimal radical subtower. For this purpose, we assume we are given a parametrization $(\T,\cP)$ as in the previous sections. We start with a technical lemma.

\begin{lemma}\label{lem: lin alg}
    Let $f\in\F_m$, let $\T_S$ be the minimal radical subtower of $\T$ containing $f$. If $\Delta_j$ appears in $f$, when written in canonical form,  then $j\in S$.
\end{lemma}
\begin{proof}
Let $j\in\{1,\ldots,m\}$ be such that $\Delta_j$ appears in the canonical expression of $f$. Then, there exist
$u_1,\ldots,u_m$ with $0\leq u_i\leq e_i-1, i\in \{1,\ldots,m\}$, and $u_j>0$, such that $\Delta_1^{u_1}\cdots\Delta_m^{u_m}$ 
appears in $f$, and hence
$\Delta_1^{u_1}\cdots\Delta_m^{u_m}\in \Base{\T,\F_m}$. Now, the result  follows from Remark \ref{remark-bases}.
\end{proof}

Using Lemma \ref{lem: lin alg}, we derive the following algorithm.

\begin{algorithm}[H]
\floatname{algorithm}{Algorithm}
    \caption{Given $f\in \F_m$, in canonical form, it computes   the minimal  subtower containing $f$.}\label{alg:1}
    
    \begin{itemize}[leftmargin=1.2em, itemsep=2pt]
        \item \textbf{Input:} An element  $f\in \F_m$.
        \item \textbf{Output:}  $R\subseteq\{1,\ldots,m\}$ associated to the minimal subtower of $\T$ containing $f$.
        \begin{enumerate}[leftmargin=1em, itemsep=2pt]
            \item  Set $R=\emptyset$. 
            \item\label{step-bucle1} Add to $R$ each $i\in\{1,\ldots,m\}$ such that $\Delta_i$ appears in $f$.
            \item\label{step-bucle2}  For each new index $i$ added to $R$ in Step \ref{step-bucle1}, repeat  Steps \ref{step-bucle1}, \ref{step-bucle2} taking $f$ as $\alpha_i$.
            \item Return $R$.
        \end{enumerate}
    \end{itemize}
\end{algorithm} 

\begin{proposition}\label{prop: alg 1}
    Algorithm \ref{alg:1} is correct.
\end{proposition}
\begin{proof}
    First, note that the algorithm finishes: indeed, let $i$ be an index added to $R$ in Step \ref{step-bucle1}. Once we apply Steps \ref{step-bucle1} and \ref{step-bucle2} to $\alpha_i$, we repeat this process for a new index $j\in\{1,\ldots,m\}$ which has $j<i$. Therefore, this process finishes.  
    
    Let $R=\{i_1,\ldots,i_k\}$ be the subset output by Algorithm \ref{alg:1}.
    We may assume that $i_1<\cdots<i_k$ is the order induced on $R$ by the initial radical tower. 
    By construction, for every $1\leq a\leq k$, all the $\Delta$'s appearing in $\alpha_{i_a}$ (modulo $\E$, see \eqref{eq-E}) are in $R$. Therefore, they must be among $\Delta_{i_1},\ldots,\Delta_{i_{a-1}}$. So,   $\alpha_{i_a}$ is contained  in the previous field and, hence, $\T_R$ is a radical subtower of $\T$. Moreover, by Steps 2 and 3, the last field of $\T_R$ contains $f$.

    Finally, by Lemma \ref{lem: lin alg}, every index $a$ added to $R$ in Step \ref{step-bucle1} is also in the associated set of the minimal radical tower. Thus $\T_R$ is the minimal radical subtower.
    \end{proof}

Using Algorithm \ref{alg:1}, we present an algorithm for calculating the minimal radical subtower associated to a parametrization $\cP$.

\begin{algorithm}[H]
    \caption{It finds the minimal   subtower associated to a parametrization}\label{alg:3}
    \begin{itemize}[leftmargin=1.2em, itemsep=2pt]
        \item \textbf{Input:} A parametrization $(\T,\cP:=(p_1,\ldots,p_r)),$ where $\cP\in\F_m^r$.
        \item \textbf{Output:} The subset $S$ associated to the minimal  subtower containing $\cP$.
         \begin{enumerate}[leftmargin=1.em, itemsep=2pt]
            \item Set $S =\emptyset$.
            \item For every $i\in \{1,\ldots,r\}$ apply Algorithm \ref{alg:1} to $p_i$, and add to $S $ the corresponding output subset.
            \item Return $S $.
        \end{enumerate}
    \end{itemize}
\end{algorithm}

\begin{example}
        In this example we illustrate  Algorithm \ref{alg:3}. We consider the tower 
        \[  \T:=\left[ \F_0=\C(T)\subsetneq \mathbb{F}_1 \subsetneq \mathbb{F}_2 \subsetneq \mathbb{F}_3 \subsetneq \mathbb{F}_4   \right] \]
 where 
        \[ \begin{array}{ll}
  \mathbb{F}_1 = \mathbb{F}_0[\Delta_1]/\langle T\Delta_1^2 - 1\rangle, &
   \mathbb{F}_2 = \mathbb{F}_1[\Delta_2]/\langle T\Delta_2^2 - T - 1\rangle,\\
  \mathbb{F}_3 = \mathbb{F}_2[\Delta_3]/\langle T^{2} \Delta_{3}^{3}-T^{2}-2 T -1\rangle, & 
  \mathbb{F}_4 = \mathbb{F}_3[\Delta_4]/\langle  \Delta_{4}^{3}-\Delta_2-1\rangle
  ,\end{array}\]
which is in canonical form. 
Let  
\[ \cP=\left(\dfrac{T+1}{\Delta_2(\Delta_{4}-1)}, \Delta_{4}+\Delta_3^{4}, 
\Delta_{1}^{6}+\Delta_{2}^{4}
\right). 
\]
We rewrite $\cP$ as
\[ 
\cP=\left(T(\Delta_4^2+\Delta_4+1), \Delta_{4}+\Delta_3^{4}, 
\Delta_{1}^{6}+\Delta_{2}^{4}
\right).
\]
The first component is in canonical form, but the second and the third are not. Expressing canonically $\cP$, we get
\[
\cP=\left(T\Delta_4^2+T\Delta_4+T, \Delta_{4}+\frac{T^2+2T+1}{T^2}\Delta_3, 
\frac{1}{T^3}+\frac{(T+1)^2}{T^2}
\right). 
\]
Following Algorithm \ref{alg:3}, we apply Algorithm  \ref{alg:1} to $p_1$, $p_2$ and $p_3$. 
Step 2 of Algorithm \ref{alg:1} for $p_1$ adds $\Delta_4$ to the minimal subtower we aim to compute. Step 3 of Algorithm \ref{alg:1} applies the same process to $\alpha_4=\Delta_2+1$, and hence, it adds $\Delta_2$ to the minimal radical subtower. Since $\alpha_2$ depends only on the variable $T$, we deduce that the minimal radical subtower containing $p_1$ is $\T_{\{2,4\}}$.  Similarly, we obtain that the minimal radical subtower containing $p_2$ is $\T_{\{2,3,4\}}$. The minimal subtower containing $p_3$ is $\T_{\emptyset}$. 
We conclude that $\mT=\T_{\{2,3,4\}}.$
    \end{example}

Note that in Algorithm \ref{alg:3}, Step 2 does depend  on   $\Delta_j$  in the representation of $p_{i}$ with respect to the basis $\Base{\T,\F_m}$. If one fixes the radical tower and wants to run the above algorithm several times for distinct radical parametrizations, it might be useful to follow the  alternative approach developed below.

For a given radical $\Delta_i$ for $i\in \{1,\ldots,m\}$, we denote the minimal subtower containing $\Delta_i$ by $\T_i$,  and by $S_i$  the subset associated to $\T_i$. The next result allows us to calculate the poset of radical subtowers.

\begin{proposition}\label{prop: 2 deltas}
    For $i,j\in\{1,\ldots,m\}$, the subset associated to the minimal radical subtower containing $\Delta_i$ and $\Delta_j$ is $S_i\cup S_j$.
\end{proposition}
\begin{proof}
    Let $R $ be the subset associated to the minimal radical subtower $\T_R$ containing $\Delta_i$ and $\Delta_j$. Since $\Delta_i,\Delta_j$ are in $\T_R$,  we get that $S_i\cup S_j\subseteq R$. On the other hand, by Proposition \ref{prop:inters of tower}, $\T_{S_i\cup S_j}$ is a radical subtower and contains $\Delta_i$ and $\Delta_j$. We conclude that $R=S_i\cup S_j$.
\end{proof}

 Proposition \ref{prop: 2 deltas} implies that
  every radical subtower of $\T$ is of the form $\T_R$ where 
   \[
   R=\bigcup_{i\in S}S_i 
   \]
   for some $S\subseteq\{1,\ldots,m\}$. The biggest choice of $S$ would be $R$, though that might not be the simplest possible set. Using  Proposition \ref{prop: 2 deltas}, we can compute the poset of radical subtowers of a radical tower.

\begin{algorithm}[H]
\floatname{algorithm}{Algorithm}
    \caption{Algorithm for finding the poset of radical subtowers}\label{alg:2}
    \begin{itemize}[leftmargin=1.2em, itemsep=2pt]
        \item \textbf{Input:} A radical tower $\T$.
        \item \textbf{Output:} The poset $\mathrm{RadSubT }(\T)$ of subsets of $\{1,\ldots,m\}$ that yields a 
        radical subtower of $\T$.
        \begin{enumerate}[leftmargin=1em, itemsep=2pt]
            \item For every $i\in \{1,\ldots,m\}$, use Algorithm \ref{alg:1} applied to $\Delta_i$ to calculate $S_i$.
            \item Return the poset $\mathrm{RadSubT }(\T)$ whose elements are all the possible finite unions of $S_1,\ldots,S_m$.
        \end{enumerate}
    \end{itemize}
\end{algorithm}

 Now, Algorithm \ref{alg:3} can be rewritten as follows.
 
\begin{algorithm}[H]
    \caption{It finds the minimal radical subtower associated to a parametrization}\label{alg:4}
    \begin{itemize}[leftmargin=1.2em, itemsep=2pt]
        \item \textbf{Input:} A parametrization $(\T,\cP:=(p_1,\ldots,p_r)),$ with $\cP\in\F_m^r$.
        \item \textbf{Output:} The subset  associated to the minimal  subtower containing $\cP$.
         \begin{enumerate}[leftmargin=1em, itemsep=2pt]
            \item Apply Algorithm \ref{alg:2} to get  $\mathrm{RadSubT }(\T)$.
            \item  Compute the set $\mathscr{S}$ of $\Delta_i$ appearing in the representation of all $p_i$ w.r.t. $\Base{\T,\F_m}$
            \item Find in $\mathrm{RadSubT }(\T)$ the smallest set $A$ containing the subindexes associated to $\mathscr{S}$.
            \item Return the subtower corresponding to $A$.
        \end{enumerate}
    \end{itemize}
\end{algorithm}

\section{Dominant Radical Parametrizations}\label{sec:dominant}

In Section \ref{sec:subtower}, given a radical parametrization $(\T,\cP)$ with $\T$ as in \eqref{tower}, we have analyzed the radical  and tower variety of $\cP$ considering subtowers  $(\T',\cP)$ with $\T'\subseteq \T$, arriving at the notion of minimal subtower of a parametrization $(\mT,\cP)$. 

In this section, we consider  a more general scenario, namely, we are given $(\T,\cP)$
and $(\T^*,\cP)$ where the only connection between the two towers is that the last floor of one of them is included in the last floor of the other. In this setting, we study the relation between the radical and tower varieties of $(\T,\cP)$ and $(\T^*,\cP)$.

Let us introduce the terminology to be used in this section. The radical towers $\T$ and $\T^*$ are 
as in Diagram \ref{fig:torres}. We introduce $\widehat{E}_i,\widehat{E}_{i}^{*}, E_i,{E}_{i}^{*}$ being defined as in \eqref{eq-Etilde}, \eqref{eq-pols}.
 Also (see \eqref{eq-E}), we consider the ideals
$\E= \langle E_1,\ldots,E_m \rangle\subset \C[\oT,\oD]$, $\E^*= \langle E_{1}^*,\ldots,E_{k}^* \rangle\subset \C[\oT,\Delta_{1}^{*},\ldots,\Delta_{k}^{*}]$.

\begin{diagram}[ht]
\begin{equation}
\begin{array}{c@{\;}c@{\;}c@{\;}c@{\;}c@{\;}c@{\;}c}
 & \raisebox{-10pt}[0pt][0pt]{\upsubsetneq} & \mathbb{F}_1=\mathbb{F}_0(\delta_1) & \subseteq\cdots\subseteq
   & \mathbb{F}_i=\mathbb{F}_{i-1}(\delta_i) & \subseteq\cdots\subseteq
   & \mathbb{F}_m=\mathbb{F}_{m-1}(\delta_m)  \Big{]} =\T \\[4pt]
\Big{[}\ \mathbb{F}_0 & & & & \versubseteq & & \\[4pt]
 & \raisebox{10pt}[0pt][0pt]{\downsubsetneq} & \mathbb{F}_1^{*}=\mathbb{F}_0(\delta_1^{*}) & \subseteq\cdots\subseteq
   & \mathbb{F}_k^{*}=\mathbb{F}_{k-1}^{*}(\delta_k^{*}) & \Big{]}=\T^{*}&
\end{array}
\end{equation}

 \noindent where 
\begin{equation*}\label{eq-delta,delta*} 
\begin{array}{l}
\text{$\F_i=\F_{i-1}(\delta_i)$, $\delta_i^{e_i} = \alpha_i \in \F_{i-1}$,   $e_i \in \mathbb{Z}_{>1}$, is a field extension of $\F_{i-1}$ with $[\F_i:\F_{i-1}]=e_i$.}  \\
\text{$\F_{i}^{*}=\F_{i-1}^{*}(\delta_{i}^{*})$, 
$(\delta_{i}^{*})^{e_{i}^{*}} = \alpha_{i}^{*} \in \F_{i-1}^{*}$, 
 $e_{i}^{*} \in \mathbb{Z}_{>1}$, is a field extension of $\F_{i-1}^{*}$ with  $[\F_{i}^*:\F_{i-1}^*]=e_{i}^*$.
}
\end{array}
\end{equation*}
\caption{$\T^*\preceq  \T$}
\label{fig:torres}
\end{diagram}

\begin{definition}\label{def-dominancia}
 Let $\T$ and $\T^*$ be two radical towers as in  Diagram \ref{fig:torres},  with $\F_k^{*}\subseteq \F_m$. We say that the tower $\T$ \textsf{dominates}  $\T^*$ and
 we write $\T^* \preceq  \T$. We say that  $(\T,\cP)$ \textsf{dominates} $ (\T^*,\cP)$, and we write $(\T^*,\cP)\preceq (\T,\cP)$,  
 if $\T^* \preceq \T$ and $\cP\in \F_{k}^{*r}\subseteq \F_{m}^{r}$.
\end{definition}

In Def.~\ref{def-dominancia},  
the inclusion
$\mathbb{F}^*_k\subseteq\mathbb{F}_m$ is understood through a field embedding
$\gamma:\mathbb{F}^*_k\hookrightarrow\mathbb{F}_m$ whose restriction to $\mathbb{F}_0$
is the identity. This embedding does not need to be unique.  Throughout this section we assume that $\gamma$ is fixed.

 \begin{remark}
 We observe that if $\T^*\subseteq \T$ then $\T^*\preceq \T$, but the converse does not hold.
 \end{remark}

In this situation, we study the relationship among the varieties $\cV_\T$, $\cV_{\T^*}$, $\cV_\cP$, and $ \cV_{\cP}^{*}$, where $\cV_{\T^*}$ and $\cV_{\cP}^{*}$ are the tower and radical varieties of the tower $\T^{*}$

 \begin{theorem}\label{theo:initial maps}
Let  $\T^*\preceq \T$. There exists a unique dominant rational map
  \[\begin{array}{ccc}
  \iota:\cV_{\T} \longrightarrow\cV_{\T^*}
  \end{array}
  \]
  such that for every radical parametrization 
$\cP\in \F_{k}^{* r}$ 
it holds that $\cV_\cP=\cV_{\cP}^{*}$, and
  the following diagram is commutative.

  \begin{equation}\label{eq:diag univ}
\begin{tikzcd}
                    &  & \cV_\T \arrow[d, "\iota"] \arrow[lld, "\cR"'] \arrow[dd, "\widehat{\pi}", bend left, shift left=3] \\
\cV_\cP=\cV_\cP^{*} &  & \cV_{\T^{*}} \arrow[ll, "\cR^{*}"] \arrow[d, "\widehat{\pi}^{*}"']                                 \\
                    &  & \C^n \arrow[llu, "\cP", bend left]                                                                
\end{tikzcd}
  \end{equation}
 Here  $\cR:\cV_\T\rightarrow\cV_\cP$, $\cR^*:\cV_{\T^*}\rightarrow\cV_{\cP}^{*}$, $\widehat{\pi}:\cV_\T\rightarrow\C^n$, and $\widehat{\pi}^{*}:\cV_{\T^*}\rightarrow\C^n$ are as in Diagram \ref{fig:diagram2}.

\end{theorem}
\begin{proof}
We can write the fields $\F_{k}^{*}$ and $\F_m$ as  (see \eqref{eq:field desc})
\[\begin{array}{ccc}
\F_{k}^{*}=\F_0[\Delta_{1}^*,\ldots,\Delta_{k}^*]/\E^*& \text{ and }&\F_m=\F_0[\Delta_1,\ldots,\Delta_m]/\E.
\end{array}
\]
We represent the inclusion of field extensions of {$\F_k^*/\F_0$ in $\F_m/\F_0$} as the injective map 
\[
\gamma:\F_k^*=\F_0[\Delta_{1}^{*},\ldots,\Delta_{k}^{*}]/\E^*
\hookrightarrow  \F_m=\F_0[\Delta_1,\ldots,\Delta_m]/\E
\]
 whose restriction to $\F_0$ is the identity. Let $\gamma(\Delta_i^{*})=g_i(\overline{\Delta})$ where $g_i\in\F_0[\overline{\Delta}]$. In particular, $g_i$ may be written as
 \begin{equation}\label{eq:dominan map gen}
 g_i=\frac{g_{iN}(\overline{T},\overline{\Delta})}{g_{iD}(T)} \ \text{ for }g_{iD}\in\C[\overline{T}] \ \text{ and }\ g_{iN}\in\C[\overline{T},\overline{\Delta}].
 \end{equation}
 Let $\alpha_{\mathrm{D}}=\alpha_{1,\mathrm{D}}(\oT)\cdots \alpha_{m,\mathrm{D}}(\oT)\alpha_{1,\mathrm{D}}^{*}(\oT)\cdots \alpha_{k,\mathrm{D}}^{*}(\oT)$ and let $U'$ be the Zariski-open subset defined by $\alpha_{\mathrm{D}}\neq 0$. In other words, $U'=U_\T\cap U_{\T^{*}}$ where $U_\T $ is as in Section \ref{sec:refinamient}.
Then, $\gamma$ after localization, restricts  to the ring injection
 $$
 \gamma: \left(\C[\oT][\Delta_{1}^{*},\ldots,\Delta_{k}^{*}]/\E^*\right)_{\alpha_{\mathrm{D}}}\hookrightarrow (\C[\oT][\Delta_{1},\ldots,\Delta_{m}]/\E)_{g_D\cdot\alpha_{\mathrm{D}} },
 $$ 
 where $g_D=g_{1D} g_{2D}\cdots g_{kD}$ and $(\C[\oT][\Delta_{1},\ldots,\Delta_{m}]/\E)_{g_D \cdot \alpha_{\mathrm{D}}}$ is the localization of $\C[\oT][\Delta_{1},\ldots,\Delta_{m}]/\E$ by $g_D\cdot \alpha_{\mathrm{D}}$ (see \eqref{eq:dominan map gen}).
By  \cite[Ch. I, Prop. 3.5 and Cor. 3.7]{Hartshorne}, this map yields a dominant morphism $$\iota:\V_\T\cap U''\to\V_{\T^*}\cap U',$$ where $U''=U'\cap\{g_{\mathrm D}(\oT)\neq0\}$.  Note that $\cV_{\T}\cap U''$ and $\cV_{\T^*}\cap U'$ are Zariski-dense in $\cV_{\T}$ and $\cV_{\T^*}$ respectively. 
Diagram \eqref{eq:diag univ} commutes since in terms of the rings, $\cR^*\circ \iota$ writes $p_1,\ldots,p_r$ of $\cR^{*}$ as elements of $\F_m$ using the polynomials \eqref{eq:dominan map gen}. 
Moreover, $\widehat{\pi}=\widehat{\pi}^{*}\circ\iota$ as $\gamma(T_i)=T_i$ for every $i=1,\ldots,n$. 
As in Theorem \ref{thm:subtower var}, using the commutativity of  Diagram \eqref{eq:diag univ}, the fact that $\iota $ is dominant and  that $\cV_\cP$  is irreducible,  one gets   that 
\[\cV_\cP=\overline{\cR(\cV_{\T})}=\overline{(\cR^*\circ\iota)(\cV_{\T})}=\overline{ \cR^*(\cV_{\T^*})}=\cV_{\cP}^*.\]

Now, assume that there exists $\iota':\cV_\T\dashrightarrow\cV_{\T^*}$ with the same property as $\iota$.  Let $\cP_\cU$ be the universal radical parametrization of $\T^*$. Applying the property of $\iota'$ to $\cP_\cU$, the following diagram commutes
  \[
  \begin{tikzcd}
\cV_{\T} \arrow[d, "\iota'"'] \arrow[rr, "\cR"] &  & \cV_{\cP_\cU}=\cV_{\cP_\cU}^{*} \\
\cV_{\T^*} \arrow[rru, "\cR^*"']             &  &             
\end{tikzcd}
  \]
By Theorem \ref{theo: univ para var}, $\cV_{\cP_\cU}=\cV_{\T^*}$ and $\cR=\iota'$. Now, in terms of the coordinates, $\cR$ is the map given by $(t_1,\ldots,t_n,\delta_{1}^{*},\ldots,\delta_{k}^{*})$. Therefore, in terms of the rings, $\cR$ is given by the map $\gamma$. We conclude that $\iota $ and $\iota'$ are equal as rational maps. 
\end{proof}

One can generalize Theorem \ref{theo:initial maps} involving the incidence variety as follows.

\begin{corollary}\label{co: univ maps}
      Let  $\T^*\preceq \T$ as in Diagram \ref{fig:torres}, and let $\iota$ be the same map as in 
    Theorem \ref{theo:initial maps}. Then,  for every radical parametrization 
$\cP\in \F_{k}^{* r}$
    there exists a unique dominant rational map $\eta:\cA\rightarrow\cA^*$ such that the following diagram commutes
\[
\begin{tikzcd}
\cA \arrow[dd, "\pi^{*}"'] \arrow[rd, "\pi"] \arrow[rr, "\eta", ] &                  & \cA^* \arrow[dd, "(\pi^{*})^*"] \arrow[ld, "(\pi)^*"'] \\
                                                                        & \cV_\cP=\cV_\cP^* &                                                   \\
\cV_\T \arrow[ru, "\cR"'] \arrow[rr, "\iota"']                  &                  & \cV_{\T^*} \arrow[lu, "\cR^*"]                     
\end{tikzcd}
    \]
    where $(\pi)^*$ and $(\pi^{*})^{*}$ are the analogous maps to $\pi$ and $\pi^{*}$ in  Diagram \ref{fig:diagram2} but for the radical tower $\T^*$.
\end{corollary}
\begin{proof}
    The proof follows from Theorem \ref{theo:initial maps} and the birationality of $\pi^{*}$ and $(\pi^{*})^{*}$.
\end{proof}

\begin{remark}
As shown in the proof of Theorem \ref{theo:initial maps}, the dominant map $\iota$ does not depend on the parametrization and it can be explicitly constructed by the polynomials $g_i$ in \eqref{eq:dominan map gen}. These polynomials are obtained by expressing the radicals $\Delta_1^*,\ldots,\Delta_k^*$ of the tower $\T^*$ by means of the radicals $\Delta_1,\ldots,\Delta_m$ of the tower $\T$. 
\end{remark}

\begin{example}\label{ex: non iso tower var 1}
Take the two distinct towers $\T:=[\F_0\subsetneq\F_1=\F_0(\sqrt{T})]$ and $\T^*:=[ \F_0\subsetneq\F_1^*=\F_0(\frac{1}{\sqrt{T}})]$. 
Note that $\F_1=\F_1^*$, but they are constructed differently as
\[
\F_1=\F_0[\Delta]/\langle\Delta^2-T\rangle\ \text{ and } \ 
\F_1^*=\F_0[\Delta^*]/\langle(\Delta^{*})^2-\frac{1}{T}\rangle.
\]
The coordinate rings of the tower variety of $\T$ and $\T^*$ are
\[
\C[T,\Delta]/\langle \Delta^2-T\rangle \ \text{ and } \ \C[T,\Delta^*]/\langle T(\Delta^*)^2-1\rangle,
\]
respectively. These two rings are non-isomorphic. However, the tower varieties of each tower are birational. Such birational map is given by 
\[
\begin{array}{ccc}
    \cV_\T &\longleftrightarrow& \cV_{\T^*}  \\
     (T,\Delta)& \mapsto &(T,\frac{\Delta}{T})\\
     (T,T\Delta^*)& \leftmapsto & (T,\Delta^*)   
\end{array}\ 
\]
Note that both maps are birational.
\end{example}

Example \ref{ex: non iso tower var 1} shows that two codominated towers (towers with the same last field) may lead to distinct tower varieties. The pathology in this case arises from allowing denominators in the construction of the tower. In other words, there exists an $\alpha_i$ with a denominator in $\overline{T}$. A natural question is what happens when no $\alpha_i$ has denominators, i.e. $\alpha_i\in\C[\overline{T},\overline\Delta]$ for every $i$. The following example shows that even in this case, the tower varieties of codominated tower might not be isomorphic.

\begin{example}
Consider the towers $\T:=[\F_0\subsetneq\F_1=\F_0(\sqrt{T})]$ and $\T^*:=[ \F_0\subsetneq\F_1^*=\F_0(\sqrt{T^3})]$. 
Note that the two fields $\F_1$ and $\F_1^*$ are the same, but they are constructed differently as
\[
\F_1=\F_0[\Delta]/\langle\Delta^2-T\rangle\ \text{ and } \ 
\F_1^*=\F_0[\Delta^*]/\langle(\Delta^*)^2-T^3\rangle.
\]
The coordinate ring of the tower varieties of $\T$ and $\T^*$ are
\[
\C[T,\Delta]/\langle \Delta^2-T\rangle \ \text{ and } \ \C[T,\Delta^*]/\langle (\Delta^*)^2-T^3\rangle,
\]
respectively. These two rings are non-isomorphic. However, the tower varieties of each tower are birational. Such birational map is given by 
\[
\begin{array}{ccc}
    {\cV}_\T &\dashrightarrow& {\cV}_{\T^*}  \\
     (T,\Delta)& \mapsto &(T,T\Delta)\\
     (T,\frac{\Delta^*}{T})& \leftmapsto & (T,\Delta^*)
     
\end{array}\ 
\]
\end{example}

\begin{remark}
    As shown in these examples, there is no  canonical choice of tower variety among distinct but very similar radical towers with same last field. Since the fraction field of all these possible tower varieties is the same (that last field), they are all birational. However, thanks to Theorem \ref{theo:initial maps} such birationality is slightly stronger. Let $\T$ and $\T^*$ be codominated towers. Then,  there exist unique dominant maps $\iota:\cV_{\T}\dashrightarrow\cV_{\T^*}$ and $\iota^*:\cV_{\T^*}\dashrightarrow\cV_{\T}$ as in Theorem \ref{theo:initial maps}. By uniqueness, $\iota\circ\iota^*=\mathrm{Id}$ and vice versa. This birational map restricted to the variables $\overline{T}$ is the identity. Therefore, the relation between these two distinct tower varieties is stronger than being just birational.
\end{remark}

\begin{corollary}\label{co: unirat}
Let $\T$ be a radical tower such that $\cV_\T$ is unirational. Then,  $\cV_{\T^*}$ is unirational for every $\T^*\preceq \T$. Moreover,  $\cV_\cP$ is unirational for any parametrization $\cP\in\F_m^r$.
\end{corollary}
\begin{proof}
    The proof follows from Diagram \ref{eq:diag univ} and the dominant map $\iota$.
\end{proof}

\begin{remark} Note that if the dimension of $\cV_\T$ is $\leq 2$, by  L\"uroth's theorem and Castelnuovo's theorem, in Corollary \ref{co: unirat} we may replace unirationality by rationality. 
   Example \ref{ex:nonluroth} shows that this is no longer true for $n\geq 3$, as Castelnuovo's Theorem does not hold.
\end{remark}

The following example exhibits a radical variety that is unirational, not rational, but the tower variety is rational.
 
\begin{example}\label{ex:nonluroth}
Consider the radical tower
\[
\T:=\left[ \begin{array}{c}
\F_0=\C(T_1,T_2,T_3)\subsetneq\F_1=\F_0[\Delta_1]/\langle \Delta_1^2-T_1\rangle
\subsetneq
\\
\F_2=\F_1[\Delta_2]/\langle \Delta_2^3-T_2\rangle
\subsetneq\F_3=\F_2[\Delta_3]/\langle \Delta_3^2-T_3\rangle
\end{array}
\right].
\]
The tower variety is 
\[
\cV_\T=\V(  \Delta_1^2-T_1,\Delta_2^3-T_2,\Delta_3^2-T_3)\subset\C^6,
\]
which is rational and can be parametrized by 
\[
\mathscr{T}:\C^3\longrightarrow\cT,\ \,(s_1,s_2,s_3)\longmapsto(s_1^2,s_2^3,s_3^2,s_1,s_2,s_3). 
\]
Now, consider the radical parametrization $\cP(\oT,\overline{\Delta})=(p_1,p_2,p_3,p_4)\in\F_3^4$ given by 
\[
\begin{array}{cc}
  p_1=\dfrac{\Delta_1(-1+T_3(1+T_2-T_1^3))}{1+T_3(1+T_2-T_1^3)}, &   p_2=  \dfrac{2\Delta_1\Delta_3T_1}{1+T_3(1+T_2-T_1^3)},
\\
\noalign{\smallskip}
  p_3 = \dfrac{-2\Delta_1\Delta_3}{1+T_3(1+T_2-T_1^3)},
 & 
p_4=\dfrac{-2\Delta_1\Delta_2\Delta_3}{1+T_3(1+T_2-T_1^3)}.
\end{array}
\]
that corresponds to  the radical variety $\cV_\cP=\V(x_2+x_1^2x_3+x_2^3+x_3^3+x_4^3)$ in $\C^4$. Using Remark \ref{Remark-first} (\ref{Rem2}), we have that 
\[ \cR\circ\mathscr{T}= \left(
\dfrac{s_{1} \left(1+\left(s_{1}^{6}-s_{2}^{3}-1\right) s_{3}^{2}\right)}{-1+\left(s_{1}^{6}-s_{2}^{3}-1\right) s_{3}^{2}}
, 
-\dfrac{2 s_{1}^{3} s_{3}}{-1+\left(s_{1}^{6}-s_{2}^{3}-1\right) s_{3}^{2}}
, 
 \right. \]
\[\left. \dfrac{2 s_{1} s_{3}}{-1+\left(s_{1}^{6}-s_{2}^{3}-1\right) s_{3}^{2}}
, 
\dfrac{2 s_{1} s_{2} s_{3}}{-1+\left(s_{1}^{6}-s_{2}^{3}-1\right) s_{3}^{2}}
\right)
 \]
parametrizes $\cV_\cP$. The degree of the map induced by the rational parametrization is 2, and hence $\cV_\cP$ is unirational.

The homogenization of $\cV_\cP$ yields  a smooth cubic threefold in $\P^4$. By \cite[Theorem 0.12]{CG72}, smooth cubic threefolds  in $\P^4$ are not rational and, hence,  $\cV_\cP$ is not rational. By Corollary \ref{co: unirat}, $\cV_\T$ is rational and  every radical variety of a radical parametrization in $\F_3^r$ is unirational. However, the radical variety might not be rational, as shown in the example.
\end{example}

\begin{definition}\label{def-strong}
    We say that a parametrization $(\T,\cP)$ is \textsf{strong} if the fraction field of $\cV_\cP$ is $\F_m$.
\end{definition}

\begin{proposition}
    A strong parametrization is minimal.  
\end{proposition}
\begin{proof}
Let $\F_k^{*}$ be the last field of the minimal radical tower $\T^{*}$ of $\cP$ and let $\K$ be the fraction field of $\cV_\cP$. Using that every map in Diagram \ref{eq:diag univ} is dominant, we get the inclusions $\K\subseteq\F_k^*\subseteq\F_m$. Since $\cP$ is  strong, we conclude that $\K=\F_k^*=\F_m$ and that $\T^*=\T$.
 \end{proof}

\begin{proposition}
    The universal parametrization is strong.
\end{proposition}
\begin{proof}
It follows from Theorem \ref{theo: univ para var}.
\end{proof}

\begin{corollary}\label{cor-fuerte}
   Let $\cP\in\F_m^r$ be a strong parametrization. Then, $\cV_\cP$ is rational (resp. unirational) if and only if $\cV_\T$ is rational (resp. unirational).
\end{corollary}

The next example shows a minimal parametrization that is not strong. 

 \begin{example}\label{ex-minimal-strong}
Consider the radical tower
\[
\mathbb{T}:=\bigl[\mathbb{F}_0=\mathbb{C}(T_1,T_2)\subsetneq\mathbb{F}_1=\mathbb{F}_0[\Delta_1]/\langle\Delta_1^2-T_1\rangle\subsetneq\mathbb{F}_2=\mathbb{F}_1[\Delta_2]/\langle\Delta_2^2-T_2^3+T_2\rangle\bigr].
\]
The tower variety is $\mathscr{V}_{\mathbb{T}}=\mathbb{V}(\Delta_1^2-T_1,\Delta_2^2-T_2^3+T_2)\subset\mathbb{C}^4$,
which is not rational, not even unirational. Indeed, the projection $(T_1,T_2,\Delta_1,\Delta_2)\mapsto(T_2,\Delta_2)$ maps $\cV_\T$
dominantly onto the elliptic curve $\cV=\mathbb V(\Delta_2^2-T_2^3+T_2)$, 
so if $\cV_\T$ were unirational,
$\cV$ would be unirational 
and hence rational by L\"uroth's theorem, which is impossible since 
$\cV$ has genus one.

Now, consider the minimal radical parametrization $(\mathbb{T},\mathscr{P})$ with
\[
\mathscr{P}=(\Delta_1+T_1^3+T_1,\ \Delta_2\Delta_1-\Delta_2-1,\ \Delta_1)\in\mathbb{F}_2^3.
\]
 Consider the projection $\kappa:\mathscr{V}_{\mathbb{T}}\to\mathbb{C}^3$ given by $(T_1,T_2,\Delta_1,\Delta_2)\mapsto(T_1,\Delta_1,\Delta_2)$, and let $(y_1,y_2,y_3)$ be the coordinates of $\mathbb{C}^3$. Then $\overline{\kappa(\mathscr{V}_{\mathbb{T}})}=\mathbb{V}(y_2^2-y_1)$, which is a rational surface parametrized by $\mathscr{Q}=(a^2,a,b)$. Moreover, we have the commutative diagram
\[
\begin{tikzcd}
\mathscr{V}_{\mathscr{P}} & \mathscr{V}_{\mathbb{T}} \arrow[l, "\mathscr{R}"'] \arrow[d, "\kappa"] \\
 & \kappa(\mathscr{V}_{\mathbb{T}}) \arrow[ul, "\mathscr{R}'"]
\end{tikzcd}
\]
where $\mathscr{R}'$ is the map given by $(y_1,y_2,y_3)\mapsto(y_2+y_1^3+y_1,\ y_3y_2-y_3-1,\ y_2)$. Composing $\mathscr{R}'$ with $\mathscr{Q}$, we obtain
\[
(a,b)\longmapsto\bigl(a+a^2+a^6,\ ab-b-1,\ a\bigr),
\]
which is birational onto its image. Hence $\mathscr{V}_{\mathscr{P}}=\mathbb{V}(x_1-x_3^6-x_3^2-x_3)\subset\mathbb{C}^3$ is rational. Since $\mathscr{V}_{\mathscr{P}}$ is rational and $\mathscr{V}_{\mathbb{T}}$ is not, Corollary \ref{cor-fuerte} shows that $\mathscr{P}$ is not strong. Therefore, $\mathscr{P}$ is a minimal parametrization but not strong.
\end{example}

\section{Radical parametrizations over subfields}\label{sec:subsfields}

Throughout this section, let $\L$ be a subfield of $\C$. Let  $\T$ be as in \eqref{tower}, $\widehat{E}_i$ as in \eqref{eq-Etilde},    $E_i$  as in \eqref{eq-ET} and $\E$ as in \eqref{eq-E}.

\begin{definition}
 We say that the tower
 \[
  \T_\L:= \left[ \F_{0,\L}:=\L(\overline{T})\subsetneq\F_{1,\L}:=\F_{0,\L}[\Delta_1]/\langle E_1\rangle\subsetneq \cdots \subsetneq\F_{m,\L}:=\F_{0,\L}[\Delta_1,\ldots,\Delta_m]/\langle E_1,\ldots,E_m\rangle\right],
    \]
 is  \textsf{$\L$--defined} if    $E_i\in \L[\oT,\oD]$ and irreducible in $\F_{i-1}[\Delta_i]$ for all $i\in \{1,\ldots,m\}$. To emphasize this, we will denote the $\L$--defined tower as $\T_{\L}$. Additionally, a parametrization in  $\F_{m,\L}^r$ is called an $\L$-defined radical parametrization and it is denoted by $\cP_{\L}$. 
\end{definition}

We may see the tower variety $\cV_{\T}$ as a variety in $\A^{n+m}_\C$ or in $\A^{n+m}_\L$. In this setting, we denote by $\cV_{\T,\L}$  the  tower variety of $\T$ as a variety over the field $\L$. Similarly, we introduce the radical variety $\cV_{\cP,\L}$.

\begin{notation}
    Up to this section, the base field of all the varieties was $\C$ and there was no need of distinguishing between the variety as sets of complex solutions and as the spectrum of its coordinate ring, denoted by $\mathrm{Spec}(R)$ for a ring $R$. When working with non-algebraically closed fields such as $\R$ or $\L$, that distinction is crucial. In the rest of the section, a variety (as $\cV_{\T,\L}$ or $\cV_{\cP,\L}$) will represent the spectrum of its coordinate ring; see e.g. \cite{Hartshorne} Ch. II,  Sect 2. The set of solutions of a variety $X$ in the field $\L$ will be denoted by $X(\L)$.  {For instance, under the assumptions of Corollary \ref{co: no denom}, the notation is
    \[
    \cV_{\T,\L}=\mathrm{Spec}\left(
    \L[\overline{T},\overline{\Delta}]/\mathcal{E}
    \right) \ \text{ and } \ \cV_{\T,\L}(\L)=
    \left\{
    (\overline{t}_0,\overline{\delta}_0)\in\L^{n+m}: (\overline{t}_0,\overline{\delta}_0)\,\text{ vanishes at }\, \mathcal{E}
    \right\}.
    \]}
\end{notation}

 {Using this notation, we derive the following result.}

 \begin{proposition}\label{prop:tower subfield}
 Let $(\T_\L,\cP_\L)$ be an $\L$-defined radical parametrization. It holds that
\begin{enumerate}
\item $\cV_{\T,\L}$ is irreducible of dimension $n$.
\item    
    $\cV_{\cP,\L}
  $   is  irreducible of dimension $n$.
  \item The saturation 
  \[
  \E:\langle \alpha_{1,\mathrm{D}}(\oT)\cdots\alpha_{m,\mathrm{D}}(\oT)\rangle^\infty
  \]
  is the ideal of $\cV_{\T,\L}$.
  \item If $\alpha_{i,\mathrm{D}}=1$, then $\E$ is the ideal of $\cV_{\T,\L}$.
\end{enumerate}
\end{proposition}
\begin{proof}
    We have the following Cartesian diagram (see e.g. \cite{Hartshorne} Ch. II,  Sect 2, for the notion of spectrum of a field)
    \[
\begin{tikzcd}
\cV_{\T} \arrow[r] \arrow[d] & {\cV_{\T,\L}} \arrow[d] \\
\mathrm{Spec}(\C) \arrow[r]          & \mathrm{Spec}(\L)              
\end{tikzcd}.
    \]
    Since the lower horizontal arrow is faithfully flat, the above horizontal arrow is faithfully flat too. In particular, $\dim(\cV_{\T,\L})=\dim( \cV_{\T}) $. Moreover, since the above horizontal arrow is surjective and $\cV_{\T}$ is irreducible, $\cV_{\T,\L}$ is also irreducible. Now, $\cV_{\cP,\L}
  $ is irreducible of dimension $n$ as it is the closure of $\cR(\cV_{\T,\L})$, $\cR$ is generically finite to one and $\cV_{\T,\L}$ is irreducible of dimension $n$.

  Statements (3) and (4) follow from the fact that the proofs of Proposition \ref{prop- new E} and Corollary \ref{co: no denom} hold also for $\L$ instead of $\C$.
\end{proof}

\subsection{The real case}

Now we focus on the case where $\L=\R$. Let $\T_\R$ be the $\R$--defined tower
\[
\T_\R :=\left[\ \F_{0,\R}:=\R(\overline{T})\subsetneq \F_{1,\R}:=\F_{0,\R}[\Delta_1]/\langle E_1\rangle\subsetneq\cdots\subsetneqq \F_{m,\R}:=\F_{0,\R}[\Delta_1,\ldots,\Delta_m]/\langle E_1,\ldots,E_m\rangle \right].
\]
 A distinction with respect to the complex case is that the field extensions of the tower might not be normal as the base field $\F_{0,\R}$ might not contain the required roots of the unity.
When $e_i$ is odd, the only solution to the equation $Y^{e_i}-\alpha_i$ in $\F_{i,\R}$ is $\Delta_i$. Whereas for $e_i$ even, the only solutions to the equation $Y^{e_i}-\alpha_i$ in $\F_{i,\R}$ are $\pm \Delta_i$.
As a consequence, the choice of which radical is used  is important in this case.

Another remark required in the real setting is the distinction between the complex radical and the real radical. In the complex setting, there are $e_i$ choices of the $e_i$--th complex radical described in \eqref{eq-raices-unidad}.
The problem is that for $e_i$ odd, none of the restrictions to $\R$ of the $e_i$ complex radicals coincide with the real radical function $\sqrt[e_i]{x}$. 

Throughout this section we consider that the parametrizations are expressed in canonical form.
An $\R$-defined parametrization $\cP\in\F_{m,\R}^r$ represents a real--valued function (maybe with empty domain) $\cP:\R^n\dashrightarrow\R^r$ given by the real radicals (not by the restriction of the complex radicals to the real numbers). We denote by $\mathrm{Dom}_\R(\cP)$ the domain of the $\R$-defined parametrization $\cP$. For the rest of this section, by convention, for the case when $e_i$ is even, the real radical $\sqrt[e_i]{\phantom{x}}$ represents the non-negative real radical. 

Let $X$ be an irreducible variety over $\C$ of dimension $\ell$. We say that $X(\R)$ is a real variety if $X$ contains a nonempty  Euclidean 
open set of $\R^\ell$.

Let $X$ be a variety over $\C$. In general, it is not true that if $X(\C)$ is irreducible then $X(\L)$ is irreducible. Here the topology of $X(\L)$ is the Zariski topology at the level of the sets of solutions. That is, the closed subsets are the zero sets of solution over $\L$ of a variety. Nevertheless, if $X$ is irreducible over $\C$ and has a simple real point, then $X(\R)$ is Zariski-dense in $X$ by \cite[Prop. 3.3.16]{BCR98} or \cite[Theorem 5.1]{sottile2018real}. 
The statement also holds for every field $\L$ that is
\emph{large} in the sense of Pop \cite{Pop96}, for example any real closed
field.
So Proposition \ref{prop:tower subfield} can be extended as
follows:

\begin{proposition}\label{prop:tower subfield-real} Let $(\T_\R,\cP_\R)$ be an $\R$-defined radical parametrization. If $\cV_{\T,\R}(\R)$ has a real simple point, it holds that
\begin{enumerate}\item   $\cV_{\T,\R}(\R)$ is irreducible and real.
     \item    
    $\cV_{\cP,\R}(\R) 
  $  is irreducible and real.
\end{enumerate}
  
\end{proposition}

We note that our colleagues, J. Caravantes, D. Sevilla and C. Villarino, have independently obtained a similar result to Prop. \ref{prop:tower subfield-real} in the one-dimensional case (private communication).

\begin{remark}
Note that, under the conditions of Proposition \ref{prop:tower subfield-real}, $\dim(\cV_{\T,\R})=\dim(\cV_{\T,\R}(\R))=\dim(\cV_{\cP,\R}(\R))$.
\end{remark}

 In this setting, we introduce the next notion.

 \begin{definition}
     The \textsf{$\R$-defined   universal radical parametrization} of an $\R$-defined  tower $\T_\R$ is the parametrization (see Section \ref{sec:universal-param})
\[
\cPR:=(\overline{T},\overline{\Delta})\in\F_{m,\R}^{n+m}.
\]
\end{definition}
As a real radical map, $\cPR$ represents the map 
\[\begin{array}{cccl}
\cPR:&\R^{n}&\dashrightarrow&\cV_{\T,\R}\subset\R^{n+m}\\
 & \overline{t}&\mapsto & (\overline{t},\sqrt[e_1]{\alpha_1(\overline{t})},\ldots, \sqrt[e_m]{\alpha_m(\overline{t})}).
\end{array}
\]
Recall that here $\sqrt[e_i]{\phantom{x}}$ represents the real  radical (positive if $e_i$ is even). In particular, $\cPR$ might not coincide with the restriction of $\cPU$ to $\R^n$. For $e_i$ even, if one wants to work with the negative real radical $-\sqrt[e_i]{{\phantom{x}}}$, one should change the tower with the change of variables $\Delta_i'=-\Delta_i$ as shown in Example \ref{ex: distinct roots}.
A problem arising from this type of parametrization, while working with $\R$, is that depending on the choice of $\cP$, the radical parametrization might have empty domain.

\begin{example}\label{ex-pat1}
We consider the $\R$--defined parametrization $(\T_\R, \cP_\R)$
\[ \T_\R=\left[\F_{0,\R}:=\R(T_1,T_2)\subsetneq \F_{0,\R}[\Delta]/ \langle \Delta^2+1+T_1^2+T_2^2\rangle \right], \,\,
\cP=\left( \Delta^3,T^2_1+T_2^2,T_2- \Delta\right).
    \]
The radical parametrization $\cP$, as function from $\R^2$ to $\R^3$, has empty domain.
    Consider the tower over the reals
    \[
    \F_{0,\R}:=\R(T_1,T_2)\subsetneq \F_{0,\R}[\Delta]/ \langle \Delta^2+1+T_1^2+T_2^2\rangle.
    \]
    Note that in this case,  $\cV_{\T,\R}(\R)=\emptyset$. The fact that  the set of  the real points of the tower variety is empty is deeply related with the emptiness of the domain of the parametrization.

In this example another pathology appears. The radical variety $\cV_{\cP,\R}(\R)$  is the cylinder of equation $X_{2}^{3}+X_{1}^{2}+3 X_{2}^{2}+3 X_{2}+1=0$ that is rational and can be parametrized as
  \[  \left(-\dfrac{H_{1}^{3}}{H_{1}^{3}+3 H_{1}^{2}+3 H_{1}+1}, 
-\dfrac{2 H_{1}^{2}+2 H_{1}+1}{H_{1}^{2}+2 H_{1}+1}, H_{2}\right) \]
    despite the radical parametrization having empty domain.
\end{example}

Example \ref{ex-pat1} exhibits the deep relation between the real algebraic geometry of the tower variety and the domain of the $\R$--defined  radical parametrization.
To describe such relation, we introduce the following notation. Given a real tower $\T_\R$, we consider the sets
\[  
M_i=\left\{
\begin{array}{ll}
\{1\}& \text{ if } e_i \text{ is odd},\\
\{\pm 1\} & \text{ if } e_i \text{ is even},
\end{array}
\right.
\]
for $1\leq i\leq m$.
Let $M:=M_1\times\cdots\times M_m$. Given a tuple $v=(v_1,\ldots,v_m)\in M$, we consider the ring map 
\[\begin{array}{cccl}
\lambda_v:&\F_{0,\R}[\oD]&\longrightarrow &\F_{0,\R}[\oD_v]\\
 & \Delta_i&\longmapsto& v_i\cdot\Delta_{v,i}.
\end{array}
\]
Here $\overline{\Delta}_v=(\Delta_{v,1},\ldots,\Delta_{v,m})$.  
In other words, $\lambda_v$ is the linear isomorphism given by sign change corresponding to the tuple $v$. By construction, such sign change could only possibly change the sign of $\Delta_i$ for $e_i$ even. 
For $v\in M$, we define the tower $\T_\R^{(v)}$ as the tower 
whose $i$-th floor $\F_{i,\R}^{(v)}$ is
\begin{equation}\label{eq:R associated tower}
\F_{i,\R}^{(v)} = \F_{i-1,\R}^{(v)}[\Delta_{v,i}]/\langle E_i^{(v)}\rangle \ \text{ where } \ E_i^{(v)}=\lambda_v(E_i),
\end{equation}
and $\F_{0,\R}^{(v)}=\F_{0,\R}$. Technically, $\lambda_v(E_i)\in\F_{0,\R}[\overline{\Delta}_v]$. However, by construction, $\lambda_v(E_i)$  is also contained in the subring $\F_{0,\R}[\Delta_{v,1},\ldots,\Delta_{v,i}]$. Therefore, we may see 
$E_i^{(v)}=\lambda_v(E_i)$
as an element of $\F_{i-1,\R}^{(v)}[\Delta_{v,i}]$
in \eqref{eq:R associated tower}. Under this construction, we introduce the following definition 

\begin{definition}
    We say $\T_\R^{(v)}$ is an \textsf{$\R$-associated tower to $\T_{\R}$}.
\end{definition}

\begin{example}\label{ex:associated towers}
Consider the tower
\[
\T_{\R}:=\left[ \begin{array}{c}
\F_{0,\R}=\R(T)\subset\F_{1,\R}=\F_{0,\R}[\Delta_1]/\langle\Delta_1^2-T\rangle
\subset\F_{2,\R}=\F_{1,\R}[\Delta_2]/\langle{\Delta_2^3-1-T}\rangle\\
\subset\F_{3,\R}=\F_{2,\R}[\Delta_3]/\langle\Delta_3^2-T-\Delta_1\rangle
\end{array}
\right].
\]
In this example, only $e_1=2$ and $e_3=2$ are even and we get that $M=\{\pm 1\}\times\{1\}\times\{\pm 1\}$. For $v=(-1,1,-1)$, the radical tower $\T_\R^{(v)}$ is given by 
\[
\T_{\R}^{(v)}:=\left[ \begin{array}{c}
\F_{0,\R}^{(v)}=\R(T)\subset\F_{1,\R}^{(v)}=\F_{0,\R}^{(v)}[\Delta_{v,1}]/\langle\Delta_{v,1}^2-T\rangle
\subset\F_{2,\R}^{(v)}=\F_{1,\R}^{(v)}[\Delta_{v,2}]/\langle\Delta_{v,2}^3-1-T\rangle\\
\subset\F_{3,\R}^{(v)}=\F_{2,\R}^{(v)}[\Delta_{v,3}]/\langle\Delta_{v,3}^2-T+\Delta_{v,1}\rangle
\end{array}
\right].
\]
Notice that the towers $\T_{\R}$ and $\T_{\R}^{(v)}$ differ mainly in the equation of the last floor of the tower. They are isomorphic via the change of variables
\[
\Delta_1\mapsto-\Delta_{v,1},\ \Delta_2\mapsto\Delta_{v,2}, \ \text{ and } \ \Delta_3\mapsto-\Delta_{v,3}.
\]
Another distinction between the two towers are their universal radical parametrizations, which are 
\[
\begin{array}{cc}
\begin{array}{cccl}
\cPR:&\R&\dashrightarrow&\cV_{\T,\R}\\
 & t& \mapsto&(t,\sqrt{t},\sqrt[3]{1+t},\sqrt{t+\sqrt{t}})
\end{array}
,
&
\begin{array}{cccl}
\cPR^{(v)}:&\R&\dashrightarrow&\cV_{\T^{(v)},\R}\\
 & t& \mapsto&(t,\sqrt{t},\sqrt[3]{1+t},\sqrt{t-\sqrt{t}})
\end{array}
\end{array}  
\]
\end{example}

Given $v\in M$ and the corresponding $\R$-associated tower $\T_\R^{(v)}$, we can construct the tower variety $\cV_{\T^{(v)},\R}$  of $\T_\R^{(v)}$. Then, the isomorphism
\[ \begin{array}{cccl}
\Lambda_v:&\R^{n+m}&\longrightarrow&\R^{n+m}\\
 & (\overline{T},\overline{\Delta}_v)&\longmapsto &
(\overline{T},v_1\Delta_{v,1},\ldots,v_m\Delta_{v,m})
\end{array}
\]
restricts to an isomorphism $\Lambda_v:\cV_{\T^{(v)},\R}\rightarrow\cV_{\T,\R}$.

\begin{definition}
    The \textsf{$\R$--associated universal parametrizations} to $\T_\R$ are the radical parametrizations
 $$\Lambda_v\circ\cPR^{(v)}:\R^n\dashrightarrow\cV_{\T,\R}$$ for $v\in M$, where  $\cPR^{(v)}$ is the $\R$--defined universal radical parametrization of  $\T_{\R}^{(v)}$. 
 
 More generally, let $\cP_\R$ be a radical parametrization and let $\cR$ be as in Diagram \ref{fig:diagram2}. The \textsf{$\R$-associated radical parametrizations to $\cP_\R$} are the radical parametrizations of the form $$\cR\circ\Lambda_v\circ \cPR^{(v)}:\R^{n}\dashrightarrow\cV_\cP$$ for $v\in M$. We write that $\cP_\R'\sim \cP_\R$ to indicate that $\cP'_\R$ is $\R$-associated  to $\cP_\R$.
\end{definition}

The following commutative diagram illustrates the definition of $\R$-associated radical parametrizations
\[
\begin{tikzcd}
{\cV_{\cP,\R}} & {\cV_{\T,\R}} \arrow[l, "\cR"']                                            & {\cV_{\T^{(v)},\R}} \arrow[l, "\Lambda_v"'] \\
               & \R^n \arrow[u, "\cP_v"'] \arrow[ru, "\cPR^{(v)}"'] \arrow[lu, "\cP_{\R}'"] &                                        
\end{tikzcd}.
\]

\begin{remark}
 Note that the $\R$-associated radical parametrizations of a radical parametrization $\cP_\R\in\F_{m,\R}^r$ might not be represented by an element in $\F_{m,\R}^r$ but of $(\F_{m,\R}^{(v)})^r$. This behavior is described in Example \ref{ex: distinct roots}. This example can be treated over $\R$ instead of over $\C$, and $\psi'$ is an $\R$-associated universal parametrization given by the vector $(-1,1)$. It is obtained by the change of variables $\Delta_1^{(v)}=-\Delta_1$ and $\Delta_2^{(v)}=\Delta_2$. However, as shown in the example such parametrization can not be represented as an element of $\F_2^3$.
\end{remark}

\begin{example}
    Following Example \ref{ex:associated towers}, we have four $\R$-associated universal radical parametrizations given by the vectors $(1,1,1),(1,1,-1),(-1,1,1),(-1,1,-1)$. These parametrizations are $\Lambda_v\circ\cPR^{(v)}$ for $v$ among these vectors. More explicitly, they are the maps from $\R$ to $\cV_{\T,\R}(\R)$ given by
    \[
    \begin{array}{rccl}
    \Lambda_{(1,1,1)}\circ \cPR^{(1,1,1)}:& t&\longmapsto&(t,\sqrt{t},\sqrt[3]{1+t},\sqrt{t+\sqrt{t}}),\\
    \Lambda_{(1,1,-1)}\circ \cPR^{(1,1,-1)}:& t&\longmapsto&(t,\sqrt{t},\sqrt[3]{1+t},-\sqrt{t+\sqrt{t}}),\\
\Lambda_{(-1,1,1)}\circ \cPR^{(-1,1,1)}:& t&\longmapsto&(t,-\sqrt{t},\sqrt[3]{1+t},\sqrt{t-\sqrt{t}}),\\
\Lambda_{(-1,1,-1)}\circ \cPR^{(-1,1,-1)}:& t&\longmapsto&(t,-\sqrt{t},\sqrt[3]{1+t},-\sqrt{t-\sqrt{t}}).
    \end{array}
    \]

\end{example}

The next result relates the real points of the tower variety and the domains of the $\R$-associated radical parametrizations of a given radical parametrization. We recall that $\widehat{\pi}:\R^{n+m}\rightarrow\R^n$ is the projection $(\overline{T},\overline{\Delta})\mapsto \overline{T}$ in Diagram \eqref{fig:diagram2}, and $A(\overline{T})$  is defined in \eqref{eq-pols} as the $\mathrm{l.c.m}$ of the denominators $\alpha_{i,\mathrm{D}}(\overline{T})$ and $p_{i,\mathrm{D}}(\overline{T})$. We denote the real domain of an $\R$--defined radical parametrization $\cP_\R$ by $\mathrm{Dom}_\R(\cP_\R)$.

\begin{theorem}\label{thm:real domain}
    Let  $(\T_\R,\cP_\R)$ be given by $\cP_\R\in\F_{m,\R}^r$. Then
    \begin{equation}\label{eq:real domains}
    \widehat{\pi}(\cV_{\T,\R}(\R))\setminus\V(A)(\R)  \subseteq \bigcup_{\cP_\R'\sim \cP_\R}\mathrm{Dom}_\R(\cP_\R').
    \end{equation}
    Moreover, for minimal parametrizations equality in \eqref{eq:real domains} holds.
\end{theorem}
\begin{proof}
Let $(\overline{t}_0,\overline{\delta}_0)$ be a real point in $ \cV_{\T,\R}$ with $A(\overline{t}_0)\neq 0$. We need to show that $\overline{t}_0$ is in the domain of some $\R$--associated parametrization   $\cP_\R'$ of $\cP_\R$. By construction, for every $v\in M$, the $\R$-associated universal parametrization $\Lambda_v\circ\cPR^{(v)}$ is a real section of the projection $\widehat{\pi}$. Moreover, the real fibers of $\widehat{\pi}$ are contained in the image of some $\R$-associated universal parametrization. In other words, for $\ot_{0}\in \R^n$, we have that
\[
\widehat{\pi}^{-1}(\ot_{0})\cap\R^{n+m} = \{\Lambda_v\circ\cPR^{(v)}(\ot_{0}): v\in M
\}\cap\R^{n+m}.
\]
Since $(\ot_{0},\oD_0)$ is a real point in the fiber $\widehat{\pi}^{-1}(\ot_{0})$, there exists an $\R$--associated universal parametrization $\cPR'$  such that $\cPR'(\ot_{0})= (\ot_{0},\oD_0)$. Therefore, $\overline{t}_0$ is in the domain of $\cPR'$, and hence, it is also in the domain of $\cR\circ\cPR'$. We conclude that $\overline{t}_0$ is contained in the right--hand side of \eqref{eq:real domains} as $\cR\circ\cPR'$ is $\R$-associated to $\cP_\R$.

     Assume now that $(\T_\R,\cP_\R)$ is minimal. Let $\ot_{0}\in \mathrm{Dom}_\R(\cP_\R')$ for an $\R$-associated parametrization $\cP_\R'$ of $\cP$. In other words, there exists $v\in M$ such that $\ot_{0}\in\mathrm{Dom}_\R(\cR\circ\Lambda_v\circ\cPR^{(v)})$. 
     To check that $\ot_{0}$ is on the left--hand side of \eqref{eq:real domains}, it is enough to show that $\Lambda_v\circ\cPR^{(v)}(\ot_{0})$ is a real point of $\cV_{\T,\R}$. Since $\Lambda_v$ is a real linear isomorphism, we may assume that $\cPR^{(v)}$ is the real universal parametrization $\cPR$ of $\T_\R$.
     
     Therefore, we need to show that $\delta_i(\ot_{0})$ is real for all $i\in\{1,\ldots,m\}$. Here, $\delta_i$ represent the real valued radical function
     $\sqrt[e_i]{\alpha_i(t)}$.
     Identifying $\delta_i$ with $\Delta_i$, we may assume that $p_1,\ldots,p_r$ and $\alpha_1,\ldots,\alpha_m$ are written  canonically,  viewing   $\F_{m,\R}$ as  $\F_{0,\R}$--vector space. Let $\Delta_{i_1},\ldots,\Delta_{i_k}$ be the $\Delta$'s that appear in the canonical representation of $p_1,\ldots,p_r$. Since $\ot_{0}$ is in the domain of $\cP_\R$, we get that $\delta_{i_1}(\ot_{0}),\ldots,\delta_{i_k}(\ot_{0})$ are real. Since $\delta_{i_1}(\ot_{0})$ is real, $\alpha_{i_1}(\ot_{0})$ is also real. In particular, $\delta_{i}(\ot_{0})$ is real for all $i\in\{1,\ldots,m\}$ such that $\Delta_i$ appears in the canonical representation of $\alpha_{i_1}$. Repeating this process recursively to each $\delta_i$ such that $\delta_i(\ot_{0})$ is real, we get that $\delta_i(\ot_{0})$ is real for every $i\in\{1,\ldots,m\}$. Indeed, this recursion is exactly the same used in Subsection \ref{sec:algorithms} for calculating the minimal radical subtower of the parametrization $\cP_\R$.
     Therefore, through this process we get that $\delta_i(\ot_{0})$ is real for every $i$ in the subset associated to the minimal radical subtower of $\T_\R$. The proof follows from the minimality of $\cP_\R$ as the set associated to the minimal radical subtower is $\{1,\ldots,m\}$.
\end{proof}

\begin{remark}
   The strength of Theorem \ref{thm:real domain} lies in the fact that the left--hand side of \eqref{eq:real domains} almost does not depend on the parametrization, whereas the right--hand side a priori depends strongly on the  chosen parametrization. More concretely, we have that 
   \[
\begin{aligned}
  \widehat{\pi}(\cV_{\T,\R}(\R))\setminus \V(A)(\R)
  &= \Bigl(\widehat{\pi}(\cV_{\T,\R}(\R))\setminus\{\alpha_{i,\mathrm{D}}(t)=0:1\leq i\leq m\}\Bigr)\\
  &\qquad \setminus\{p_{i,\mathrm{D}}(t)=0:1\leq i\leq r\}.
\end{aligned}
\]
    Note that 
    \begin{equation}\label{eq:some projection}
    \widehat{\pi}(\cV_{\T,\R}(\R))\setminus\{\alpha_{i,\mathrm{D}}(t)=0:1\leq i\leq m\}
    \end{equation}
    is the projection via $\widehat{\pi}$ of the real points of a nonempty Zariski-open subset of the tower variety and it is intrinsic from the tower. Moreover, the left--hand side of \eqref{eq:real domains} is obtained by taking out from \eqref{eq:some projection} the points in 
   the proper Zariski-closed subset $\{p_{i,\mathrm{D}}(t)=0:1\leq i\leq r\}$ arising from the denominators of the parametrization. Note that  by construction, the tower variety is never contained in this closed subset. Hence, \eqref{eq:real domains} and \eqref{eq:some projection} differ by a closed subset. When the real points of the tower variety also have dimension $n$,  this closed subset is proper in \eqref{eq:some projection}. So in this case, the difference between \eqref{eq:real domains} and \eqref{eq:some projection} is on a proper Zariski-closed subset. If the real points of the tower variety do not form an $n$-dimensional set, it could happen that these real points are contained in $\{p_{i,\mathrm{D}}(t)=0:1\leq i\leq r\}$.
\end{remark}

The idea behind Theorem \ref{thm:real domain} is that the real points of the tower variety encode the domain of every $\R$-defined radical parametrization $(\T_\R,\cP_\R)$ and its $\R$--associated radical parametrizations. More concretely,
from the proof of Theorem \ref{thm:real domain} we obtain a covering by semialgebraic sets of the real points of the tower variety given by the domains of the $\R$-associated real universal parametrizations of $\T_\R$. More precisely, we have that 
\begin{equation}\label{eq:subdiv of real tower}
\begin{array}{c}
\displaystyle
\cV_{\T,\R}(\R)\setminus \{\alpha_{i,\mathrm{D}}(t)=0:1\leq i\leq m\} = \bigcup_{\cPR'\sim\cPR} \mathrm{Im}_\R(\cPR')=\\ 
\displaystyle
\bigcup_{\cPR'\sim\cPR} \cPR'(\mathrm{Dom}_\R(\cPR'))
\end{array}
\end{equation}
For any minimal parametrization $\cP_\R\in\F_{m,\R}^r$, its real domain will be the projection of one of these
semialgebraic sets minus the zeros of the denominators of $\cP_\R$.

\begin{example}
     Consider the real tower
    \begin{equation}\label{eq:tower real example curve}
    \T_\R:=\left[
    \F_{0,\R}=\R(T)\subsetneq\F_{1,\R}=\F_{0,\R}[\Delta_1]/\langle \Delta_1^2+T^2-1\rangle\subsetneq\F_{2,\R}=\F_{1,\R}[\Delta_2]/\langle \Delta_2^2-T\Delta_1\rangle
    \right].
    \end{equation}
    Over the complex numbers, the tower variety  $\cV_\T$ is an affine chart of a space elliptic curve. Over $\R$, the tower variety 
    $\cV_{\T,\R}$ is a space curve in $\R^3$ with two ovals (see Figure \ref{fig:subdiv real rad curve}). The $\R$--defined universal radical parametrization is the map
    \[
    \begin{array}{cccl}
         \cPR:&\R&\dashrightarrow&\cV_{\T,\R}\subset \R^3  \\
         & t&\mapsto&(t,\sqrt{1-t^2},\sqrt{t\sqrt{1-t^2}}),
    \end{array}
    \]
    and its domain is the closed interval $[0,1]$. In this example, $M=\{\pm 1\}^2$ and the $\R$-associated universal parametrizations, together with their domains, are
    \[
    \begin{array}{cccccccc}
    \Lambda_{(1,-1)}\circ\cPR^{(1,-1)}:&[0,1]&\longrightarrow&\cV_{\T,\R},
   & &t&\longmapsto&(t,\sqrt{1-t^2},-\sqrt{t\sqrt{1-t^2}})
    \\
    \Lambda_{(-1,1)}\circ\cPR^{(-1,1)}:&[-1,0]&\longrightarrow&\cV_{\T,\R},
     & &t&\longmapsto&(t,-\sqrt{1-t^2},\sqrt{-t\sqrt{1-t^2}})
    \\
    \Lambda_{(-1,-1)}\circ\cPR^{(-1,-1)}:&[-1,0]&\longrightarrow&\cV_{\T,\R},
     & &t&\longmapsto&(t,-\sqrt{1-t^2},-\sqrt{-t\sqrt{1-t^2}})
    \end{array}
    \]
    Figure \ref{fig:subdiv real rad curve} illustrates the subdivision of the real points of the tower variety given in \eqref{eq:subdiv of real tower}. The arc of $\cV_{\T,\R}(\R)$ illustrated in  Figure \ref{fig:subdiv real rad curve} in red corresponds to the image of $\cPR$. The purple corresponds to $\Lambda_{(1,-1)}\circ\cPR^{(1,-1)}$, the blue to $ \Lambda_{(-1,1)}\circ\cPR^{(-1,1)}$ and the green to $ \Lambda_{(-1,-1)}\circ\cPR^{(-1,-1)}$. Any minimal parametrization $\cP_\R
    \in\F_{m,\R}^r$ has $4$ $\R$--associated radical parametrizations. Their domain  is the projection of one of these arcs of $\cV_{\T,\R}(\R)$ via $\widehat{\pi}$,  up to possibly a finite number of points   given by the denominators of $\cP$.
    \begin{figure}
        \centering
        \includegraphics[width=0.5\linewidth]{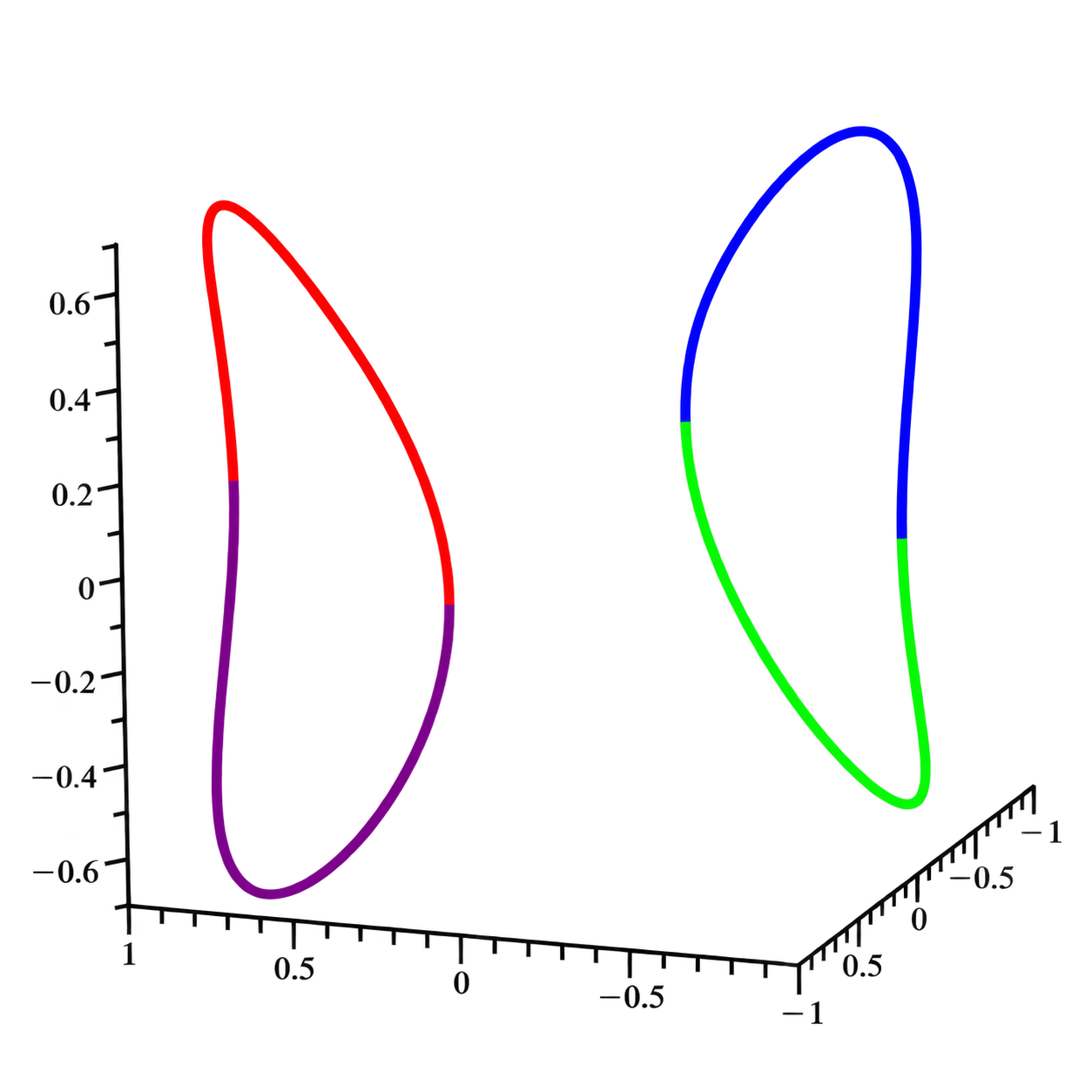}
        \caption{Real points of the tower variety associated to the real radical tower \eqref{eq:tower real example curve}}
        \label{fig:subdiv real rad curve}
    \end{figure}
\end{example}

\begin{example}
  Consider the real tower
\begin{equation}\label{eq:tower real example surface}
\begin{aligned}
\T_\R:=\biggl[
\F_{0,\R}=\R(T_1,T_2)
&\subsetneq
\F_{1,\R}
=\F_{0,\R}[\Delta_1]/
\left\langle
\Delta_1^4+\frac14-T_1^2-T_2^2
\right\rangle
\\
&\subsetneq
\F_{2,\R}
=\F_{1,\R}[\Delta_2]/
\left\langle
\Delta_2^2-1+\Delta_1^2+T_1^2+T_2^2
\right\rangle
\biggr].
\end{aligned}
\end{equation}
        Over $\R$, the tower variety  $\cV_{\T,\R}$  is a surface in $\R^4$ which is contained in the real $3$-dimensional sphere $S^3\subset \R^4$. Since the point $e_3:=(0, 0, 1, 0)\in S^3$ is not in $\cV_{\T,\R}$, we can use the stereographic projection $S^3\setminus\{e_3\}\rightarrow \R^3$ to find a space model of the tower variety. Using the coordinates $x,y,z$ of $\R^3$, the tower variety is mapped to the surface in $\R^3$ defined by the equation
    \[
    4\left(x^{2}+y^{2}+z^{2}-1\right)^{4}
+\left(x^{2}+y^{2}+z^{2}+1\right)^{2}
\left[\left(x^{2}+y^{2}+z^{2}+1\right)^{2}-16\left(x^{2}+y^{2}\right)\right]=0.
    \]
    Figure \ref{fig:subdiv real rad surface} depicts this space model of the real points of the tower variety. In this example, $M=\{\pm1\}^2$ and there are four $\R$-associated universal radical parametrizations given by possible choices of positive or negative square and $4$-th  real roots. The image of these four radical parametrization are four semialgebraic sets that cover  $\cV_{\T,\R}(\R)$. Each of these patches is illustrated in different colors in Figure \ref{fig:subdiv real rad surface}. Now, any minimal parametrization $\cP_\R\in\F_{m,\R}^r$ has four $\R$-associated radical parametrizations. Their domains are the projections of one of these patches 
    of $\cV_{\T,\R}(\R)$ via $\widehat{\pi}$, up to possibly a real algebraic curve given by the denominators of $\cP_\R$.
   \begin{figure}
    \centering
   \includegraphics[width=7cm]{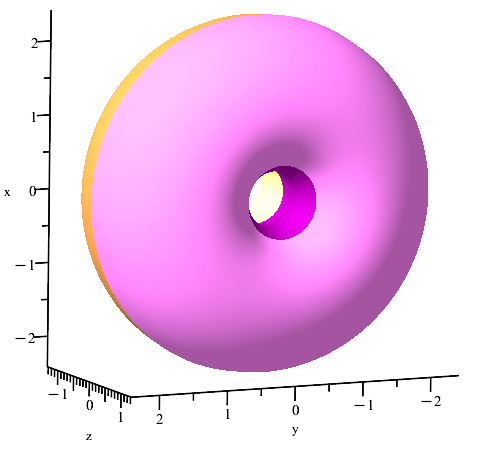} 
    \includegraphics[width=7cm]{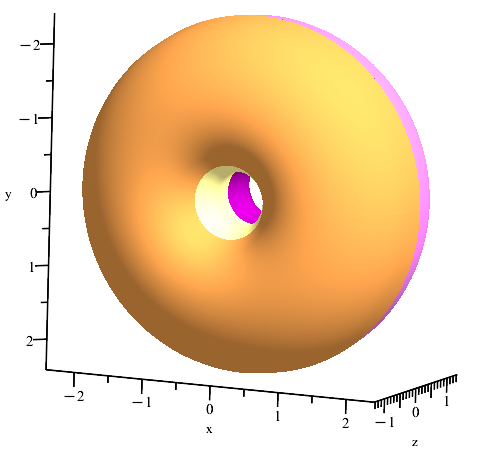} 
    \caption{Real points of the tower variety associated to the tower \eqref{eq:tower real example surface}.}
    \label{fig:subdiv real rad surface}
\end{figure}
\end{example}

\begin{corollary}\label{cor: real points para}
Assume that $\mathscr{V}_{\T,\R}(\R)\setminus\{\alpha_{i,\mathrm{D}}(t)=0:1\leq i\leq m\}$ is non empty.
 Then, for every $\R$--defined radical parametrization $\cP_\R$ in $\T_\R$ with 
 \[
 \mathscr{V}_{\T,\R}(\R)\setminus\V(A)(\R)\neq \emptyset,
 \]
  there exists an $\R$-associated radical parametrization $\cP_\R'\sim\cP_\R$ with nonempty domain. Moreover, if $\mathscr{V}_{\T,\R}(\R)$ is Zariski dense in $\mathscr{V}_\T$, then, for every $\R$--defined radical parametrization $\cP_\R$ in $\T_\R$, there exists an $\R$-associated radical parametrization $\cP_\R'\sim\cP_\R$ such that the domain of $\cP_\R'$ contains a Euclidean open subset and $\mathscr{V}_{\cP_\R}(\R)$ is Zariski dense in  $\mathscr{V}_{\cP_\R}$. In particular, $\mathscr{V}_{\cP_\R}(\R)$ contains a Euclidean open subset of dimension $n$.
\end{corollary}

\begin{remark}
     A common situation where the assumptions of Corollary \ref{cor: real points para} are satisfied is when $\mathscr{V}_{\T,\R}$ has a real smooth point; see \cite[Prop. 3.3.16]{BCR98} or \cite[Theorem 5.1]{sottile2018real}.
\end{remark}

\begin{corollary}\label{cor: no real points para}
If $\mathscr{V}_{\T,\R}(\R)\setminus\{\alpha_{i,\mathrm{D}}(t)=0:1\leq i\leq m\}$ is empty, then  every minimal $\R$--defined radical parametrizations $(\T_\R,\cP_\R)$ 
and its
$\R$-associated radical parametrizations have empty domain.
\end{corollary}

\begin{remark}
  By Corollary \ref{cor: no real points para}, if $\mathscr{V}_{\T,\R}(\R)=\emptyset$, every minimal $\R$-defined radical parametrization in $\T_\R$ has empty domain. However, this does not imply that $\cV_{\cP_\R}(\R)$ is empty as shown in Example \ref{ex-pat1}.
\end{remark}

\section{Computational improvements and experiments\protect\footnote{The experiments reported in this section were designed by the authors. Their implementation was carried out with the AI system Claude (Anthropic) under the supervision of the authors. The authors reviewed the methodology, the code and the results.
}}\label{sec:computation}
In this section, we present two experiments where we analyze the running time for the computation of $\cV_\T$ and $\cV_\cP$ using the results in \cite{RSV2017} in comparison with the results in this paper.   
All the methods were implemented in \textsc{Singular}~\cite{singular} with the same elimination routine, so that the differences in time come from the methods and not from the implementation.
Each computation had a limit of $600$\,s.

The \textsc{Singular} scripts, the input data, the certification scripts, the  timings  and the ideals computed are available on Zenodo~\cite{zenodo}. They allow the experiment to be reproduced and extended to other examples.

\subsection{Design of the experiments}

We consider two experiments. The first one deals with the computation of the tower variety $\cV_\T$ and the second with the computation of the radical tower variety $\cV_\cP$.  For the experiments, we analyze the following procedure
\begin{itemize}
\item[(a)] Method (a) is the \emph{classical approach}. It is based directly on the results in \cite{RSV2017}.  Thus, it needs to check the irreducibility of the incidence variety $\cA$ to choose the component $\mathcal{B}$; even though we know, by our Theorem \ref{thm: incidence var}, that $\cA$ is irreducible and hence $\cA=\mathcal{B}$. If necessary, in the elimination processes from $\cD$ to $\cA$ and to either $\cV_\T$ or $\cV_\cP$ saturation of the ideals are computed. 
\item[(b)] Method (b) is the direct adaptation of Method (a) to the current paper. Thus, it follows the scheme of (a) but the irreducibility is not checked. 
\item[(c)] Method (c) is for the first experiment. First it writes the tower $\T$ in canonical form  and apply Algorithm \ref{alg: tower} to get the isomorphic tower $\T'$. The output is the ideal $\langle E_1',\dots,E_m'\rangle$ of $\mathscr{V}_{\mathbb{T}'}$, which is birational to $\mathscr{V}_{\mathbb{T}}$  (see Theorem \ref{theo-change of var}).
\item[(d)] Method (d) is for the second experiment. First it writes the tower $\T$ and  the radical parametrization $\cT$ in canonical form. Then, it compute the minimal radical subtower $\mathbb{T}_S$ of $\mathscr{P}$ with Algorithms~\ref{alg:1} and \ref{alg:3}. Next, it determines $\mathscr{V}_{\mathscr{P}}=\mathscr{V}_{\mathscr{P}_{\mathbb{T}_S}}$ (see Theorem~\ref{thm:subtower var}  ) from the variety $\mathscr{D}_{\mathscr{P}}$ of $(\mathbb{T}_S,\mathscr{P})$, eliminating $W$, $\overline{T}$ and $\Delta_i$, $i\in S$.
\end{itemize}

\subsection{Test data}
For our experiments, we consider 2 towers of 3 floors, 3 towers of 5 floors, 2 towers of 7 floors  and 5 towers of 10 floors.  The data is available on Zenodo~\cite{zenodo}, but for convenience of the reader we have included them in Appendix \ref{app:towers}.  In all the examples $n=1$ and $r=2$, so $\mathscr{V}_{\mathscr{P}}$ is always a plane curve. 
The floors of the towers are defined as 
\[
\Delta_i^{e_i}=\alpha_i=\frac{(T-a_i)\,u_i}{d_i},\qquad e_i\in\{2,3\},\quad a_i\in\mathbb{Z},
\]
where $u_i$ and $d_i$ have low degree and may depend on $\oD$.   The parametrization $\cP$ is built from the radicals of a proper radical subtower and so that  the minimal radical subtower of $\mathscr{P}$  cannot be read off from the given expression of $\mathscr{P}$.

\subsection{Timings and results}
The computations have been executed in \textsc{Singular}~4.3.2~\cite{singular}.  To compare the methods on equal terms, all the eliminations and saturations  were computed using  the same routine, \texttt{elim(\dots,"std")} from \texttt{elim.lib}. The computations ran on an Intel Xeon processor at 2.80\,GHz with two cores and 7\,GB of memory, running at most two computations at a time, with a limit of 600\,s and 3\,GB per computation. We report wall-clock times: the median of five runs for computations that finished in less than 60\,s, and a single run otherwise. The scripts and the timings of all the computations are available on Zenodo~\cite{zenodo}.

Tables~\ref{tab:times}, ~\ref{tab:summary} show the computing times of the two experiments.
The first  one gives the time of every computation, and the second  summarizes, for each number of floors, in how many towers each method finished before the time limit.

\begin{table}[H]
\centering
\small
\setlength{\tabcolsep}{4pt}
\begin{tabular}{cc r c r | r r r | r r r r}
\hline
& & & & & \multicolumn{3}{c|}{Experiment 1: $\mathscr{V}_{\mathbb{T}}$} & \multicolumn{4}{c}{Experiment 2: $\mathscr{V}_{\mathscr{P}}$}\\
$m$ & Tower & $\deg(\mathbb{T})$ & $S$ & $\deg(\mathbb{T}_S)$ & (a) & (b) & (c) & (a) & (b) & (d)  \\
\hline
 {3} & A & 12 & $\{1,3\}$ & 6 & 0.268 & 0.043 & \textbf{0.002} & 0.842 & \textbf{0.624} & 1.87   \\
 & B & 12 & $\{2,3\}$ & 4 & 0.197 & 0.146 & \textbf{0.002} & 0.193 & \textbf{0.133} & 0.630   \\
\hline
 {5} & A & 48 & $\{1,2,4\}$ & 12 & $>600$ & $>600$ & \textbf{0.003} & $>600$ & $>600$ & $>600$  \\
 & B & 48 & $\{1,2,4,5\}$ & 24 & $>600$ & $>600$ & \textbf{0.004} & $>600$ & $>600$ & $>600$   \\
 & C & 48 & $\{3,4\}$ & 4 & $>600$ & 7.43 & \textbf{0.004} & $>600$ & 7.44 & \textbf{0.153}  \\
\hline
 {7} & A & 192 & $\{1,2,3\}$ & 8 & $>600$ & $>600$ & \textbf{2.35} & $>600$ & $>600$ & $>600$   \\
 & B & 192 & $\{1,2,3\}$ & 12 & $>600$ & $>600$ & \textbf{0.027} & $>600$ & $>600$ & $>600$   \\
\hline
 {10} & A & 5184 & $\{1,5\}$ & 4 & $>600$ & $>600$ & \textbf{0.010} & $>600$ & $>600$ & \textbf{147.6}   \\
 & B & 1536 & $\{2,6\}$ & 4 & $>600$ & $>600$ & \textbf{0.025} & $>600$ & $>600$ & \textbf{37.29}   \\
 & C & 2304 & $\{1,2\}$ & 6 & $>600$ & $>600$ & \textbf{0.013} & $>600$ & $>600$ & \textbf{195.1}  \\
 & D & 2304 & $\{1,2\}$ & 4 & $>600$ & $>600$ & \textbf{0.244} & $>600$ & $>600$ & \textbf{36.50}  \\
 & E & 1536 & $\{1,5,10\}$ & 8 & $>600$ & $>600$ & \textbf{0.010} & $>600$ & $>600$ & $>600$  \\
\hline
\end{tabular}
\caption{Computing times in seconds. $\deg(\mathbb{T})=[\mathbb{F}_m:\mathbb{F}_0]$; $S$ is the subset of the minimal radical subtower of $\mathscr{P}$ and $\deg(\mathbb{T}_S)$ its degree over $\mathbb{F}_0$. For each experiment, the best time in each row is in bold. $>600$: the time limit was reached.}
\label{tab:times}
\end{table}

\begin{table}[H]
\centering
\small
\begin{tabular}{@{\hspace{12pt}}c c | c c c | c c c@{\hspace{12pt}}}
\hline
& &
\multicolumn{3}{c|}{Experiment 1: $\mathscr{V}_{\mathbb{T}}$} &
\multicolumn{3}{c@{\hspace{12pt}}}{Experiment 2: $\mathscr{V}_{\mathscr{P}}$}\\
$m$ & Towers & (a) & (b) & (c) & (a) & (b) & (d)\\
\hline
3 & 2
& \textbf{2/2} & \textbf{2/2} & \textbf{2/2}
& \textbf{2/2} & \textbf{2/2} & \textbf{2/2}\\
5 & 3
& 0/3 & 1/3 & \textbf{3/3}
& 0/3 & \textbf{1/3} & \textbf{1/3}\\
7 & 2
& 0/2 & 0/2 & \textbf{2/2}
& 0/2 & 0/2 & 0/2\\
10 & 5
& 0/5 & 0/5 & \textbf{5/5}
& 0/5 & 0/5 & \textbf{4/5}\\
\hline
Total & 12
& 2/12 & 3/12 & \textbf{12/12}
& 2/12 & 3/12 & \textbf{7/12}\\
\hline
\end{tabular}
\caption{Number of towers for which each method finished within 600\,s.
For each experiment, the highest nonzero count in each row is shown
in bold, including ties.}
\label{tab:summary}
\end{table}

\subsection{Analysis of the results}

The experiments support the results of Section \ref{sec:refinamient} and \ref{sec:subtower}. Algorithm~\ref{alg: tower} gives the tower variety of the isomorphic tower $\mathbb{T}'$ for towers  up to ten floors and degree $5184$, where the previous approach, from \cite{RSV2017}, does not finish.   For the radical variety, working on the minimal radical subtower provided by Algorithms~\ref{alg:1} and \ref{alg:3} is what makes the computation feasible for towers with many floors: it succeeded for four of the five towers with ten floors, whereas neither the classical approach nor the same method applied to the whole tower finished. 

For some cases, no method finish the computation.   In all these cases the bottleneck is a single Gr\"obner basis computation, whose cost grows very quickly with the number of variables and the degrees involved.

\newpage
\appendix
\section{Radical towers and parametrizations in Section \ref{sec:computation}}\label{app:towers}

This appendix lists the twelve radical towers $\mathbb{T}$ and parametrizations $\mathscr{P}=(p_1,p_2)$ used in the experiments.

\paragraph{Towers with 3 floors}
\paragraph{Tower 3A} 
\[
\Delta_{1}^{2} = \dfrac{3 - T}{T + 1},\quad \Delta_{2}^{2} = \dfrac{\left(\Delta_{1} + T\right) \left(T + 2\right)}{\Delta_{1} - 1},\quad \Delta_{3}^{3} = \dfrac{\left(\Delta_{1} + T^{2}\right) \left(T - 4\right)}{T + 1},
\]
\[
p_{1} = \dfrac{\Delta_{3} + T}{\Delta_{1} + 1},\quad p_{2} = \dfrac{\Delta_{1} \Delta_{3}^{2} \left(T + 2\right) - \Delta_{2}^{2} \left(\Delta_{1} - 1\right) + T \left(T + 2\right)}{T + 2}
\]
 \paragraph{Tower 3B}  
\[
\Delta_{1}^{3} = \dfrac{T - 1}{T + 2},\quad \Delta_{2}^{2} = \dfrac{\left(T + 3\right) \left(\Delta_{1}^{3} \left(T + 2\right) + 2\right)}{T - 2},\quad \Delta_{3}^{2} = \dfrac{\left(\Delta_{2} + T\right) \left(T - 5\right)}{\Delta_{2} + 1},
\]
\[
p_{1} = \Delta_{1}^{3} \left(T + 2\right) + \Delta_{3},\quad p_{2} = \dfrac{\Delta_{2} \Delta_{3} + T}{\Delta_{2} - T}
\]
\paragraph{Towers with 5 floors}
\paragraph{Tower 5A}  
\begin{gather*}
\Delta_{1}^{2} = \dfrac{2 - T}{T + 3},\quad \Delta_{2}^{2} = \dfrac{\left(\Delta_{1} + 2\right) \left(T + 1\right)}{\Delta_{1} + T},\quad \Delta_{3}^{2} = \dfrac{\left(T - 3\right) \left(T^{2} + 1\right)}{T + 4},\\
\Delta_{4}^{3} = \dfrac{\left(\Delta_{2} + T\right) \left(T + 6\right)}{T - 1},\quad \Delta_{5}^{2} = \dfrac{\left(\Delta_{1} + \Delta_{3}\right) \left(T - 4\right)}{\Delta_{3} + 1},
\end{gather*}
\[
p_{1} = \dfrac{\Delta_{4} + T}{\Delta_{2} - 1},\quad p_{2} = \Delta_{1} \Delta_{4} + \Delta_{3}^{2} \left(T + 4\right)
\]
 \paragraph{Tower 5B}  
\begin{gather*}
\Delta_{1}^{3} = \dfrac{T - 1}{T + 2},\quad \Delta_{2}^{2} = \dfrac{\left(\Delta_{1} + 1\right) \left(T + 3\right)}{T - 2},\quad \Delta_{3}^{2} = \dfrac{\left(\Delta_{1}^{2} + T\right) \left(T - 5\right)}{\Delta_{2} + 1},\\
\Delta_{4}^{2} = \dfrac{\left(\Delta_{1} + T\right) \left(T + 4\right)}{T + 1},\quad \Delta_{5}^{2} = \dfrac{\left(\Delta_{2} + \Delta_{4}\right) \left(T - 6\right)}{T + 5},
\end{gather*}
\[
p_{1} = \Delta_{3}^{2} \left(\Delta_{2} + 1\right) + \Delta_{5},\quad p_{2} = \dfrac{\Delta_{2} \Delta_{5} + T}{\Delta_{4} + 1}
\]

\paragraph{Tower 5C}  
\begin{gather*}
\Delta_{1}^{2} = \dfrac{T + 1}{T - 3},\quad \Delta_{2}^{3} = \dfrac{\left(\Delta_{1} + T\right) \left(T - 2\right)}{\Delta_{1} - 2},\quad \Delta_{3}^{2} = \dfrac{\left(T + 5\right) \left(T^{2} + T + 1\right)}{T - 4},\\
\Delta_{4}^{2} = \dfrac{\left(\Delta_{3} + T\right) \left(T - 1\right)}{T + 2},\quad \Delta_{5}^{2} = \dfrac{\left(\Delta_{2} + \Delta_{4}\right) \left(T + 3\right)}{\Delta_{4} - 1},
\end{gather*}
\[
p_{1} = \dfrac{\Delta_{3} + \Delta_{4}}{\Delta_{3} - T},\quad p_{2} = \dfrac{\Delta_{1}^{2} \left(T - 3\right)^{2} + \Delta_{4} \left(T - 3\right)}{T - 3}
\]
\paragraph{Towers with 7 floors}
\begin{gather*}
\Delta_{1}^{2} = \dfrac{\left(T + 7\right) \left(T^{2} - T + 2\right)}{T + 1},\quad \Delta_{2}^{2} = \dfrac{\left(\Delta_{1} - 3\right) \left(T + 11\right)}{T + 2},\quad \Delta_{3}^{2} = \dfrac{\left(T + 5\right) \left(\Delta_{2} - 2 T - 2\right)}{T - 8},\\
\Delta_{4}^{3} = \dfrac{\left(T - 9\right) \left(\Delta_{3} + 3 T - 2\right)}{T + 3},\quad \Delta_{5}^{2} = \dfrac{\left(T + 9\right) \left(\Delta_{2} + T - 5\right)}{T - 2},\\
\Delta_{6}^{2} = \dfrac{\left(T + 8\right) \left(\Delta_{4} + 3 T - 4\right)}{\Delta_{4} - 3},\quad \Delta_{7}^{2} = \dfrac{\left(T + 4\right) \left(\Delta_{6} + T + 2\right)}{\Delta_{6} - 3},
\end{gather*}
\[
p_{1} = \dfrac{\Delta_{3} + T}{\Delta_{2} + 2},\quad p_{2} = \Delta_{1} \Delta_{3} + \Delta_{4}^{3} \left(T + 3\right) + T - \left(T - 9\right) \left(\Delta_{3} + 3 T - 2\right)
\]

\paragraph{Tower 7B} 
\begin{gather*}
\Delta_{1}^{3} = \dfrac{\left(T - 8\right) \left(T^{2} + T - 2\right)}{T + 1},\quad \Delta_{2}^{2} = \dfrac{\left(T - 9\right) \left(\Delta_{1} + T - 5\right)}{T + 6},\\
\Delta_{3}^{2} = \dfrac{\left(T + 8\right) \left(\Delta_{2} - T + 5\right)}{\Delta_{2} + 4},\quad \Delta_{4}^{2} = \dfrac{\left(T - 10\right) \left(\Delta_{3} + 2 T + 3\right)}{\Delta_{3} - 4},\\
\Delta_{5}^{2} = \dfrac{\left(T - 2\right) \left(T^{2} - 2 T + 4\right)}{T + 7},\quad \Delta_{6}^{2} = \dfrac{\left(T + 4\right) \left(T^{2} - 1\right)}{T + 3},\quad \Delta_{7}^{2} = \dfrac{\left(T - 4\right) \left(\Delta_{1} + 2 T + 5\right)}{T + 7},
\end{gather*}
\[
p_{1} = \dfrac{\Delta_{3} + T}{\Delta_{2} + 2},\quad p_{2} = \Delta_{2} \Delta_{3} + \Delta_{4}^{2} \left(\Delta_{3} - 4\right) + T - \left(T - 10\right) \left(\Delta_{3} + 2 T + 3\right)
\]

\paragraph{Towers with 10 floors}
\paragraph{Tower 10A}  
\begin{gather*}
\Delta_{1}^{2} = \dfrac{\left(T + 9\right) \left(T^{2} - 3 T - 4\right)}{T - 6},\quad \Delta_{2}^{3} = \dfrac{\left(\Delta_{1} + 2\right) \left(T - 8\right)}{\Delta_{1} - 2},\\
\Delta_{3}^{2} = \dfrac{\left(T - 11\right) \left(\Delta_{2} + 3 T + 2\right)}{T + 8},\quad \Delta_{4}^{3} = \dfrac{\left(T + 6\right) \left(\Delta_{3} + 2 T - 5\right)}{T - 1},\\
\Delta_{5}^{2} = \dfrac{\left(T + 4\right) \left(\Delta_{1} - 3 T + 1\right)}{\Delta_{1} + 4},\quad \Delta_{6}^{2} = \dfrac{\left(\Delta_{2} - 4\right) \left(T - 2\right)}{T - 4},\quad \Delta_{7}^{3} = \dfrac{\left(T + 11\right) \left(\Delta_{1} - 2 T + 1\right)}{T + 8},\\
\Delta_{8}^{3} = \dfrac{\left(T + 2\right) \left(\Delta_{7} + T - 5\right)}{T + 5},\quad \Delta_{9}^{2} = \dfrac{\left(\Delta_{5} + 3\right) \left(T + 3\right)}{\Delta_{5} - 3},\quad \Delta_{10}^{2} = \dfrac{\left(T - 10\right) \left(\Delta_{7} + T + 1\right)}{\Delta_{7} + 2},
\end{gather*}
\[
p_{1} = \dfrac{\Delta_{5} + T}{\Delta_{5} + 2},\quad p_{2} = \Delta_{1} \Delta_{5} + \Delta_{9}^{2} \left(\Delta_{5} - 3\right) + T - \left(\Delta_{5} + 3\right) \left(T + 3\right)
\]

\paragraph{Tower 10B}
\begin{gather*}
\Delta_{1}^{3} = \dfrac{\left(T - 6\right) \left(T^{2} - 2 T - 2\right)}{T + 2},\quad \Delta_{2}^{2} = \dfrac{\left(T - 9\right) \left(T^{2} + 5\right)}{T - 1},\quad \Delta_{3}^{2} = \dfrac{\left(T + 9\right) \left(\Delta_{1} - 3 T - 2\right)}{\Delta_{1} + 4},\\
\Delta_{4}^{2} = \dfrac{\left(T + 5\right) \left(\Delta_{3} + 3 T + 1\right)}{T + 2},\quad \Delta_{5}^{2} = \dfrac{\left(T - 4\right) \left(\Delta_{2} - 3 T - 5\right)}{\Delta_{2} + 2},\\
\Delta_{6}^{2} = \dfrac{\left(T - 2\right) \left(\Delta_{2} + 3 T + 3\right)}{\Delta_{2} - 1},\quad \Delta_{7}^{2} = \dfrac{\left(T + 1\right) \left(\Delta_{6} + 3 T + 2\right)}{\Delta_{6} - 2},\\
\Delta_{8}^{2} = \dfrac{\left(T - 8\right) \left(T^{2} - T + 1\right)}{T + 2},\quad \Delta_{9}^{2} = \dfrac{\left(T - 12\right) \left(\Delta_{6} + 3 T - 2\right)}{T + 3},\quad \Delta_{10}^{2} = \dfrac{\left(\Delta_{7} - 3\right) \left(T - 3\right)}{\Delta_{7} + 3},
\end{gather*}
\[
p_{1} = \dfrac{\Delta_{6} + T}{\Delta_{2} + 2},\quad p_{2} = \Delta_{1}^{3} \left(T + 2\right) + \Delta_{2} \Delta_{6} + T - \left(T - 6\right) \left(T^{2} - 2 T - 2\right)
\]

\paragraph{Tower 10C} 
\begin{gather*}
\Delta_{1}^{2} = \dfrac{\left(T - 5\right) \left(T^{2} + T - 5\right)}{T + 6},\quad \Delta_{2}^{3} = \dfrac{\left(T - 2\right) \left(\Delta_{1} + T + 2\right)}{T + 3},\quad \Delta_{3}^{2} = \dfrac{\left(T - 9\right) \left(T^{2} - 2 T + 1\right)}{T + 3},\\
\Delta_{4}^{2} = \dfrac{\left(T - 11\right) \left(T^{2} + T - 4\right)}{T + 7},\quad \Delta_{5}^{2} = \dfrac{\left(T - 6\right) \left(\Delta_{3} - T + 2\right)}{T + 9},\\
\Delta_{6}^{2} = \dfrac{\left(T + 11\right) \left(T^{2} - 2 T + 1\right)}{T - 1},\quad \Delta_{7}^{2} = \dfrac{\left(T - 8\right) \left(T^{2} + 3 T + 3\right)}{T + 6},\\
\Delta_{8}^{3} = \dfrac{\left(\Delta_{3} + 2\right) \left(T + 1\right)}{\Delta_{3} - 2},\quad \Delta_{9}^{2} = \dfrac{\left(T + 10\right) \left(T^{2} - T + 5\right)}{T + 3},\quad \Delta_{10}^{2} = \dfrac{\left(T + 5\right) \left(\Delta_{8} + 2 T - 1\right)}{\Delta_{8} + 4},
\end{gather*}
\[
p_{1} = \dfrac{\Delta_{2} + T}{\Delta_{2} + 2},\quad p_{2} = \Delta_{2}^{2} + \Delta_{3}^{2} \left(T + 3\right) + T - \left(T - 9\right) \left(T^{2} - 2 T + 1\right)
\]

\paragraph{Tower 10D} 
\begin{gather*}
\Delta_{1}^{2} = \dfrac{\left(T + 1\right) \left(T^{2} + T + 2\right)}{T + 6},\quad \Delta_{2}^{2} = \dfrac{\left(T - 12\right) \left(\Delta_{1} - T - 2\right)}{\Delta_{1} - 4},\quad \Delta_{3}^{2} = \dfrac{\left(T - 8\right) \left(T^{2} + 2 T + 4\right)}{T + 2},\\
\Delta_{4}^{2} = \dfrac{\left(T - 10\right) \left(\Delta_{2} - 3 T + 3\right)}{\Delta_{2} + 4},\quad \Delta_{5}^{2} = \dfrac{\left(T - 9\right) \left(T^{2} - 3\right)}{T - 1},\quad \Delta_{6}^{2} = \dfrac{\left(T + 7\right) \left(\Delta_{5} + 2 T + 1\right)}{T - 6},\\
\Delta_{7}^{3} = \dfrac{\left(T - 5\right) \left(\Delta_{4} - 2 T + 2\right)}{\Delta_{4} - 3},\quad \Delta_{8}^{3} = \dfrac{\left(T + 4\right) \left(\Delta_{7} + 3 T - 1\right)}{\Delta_{7} - 4},\\
\Delta_{9}^{2} = \dfrac{\left(T - 3\right) \left(\Delta_{8} - 2 T - 3\right)}{T + 2},\quad \Delta_{10}^{2} = \dfrac{\left(T - 7\right) \left(\Delta_{3} + T + 3\right)}{T + 5},
\end{gather*}
\[
p_{1} = \dfrac{\Delta_{2} + T}{\Delta_{1} + 2},\quad p_{2} = \Delta_{2}^{2} + \Delta_{7}^{3} \left(\Delta_{4} - 3\right) + T - \left(T - 5\right) \left(\Delta_{4} - 2 T + 2\right)
\]

\paragraph{Tower 10E} 
\begin{gather*}
\Delta_{1}^{2} = \dfrac{\left(T + 12\right) \left(T^{2} - T - 3\right)}{T + 1},\quad \Delta_{2}^{2} = \dfrac{\left(T + 8\right) \left(\Delta_{1} - 2 T + 3\right)}{\Delta_{1} + 4},\quad \Delta_{3}^{2} = \dfrac{\left(T + 5\right) \left(\Delta_{1} - 2 T - 3\right)}{\Delta_{1} - 4},\\
\Delta_{4}^{2} = \dfrac{\left(T - 5\right) \left(\Delta_{1} - 2 T - 2\right)}{T - 4},\quad \Delta_{5}^{2} = \dfrac{\left(T + 4\right) \left(\Delta_{1} + 3 T - 1\right)}{\Delta_{1} + 3},\quad \Delta_{6}^{2} = \dfrac{\left(T + 9\right) \left(\Delta_{5} + 2 T + 5\right)}{T - 3},\\
\Delta_{7}^{2} = \dfrac{\left(T - 10\right) \left(\Delta_{2} - T - 3\right)}{\Delta_{2} + 2},\quad \Delta_{8}^{3} = \dfrac{\left(T - 11\right) \left(\Delta_{4} - 2 T + 1\right)}{\Delta_{4} - 2},\\
\Delta_{9}^{2} = \dfrac{\left(T - 2\right) \left(\Delta_{6} - T - 4\right)}{T + 7},\quad \Delta_{10}^{2} = \dfrac{\left(T + 10\right) \left(\Delta_{5} + 3 T + 4\right)}{\Delta_{5} + 4},
\end{gather*}
\[
p_{1} = \dfrac{\Delta_{10} + T}{\Delta_{5} + 2},\quad p_{2} = \Delta_{10}^{2} + \Delta_{8}^{3} \left(\Delta_{4} - 2\right) + T - \left(T - 11\right) \left(\Delta_{4} - 2 T + 1\right)
\]

\begin{thebibliography}{12}

\bibitem{AgarwalEtAl2022}
P.~K. Agarwal, B.~Aronov, E.~Ezra, M.~J. Katz, M.~Sharir.
\newblock Intersection queries for flat semi-algebraic objects in three
dimensions and related problems.
\newblock In \emph{38th International Symposium on Computational Geometry
(SoCG 2022)}, LIPIcs, 2022.

\bibitem{Bartkhin}
Batkhin, A.B. A real variety with boundary and its global parameterization. Program Comput Soft 43, 75–83 (2017).
\bibitem{BesierFesti2021}
M.~Besier and D.~Festi,
  Rationalizability of square roots,
 \emph{Journal of Symbolic Computation}  {106}, 48--67 (2021).

\bibitem{BCR98}
J.~Bochnak, M.~Coste, M.-F.~Roy,
\emph{Real Algebraic Geometry},
Ergebnisse der Mathematik und ihrer Grenzgebiete (3), vol.~36,
Springer-Verlag, Berlin, 1998.

\bibitem{CSSV2021} Caravantes, J., Sendra, J.R., Sevilla, D. Villarino C. Transforming ODEs and PDEs from Radical Coefficients to Rational Coefficients. \textit{Mediterr. J. Math.} 18, 96 (2021).


\bibitem{CaravantesSendraSevillaVillarino2024}
J.~Caravantes, J.~R.~Sendra, D.~Sevilla, and C.~Villarino,
  Sufficient conditions for the surjectivity of radical curve parametrizations,
  \emph{Journal of Algebra} \textbf{640} (2024), 129--146.

\bibitem{CouveignesKammerer2012}
J.-M.~Couveignes and J.-G.~Kammerer,
 The geometry of flex tangents to a cubic curve and its parameterizations,
  \emph{Journal of Symbolic Computation}  {47}, no.~3, 266--281 (2012).

\bibitem{CLO2015}
D. A. Cox, J. Little, D. O'Shea,  \textit{Ideals, Varieties, and Algorithms: An Introduction to Computational Algebraic Geometry and Commutative Algebra}, 3rd ed., Undergraduate Texts in Mathematics, Springer, 2015.


\bibitem{CG72}
C.~H. Clemens and P.~A. Griffiths,
\emph{The intermediate Jacobian of the cubic threefold},
Ann. of Math. (2) \textbf{95} (1972), no.~2, 281--356.

\bibitem{singular} W.~Decker, G.-M.~Greuel, G.~Pfister, H.~Sch\"onemann. \textsc{Singular} 4-3-2 --- A computer algebra system for polynomial computations. \url{https://www.singular.uni-kl.de}, 2023.


\bibitem{FS2024} Falkensteiner, S., Sendra, J.R. Transforming Radical Differential Equations to Algebraic Differential Equations. \textit{Mediterr. J. Math.} 21, 87 (2024).  

 
\bibitem{Grasegger2014} 
G.~Grasegger, 
\newblock Radical solutions of first order autonomous algebraic ordinary differential equations, 
\newblock in: \emph{Proceedings of the 39th International Symposium on Symbolic and Algebraic Computation (ISSAC 2014)}, ACM, 2014, pp.~217--223.

\bibitem{GraseggerVo2017}
G.~Grasegger and N.~T.~Vo,
  An algebraic-geometric method for computing Zolotarev polynomials,
  in: \emph{Proceedings of the 2017 ACM International Symposium on Symbolic and Algebraic Computation (ISSAC '17)},
ACM  2017, pp.~173--180.


\bibitem{Harrison2013}
M.~Harrison,
  Explicit solution by radicals, gonal maps and plane models of algebraic curves of genus 5 or 6,
  \emph{Journal of Symbolic Computation}  {51}  (2013), 3--21.

\bibitem{Hartshorne}
R.~Hartshorne,
\emph{Algebraic Geometry},
Graduate Texts in Mathematics, vol.~52,
Springer-Verlag, New York, 1977.

\bibitem{KKS2025}
Kubjas, K., Kuznetsova, O.  Sodomaco, L. Algebraic degree of optimization over a variety with an application to P-norm distance degree. Acta Univ. Sapientiae Math. 17, 3 (2025).

\bibitem{Lewis2018}
R.~H.~Lewis,
 Resultants, implicit parameterizations, and intersections of surfaces,
  in: J.~H.~Davenport, M.~Kauers, G.~Labahn, and J.~Urban (eds.),
\emph{Mathematical Software -- ICMS 2018},
Springer International Publishing, Cham, 2018, pp.~310--318.


\bibitem{PirolaSchlesinger2005}
G.~P.~Pirola and E.~Schlesinger,
  A curve algebraically but not rationally uniformized by radicals,
 \emph{J. Algebra} \textbf{289} (2005), 412--420.

\bibitem{Pop96}
F.~Pop,
Embedding problems over large fields,
\emph{Ann. of Math. (2)} \textbf{144} (1996), no.~1, 1--34.

\bibitem{RackVajda2021}
H.-J.~Rack and R.~Vajda,
  An explicit radical parametrization of Zolotarev polynomials of
degree 7 in terms of nested square roots,
  \emph{Advanced Studies: Euro-Tbilisi Mathematical Journal}
 {14} no.~4, 37--60. (2021), 
\bibitem{zenodo} J.~Sendra-Arranz, J.~R.~Sendra, J.~Sendra.
Computational experiments on radical towers and their parametrizations: scripts and data (version 1.0).
Zenodo, 2026. \url{https://doi.org/10.5281/zenodo.23623629}
\bibitem{SendraSevilla2011}
J.~R. Sendra, D.~Sevilla.
 Radical parametrizations of algebraic curves by adjoint curves.
  \emph{Journal of Symbolic Computation} 46, 1030–1038 2011.

\bibitem{SendraSevilla2013}
J.~R. Sendra, D.~Sevilla.
  First steps towards radical parametrization of algebraic surfaces.
  \emph{Computer Aided Geometric Design}  30(4), 374–388, 2013.


\bibitem{RSV2017}
J.~R. Sendra, D. Sevilla, C. Villarino, Algebraic and algorithmic aspects of radical parametrizations,  \textit{Comput. Aided Geom. Des.} 55 (2017) 1--14.

\bibitem{sottile2018real}
F.~Sottile,
\emph{Real algebraic geometry for geometric constraints},
in \emph{Handbook of Geometric Constraint Systems Principles},
Chapman \& Hall/CRC, Boca Raton, 2018, pp.~273--285.

\bibitem{Zariski1926}
O.~Zariski,
\newblock Sull'impossibilità di risolvere parametricamente per radicali
un'equazione algebrica $f(x,y)=0$ di genere $p>6$ a moduli generali,
 \emph{Atti Accad. Naz. Lincei Rend., Cl. Sci. Fis. Mat. Natur. Ser. VI}
\textbf{3} (1926), 660--666.
\end{thebibliography}
\end{document}